\documentclass[letterpaper,11pt]{amsart}
\usepackage[margin=1in]{geometry}

\usepackage{latexsym,amsfonts,amssymb, amstext, mathrsfs, framed}
\usepackage{youngtab}
\usepackage{graphicx,tikz, subcaption, tikz-cd}
\usepackage{mathrsfs,hyperref, amssymb}
\usepackage[utf8]{inputenc}
\usepackage[all]{xy}
\usetikzlibrary{decorations.markings}
\usepackage{amsmath}
\usepackage{amsthm}
\usepackage{psfrag}
\usepackage{epsfig}
\usepackage{enumitem}
\usepackage{tikz}
\usepackage{caption, wrapfig}
\usepackage{csquotes}
\usepackage{tcolorbox}
\tcbset{boxrule=0.4pt, colback=white, colframe=black}

\usetikzlibrary{decorations.pathmorphing}
\tikzstyle{vertex}=[circle, draw, inner sep=0pt, minimum size=6pt]

\usetikzlibrary{positioning}
\usepackage{xcolor} 

\newcommand{\uq}{U_q(\mathfrak{sl}_n)} 
\newcommand{\dyk}{Dyck_{\vec{b}_1 \Delta \vec{b}_2}} 
\newcommand{\flex}{flex_{\vec{b}_1, \vec{b}_2}} 

\newtheorem{theorem}{Theorem}
\newtheorem{definition}[theorem]{Definition}
\newtheorem{lemma}[theorem]{Lemma}
\newtheorem{remark}[theorem]{Remark}
\newtheorem{example}[theorem]{Example}
\newtheorem{proposition}[theorem]{Proposition}
\newtheorem{corollary}[theorem]{Corollary}

\title{Stranding $\mathfrak{sl}_n$ webs}
\date{}

\author{Heather M. Russell}
\address{University of Richmond, Jepson Hall, Richmond VA 23173}
\email{hrussell@richmond.edu}

\author{Julianna Tymoczko}
\address{Smith College, Northampton MA 01063}
\email{jtymoczko@smith.edu}

\begin{document}
\begin{abstract} 
Webs are planar directed graphs that encode invariant vectors for tensor products of fundamental $U_q(\mathfrak{sl}_n)$-representations. For $\mathfrak{sl}_2$ and $\mathfrak{sl}_3$, web calculus is governed by effective reduction rules and well-understood reduced-web bases, but these features break down in higher rank. In this paper we develop a global combinatorial framework for untagged $\mathfrak{sl}_n$ webs, called \emph{strandings}, which organizes local labeling data into systems of colored directed paths on the web graph.

Our main result is an explicit state-sum formula for the invariant vector of an untagged web as a weighted sum over its valid strandings.  The usual labeling-based construction requires a choice of decomposition into elementary web pieces and a vertex-by-vertex coefficient calculation.  Our global reformulation avoids both.
We prove that the vectors produced by strandings are $U_q(\mathfrak{sl}_n)$-invariant, compare the untagged theory with the tagged-web framework of Cautis, Kamnitzer, and Morrison, and use this comparison to obtain a complete set of relations for untagged web graphs. We also give applications of strandings to nonvanishing results, basis constructions from tableaux, and connections with Springer-theoretic combinatorics. \end{abstract}

\maketitle

 \section{Introduction}\label{section: introduction}
 
Webs are plane graphs that arise naturally in the representation theory of quantum groups, with important connections to knot theory, algebraic geometry, and categorification. Given a semisimple Lie algebra $\mathfrak{g}$, a web graph for $\mathfrak{g}$ encodes a $U_q(\mathfrak{g})$-equivariant map between tensor products of fundamental $U_q(\mathfrak{g})$-representations (or equivalently, a $U_q(\mathfrak{g})$-invariant vector). Webs are not only diagrammatic models for representation theory but also concrete computational tools: one can identify web bases for invariant spaces, compare these to other well-known bases, study diagrammatic actions on web vectors, and relate combinatorial operations on web graphs to algebraic or geometric structures. This is best developed in the low-rank setting, especially for $\mathfrak{sl}_2$ and $\mathfrak{sl}_3$, but becomes substantially more difficult beyond it.
 
The main challenge in analyzing $\mathfrak{sl}_n$ webs for $n\geq 4$ is that they do not share two central features of $\mathfrak{sl}_2$ and $\mathfrak{sl}_3$ webs that are fundamental to web calculations in the literature.
\begin{itemize}
    \item First, in $\mathfrak{sl}_2$ and $\mathfrak{sl}_3$, the relations can be arranged so that they become reduction rules that always simplify the web graph by virtually any natural combinatorial measure. This yields a straightforward deterministic algorithm to decompose an arbitrary web graph into its fundamental building blocks, 
    together with transparent characterizations of when a web is reduced.
    \item Second, in $\mathfrak{sl}_2$ and $\mathfrak{sl}_3$, the basis of invariant vectors coming from reduced webs happens to satisfy 
    an upper-triangularity property with respect to a commonly-used 
    ordering on terms. Moreover, the leading term of a reduced web's vector can be identified directly from the structure of the web graph. These features provide a powerful combinatorial handle on linear independence, decomposition, rotation, and change-of-basis questions.
\end{itemize}
Beginning already at $n=4$, neither feature persists in any obvious general form: the known relations no longer behave like a straightforward reduction procedure, and there is no similarly-transparent reduced-web basis. A natural question is therefore: 
\begin{tcolorbox}[boxrule=0.4pt, colback=white, colframe=black]
\begin{center}
    \emph{What combinatorial structure remains once the low-rank simplicity breaks down?}
\end{center}
\end{tcolorbox}

This paper introduces a new global combinatorial framework on untagged $\mathfrak{sl}_n$ webs, called \emph{strandings}, designed to address this question. A \emph{strand} is, roughly speaking, a colored directed path on the web graph that either connects boundary vertices or forms a closed curve. A \emph{stranding} is a compatible collection of such strands; a single web graph supports many valid strandings. Strands reinterpret local labeling data as a global structure on the graph, and the key observation is that the resulting strandings collectively encode the web vector: their orientations (clockwise or counterclockwise) determine the coefficient of each term. Our main result is an explicit formula expressing the web vector of a web graph $G$ as a weighted sum over its valid strandings,
\[
    f(G) \;=\; \sum_{S \,\in\, \mathcal{S}tr(G)} c(S).
\]
This formula is completely global and intrinsically an invariant of the combinatorial graph, meaning it does not require one to fix a decomposition of the web into cups, caps, and trivalent vertices and track local coefficient contributions via that decomposition.
 
\begin{figure}[ht]
\centering
\begin{tikzpicture}[scale=.9, yscale=-1]
\draw[style=dashed, <->] (.5,.25)--(8.5,.25);
\begin{scope}[thick,decoration={
    markings,
    mark=at position 0.5 with {\arrow{>}}}
    ]
   \draw[blue,postaction={decorate}] (1,.25)--(2,1);
\draw[blue,postaction={decorate}] (1.97,1)--(1.97,.25);
\draw[red,postaction={decorate}] (2.03,.25)--(2.03,1);
 \draw[teal, postaction={decorate}] (2.97,.25)--(3.22,1);
  \draw[red, postaction={decorate}] (3.28,1)--(3.03,.25);
      \draw[red, postaction={decorate}] (3.97,.25)--(3.72,1);
           \draw[blue, postaction={decorate}] (3.78,1)--(4.03,.25);
   \draw[blue, postaction={decorate}] (3.25,2)--(3.25,1);
    \draw[red, postaction={decorate}] (3.25,2.04)--(4.5,2.54);
        \draw[blue, postaction={decorate}] (4.5,2.47)--(3.25,1.97);
        \draw[teal, postaction={decorate}] (4.5,2.5)to[out=0, in=120](8,.25);
     \draw[teal, postaction={decorate}] (3.75,1)--(4.5, 1.5);
     \draw[red,postaction={decorate}] (4.5, 1.45)--(5,.95);
          \draw[blue,postaction={decorate}] (5, 1.05)--(4.5,1.55);
       \draw[teal, postaction={decorate}] (5.03, .25)--(5.03,1);
         \draw[red, postaction={decorate}] (4.97, 1)--(4.97,.25);
         \draw[teal, postaction={decorate}] (6.5, 1)--(6,.25);
           \draw[blue, postaction={decorate}] (7, .25)--(6.5,1);
           \end{scope}
 
\begin{scope}[decoration={
    markings,
    mark=at position 0.5 with {\arrow{>}}}
    ]
           \draw[red, thick, postaction={decorate}] (2,1)--(3.25,2);
   \draw[blue, thick, postaction={decorate}] (3.25, 1.07)--(3.75,1.07);
      \draw[teal, thick, postaction={decorate}] (3.25, 1)--(3.75,1);
         \draw[red, thick, postaction={decorate}] (3.75,.93)--(3.25, .93);
       \draw[blue, thick, postaction={decorate}](6.5, 1.03)--(5,1.03);
    \draw[teal, thick, postaction={decorate}](5, .97)--(6.5,.97);
    \draw[blue, thick, postaction={decorate}] (4.45,1.5)--(4.45,2.5);
      \draw[red, thick, postaction={decorate}] (4.5,2.5)--(4.5,1.5);
        \draw[teal, thick, postaction={decorate}] (4.55,1.5)--(4.55,2.5);
         \end{scope}

\draw[radius=.08, fill=black](1,.25)circle;
\draw[radius=.08, fill=black](2,.25)circle;
\draw[radius=.08, fill=black](3,.25)circle;
\draw[radius=.08, fill=black](4,.25)circle;
\draw[radius=.08, fill=black](5,.25)circle;
\draw[radius=.08, fill=black](6,.25)circle;
\draw[radius=.08, fill=black](7,.25)circle;
\draw[radius=.08, fill=black](8,.25)circle;
\draw[radius=.08, fill=black](2,1)circle;
\draw[radius=.08, fill=black](3.25,1)circle;
\draw[radius=.08, fill=black](3.75,1)circle;
\draw[radius=.08, fill=black](5,1)circle;
\draw[radius=.08, fill=black](6.5,1)circle;
\draw[radius=.08, fill=black](4.5, 1.5)circle;
\draw[radius=.08, fill=black](3.25, 2)circle;
\draw[radius=.08, fill=black](4.5, 2.5)circle;
\end{tikzpicture}
\caption{A web together with a valid stranding.
Strandings reorganize the usual local labeling data into colored directed
paths on the graph.}
\label{fig:intro-stranding}
\end{figure}
 
\begin{figure}[ht]
\centering
\begin{tikzpicture}[>=stealth, node distance=2.5cm, font=\small]
  \tikzstyle{box} = [draw, rounded corners=4pt, minimum height=1.3cm,
                     text width=3.2cm, align=center, inner sep=6pt]
  \node[box] (G)  {Web graph $G$\\[3pt]\footnotesize(directed, edge-weighted plane graph)};
  \node[box, right=of G] (S)
        {Valid strandings $\sigma \in \mathcal{S}(G)$\\[3pt]\footnotesize(global colored directed paths)};
  \node[box, right=of S] (f)
        {$\displaystyle f(G)$\\[3pt]\footnotesize($U_q(\mathfrak{sl}_n)$-invariant)};
  \draw[->, thick] (G) -- (S)
        node[midway, above, font=\footnotesize] {stranding};
  \draw[->, thick] (S) -- (f)
        node[midway, above, font=\footnotesize] {sum};
\end{tikzpicture}
\caption{Overview of the stranding construction. Each untagged web graph $G$ is associated to a collection of valid strandings---systems of colored directed paths drawn globally on the graph---whose weighted sum yields an explicit $U_q(\mathfrak{sl}_n)$-invariant vector $f(G)$.}
\label{fig:stranding-overview}
\end{figure}
 
It is fruitful to compare our construction
with existing state-sum formulas for web vectors. In the standard description, a \emph{state} is a labeling of web edges by subsets of $[n]$ (or equivalently, binary vectors), subject to certain compatibility conditions at vertices. Each state contributes a term to the web vector, and the corresponding coefficient is obtained by multiplying \emph{local} contributions associated to  each vertex as well as other elementary pieces of the web. Strandings provide another way of organizing the same underlying algebraic data of terms in an invariant vector. The difference is that strandings convert this data into a global system of colored directed paths on the graph, making visible certain large-scale combinatorial features that are difficult to see in the usual local formulation.

We view stranding as part of a thirty-year project to identify the correct
global combinatorial structures hidden inside local labeling models for webs.
Murakami--Ohtsuki--Yamada evaluate a closed colored plane graph via a state
sum over admissible labelings, converting those labelings into oriented curves
running along the graph. The weight of a state combines the rotation number of these
curves and local contributions at each vertex \cite{MOY}. Similarly, Sikora's state sum
for closed $n$-valent webs has both global and local contributions \cite{Sikora}. Robert
reformulated the MOY evaluation so that the local contributions disappear. He
assigns to each coloring of a closed web a degree entirely determined by rotation numbers of
bicolored cycles \cite{Robert}, and he proves that the resulting evaluation
agrees with that of MOY using relations of Wu \cite{Wu}. For $n=3$ he
describes this as a $q$-deformation of a result of Jaeger
\cite{jaeger1992new, robert2013edgecolorings}. The evaluation was later extended to foams
\cite{RobertWagner}. Robert's bicolored cycles are our closed $(i,j)$-strands,
and for closed webs, our strand exponent agrees with his degree up to normalization.

For $\mathfrak{sl}_3$ webs with boundary, Khovanov and Kuperberg give a state
sum that translates labeling data into oriented curves, with local coefficient
contributions \cite{KK}. Bigelow's approach to the $SL_n$ spider attaches to a
labeled web a \emph{cobweb} whose strands are labeled by roots; these are
our $(i,j)$-strands, and they are defined for webs with boundary. He evaluates
closed cobwebs by a global rule but does not assemble open cobwebs into a
vector \cite{Bigelow}. Robert likewise defines a degree for webs with colored
boundary and uses it to verify skein relations
\cite[Definitions 2.8--2.10]{Robert}. Fraser, Lam, and Le work with the same
untagged model at $q=1$ and express each coefficient as a signed count of
consistent labelings \cite[Lemma 5.4]{FraserLamLe}. Gaetz, Pechenik,
Pfannerer, Striker, and Swanson give a quantum state sum over labelings
\cite[Theorem 2.9]{GaetzetalRotation} and remark that an explicit monomial
expansion for general webs is hard to find in the literature. The contribution
of the present paper is the case with boundary: we give an explicit global
coefficient for each stranding of a web with boundary, and we prove that these
coefficients assemble into a $U_q(\mathfrak{sl}_n)$-invariant vector.
 
We work with \emph{untagged} $\mathfrak{sl}_n$ web graphs, essentially those used by Kuperberg in his original paper on webs \cite{KuperbergWebs}, by Fontaine and Westbury in their work on higher-rank webs and bases \cite{FontaineThesis, FON, WestburyBases}, for closed graphs, by Robert \cite{Robert}, and, at $q=1$, by Fraser, Lam, and Le \cite{FraserLamLe}. These are directed, edge-weighted plane graphs satisfying a flow condition mod $n$ at each trivalent interior vertex. There are several closely related models of $\mathfrak{sl}_n$ webs in the literature, most notably the tagged webs of Cautis, Kamnitzer, and Morrison (CKM) \cite{CKM}. The tagged setting is important for certain categorical and sign-sensitive purposes, but tags also complicate explicit combinatorial calculations and obscure some features that are familiar from the classical $\mathfrak{sl}_3$ picture. One ancillary aim of this paper is therefore to give a precise and complete framework for the untagged setting together with tools that are well adapted to direct computation.
 
\subsection*{Main results}
 
Our main construction associates to each untagged web graph $G$ a vector $f(G)$ given by a weighted sum over the valid strandings of $G$. More precisely, Theorem~\ref{thm:invariant vector flow formula} gives one term of $f(G)$ for each valid stranding of $G$, with coefficients computed from associated strand data. This is a global reformulation of the usual labeling-based state sum. We prove that $f(G)$ is $\uq$-invariant, and in Section~\ref{section:tagless vs tagged} we use a comparison with the tagged-web framework of CKM \cite{CKM} to show that the resulting map is surjective and to identify an explicit set of generating relations for its kernel. In particular, we obtain a complete set of relations for untagged web graphs parallel to the tagged CKM relations.
 
The body of the paper uses strandings to address problems that are difficult in higher rank precisely because reduction is no longer available:
 
\begin{enumerate}
    \item \textbf{Nonvanishing} (Theorem~\ref{thm: nonzero coefficient}): a nonvanishing result for $f(G)$, proving that the stranding formula is always nontrivial with an argument that follows almost trivially from the global structure.
 
    \item \textbf{Distinguished strandings and the nontriviality criterion} (Theorem~\ref{thm: every web has a stranding}): constructing a canonical base stranding for any web graph, from which we recover the well-known numerical criterion for nontriviality of the invariant space.
 
    \item \textbf{Web bases from tableaux} (Section~\ref{section: constructing basis webs from tableaux}, Theorem~\ref{theorem: using Fontaine to prove basis}):  algorithmically building a family of web graphs built  from rectangular standard Young tableaux whose web vectors form a basis of $\textup{Inv}(\vec{1})$, that avoids the recursive construction in (and slightly generalizes) work of Fontaine \cite{FON} and Westbury \cite{WestburyBases}. Elias \cite{Eli15} constructs closely related bases, the light ladders, by an
inductive algorithm in the tagged setting.

    \item \textbf{Complete relations} (Theorem~\ref{thm:Fontaine relations}): producing a complete set of relations for untagged $\mathfrak{sl}_n$ web graphs.
\end{enumerate}
In addition, we sketch applications proven in forthcoming papers:
\begin{enumerate} \setcounter{enumi}{4}
    \item \textbf{Tableau combinatorics for $\mathfrak{sl}_3$} (Section~\ref{section: stranding tableau combo}): interaction of strandings with classical tableau operations (depth and evacuation), extending familiar low-rank structure into the stranding framework.
 
    \item \textbf{Springer fibers} (Section~\ref{subsec:springer}): a connection between stranded webs and Springer-theoretic combinatorics.
 
\end{enumerate}
 
These applications show that strandings are not merely an alternative encoding of the usual state sum, but a useful framework for organizing higher-rank web combinatorics.
 
\subsection*{Broader context}
 
This work fits into a broader literature seeking effective combinatorial models for webs. For $\mathfrak{sl}_2$, webs appear as Temperley--Lieb diagrams, which are also noncrossing matchings of longstanding combinatorial interest \cite{RumerTellerWeyl, TemperleyLieb}. Kuperberg introduced the term ``webs'' and developed an explicit framework for rank-two Lie algebras, including type $A$ webs \cite{KuperbergWebs}. In higher rank, Kim and Morrison constructed $\mathfrak{sl}_4$ and general $\mathfrak{sl}_n$ web theories and identified large families of necessary relations \cite{Kim, MorrisonThesis}; Fontaine and Westbury studied basis constructions in the untagged setting \cite{FON, WestburyBases}; and CKM established a complete set of relations in the tagged setting \cite[Theorem 3.3.1]{CKM}. Sikora \cite{Sikora} gives a different diagrammatic model, built from $n$-valent ribbon graphs in $\mathbb{R}^3$ whose edges all carry the standard representation and uses it to define an $SU_n$-skein module. Poudel \cite{Poudel} shows Sikora's category is equivalent, as a ribbon category, to the full subcategory of the CKM spider generated by the standard representation and its dual. The comparison in Section~\ref{section:tagless vs tagged} is of a different nature: rather than an equivalence of categories, it gives explicit maps between spaces of untagged and tagged web graphs with a fixed boundary weight vector, from which we extract a generating set of relations.
 
More recently, webs have been studied via \emph{plabic} (planar bicolored) graphs. Postnikov introduced plabic graphs to encode combinatorial structure on the Grassmannian, including parametrizations of points via dimer covers \cite{Postnikov, PostnikovSW2009} and the cluster-algebra structure on the coordinate ring of the Grassmannian \cite{FP}. Fraser, Lam, and Le \cite{FraserLamLe} relate plabic graphs to classical $\mathfrak{sl}_n$ webs by superimposing almost-perfect matchings: each superposition of dimer covers yields a linear combination of web graphs.   Closely related are the dimer models studied by Douglas, Kenyon, and Shi \cite{DKS}. They relate $n$-fold dimer covers to webs and local systems, and  Douglas, Kenyon, Ovenhouse, Panitch, and Tata \cite{DKOPT}, introduce a quantum version. Plabic web graphs can be viewed as a quotient of what are classically called web graphs by relations that suppress the order in which a sequence of merges or splits occur---a distinction the classical framework used in this paper retains. Web graphs in combinatorics tend to appear as plabic graphs, especially in questions about bases 
(see, e.g., \cite{GaetzetalRotation}), while representation-theoretic and knot-theoretic work tends to use classical webs.
 
The case of $\mathfrak{sl}_4$ already shows both the possibilities and the difficulties. Hagemeyer observed that certain edges in an $\mathfrak{sl}_4$ web graph can be contracted \cite{Hagemeyer}, and this was used to great effect by Gaetz, Pechenik, Pfannerer, Striker, and Swanson in the construction of an $\mathfrak{sl}_4$ web basis with desirable properties such as rotational invariance \cite{GaetzetalRotation}. This is an important success beyond low rank. At the same time, the amount of additional structure needed even in rank four underscores the difficulty of the problem in higher rank, and at present, they do not believe their methods extend to all $n$. The present paper takes a different approach: rather than seeking a replacement for low-rank reduction or passing to a different planar model, we define a global combinatorial structure directly on the web graph itself.
 
Strandings for webs with boundary are also closely connected to the combinatorics of tableaux and noncrossing matchings. In the $\mathfrak{sl}_2$ and $\mathfrak{sl}_3$ literature, one often identifies a web with a single matching or tableau. Strandings reveal a richer picture: each web graph naturally carries a \emph{family} of multicolored noncrossing matchings, one for each valid stranding; these interact naturally with the tableau combinatorics used to construct basis webs in low rank. In this sense, strandings extend familiar low-rank combinatorics while also revealing structure that is invisible from the perspective of a single leading term. 

\section{Acknowledgements}
The first author was supported by an AMS-Simons Research Enhancement Grant for PUI Faculty and a University of Richmond sabbatical fellowship.  The second author was supported by NSF-DMS grants 2349088 and 1800773, as well as an AWM-MERP fellowship. The authors also gratefully acknowledge the support of the Budapest Semesters in Mathematics Director's Mathematician in Residence program and the Smith College SURF program, as well as Charlie Frohman, Christian Gaetz, Iva Halacheva, Mee Seong Im, Mikhail Khovanov, Brendan Rhoades, Anne Schilling, and for many helpful discussions with Vikhyat Agarwal, Lucas Adams Cowan, Brett Barnes, Michael Bo, Madeline Burns, Jennie Campbell, Junyang Chen, Ioana-Andreea Cristescu, Erica Duke,  Ronja Eilfort, Zoey Fan, Felicia Flores, Eleanor Gallay, Zacharie Georges, Emily Hafken, Annabelle Hendrickson, Michael Kee, Veronica Lang, Grace Mabli, Blaise Marsho, Jacob Martin, Jade Mawn, Bella Mohren, Eric Neuhaus, Miriam Poe, Natali Sabri, Caitlin Sales, Kerry Seekamp, Orit Tashman, Beatrice Tauer, Weiyi Wan, Michael Wang, Veronica Wang, Ava Weninger, Caroline White, Miah Wilson, and Lance Wong.  

\section{Preliminaries}\label{section: preliminaries}

We begin by fixing notation and background for untagged web spaces and spaces of $\uq$-invariants, which will serve as the target of our construction. 

Fix $n\ge 2$. An $\mathfrak{sl}_n$ web is a directed, edge-weighted plane graph with boundary. Boundary vertices are univalent and lie on a single horizontal axis, and the graph is embedded in the half-plane below this axis. Edge weights lie in $\{1,\ldots,n-1\}$. Webs are considered up to planar isotopy relative to the boundary.

For a vertex $v$ and an edge $e$ incident to $v$, define
\begin{equation}\label{eq: define sigma_v}
\sigma_v(e)=
\begin{cases}
1 & \textup{if $e$ is directed into $v$},\\
-1 & \textup{if $e$ is directed out of $v$}.
\end{cases}
\end{equation}

\begin{remark}
Some authors draw webs above a horizontal boundary axis or within a disk with boundary vertices on a circle rather than a line. These are combinatorially equivalent conventions. We use a horizontal axis because it is well suited to invariant vectors.

More generally, one may consider webs between two horizontal boundary axes, as in \cite{CKM, MorrisonThesis}. That setting emphasizes the categorical structure in which webs are morphisms that are composed via stacking. The two contexts are related by the standard isomorphism
$$
Hom(U,V)\cong U^*\otimes V.
$$

The web vector formula in Section~\ref{section:main construction} extends symmetrically to webs between two horizontal boundary axes; see Remark~\ref{rem:symmetry} for the precise form. We work with the single-axis convention in the body of the paper because our applications use boundary-face combinatorics specific to that setting.
\end{remark}

\subsection{Untagged webs}

We work primarily with the untagged web model of Fontaine and Westbury \cite{FontaineThesis,FON,WestburyBases}. This generalizes Temperley--Lieb diagrams and Kuperberg's $\mathfrak{sl}_3$ webs \cite{KuperbergWebs, TemperleyLieb}.

\begin{definition}[Untagged webs]\label{definition: untagged web}
An \emph{untagged $\mathfrak{sl}_n$ web} is an $\mathfrak{sl}_n$ web graph whose interior vertices are trivalent and satisfy the following \emph{conservation of flow modulo $n$} condition: if $v$ is a trivalent vertex with incident edges $e_1,e_2,e_3$ of weights $\ell_1,\ell_2,\ell_3$, then
$$
\sum_{i=1}^3 \sigma_v(e_i)\ell_i \equiv 0 \pmod n.
$$

Let $G$ be an untagged web with boundary vertices $v_1,\ldots,v_m$ read from left to right. Let $e_i$ be the edge incident to $v_i$, and let $\ell_i$ be its weight. Define
$$
k_i=
\begin{cases}
\ell_i & \textup{if } \sigma_{v_i}(e_i)=1,\\
n-\ell_i & \textup{if } \sigma_{v_i}(e_i)=-1.
\end{cases}
$$
The vector $\vec{k}=(k_1,\ldots,k_m)$ is the \emph{boundary weight vector} of $G$.

Let $F(\vec{k})$ be the set of all untagged $\mathfrak{sl}_n$ webs with boundary weight vector $\vec{k}$, and let $\mathcal{F}(\vec{k})$ be the free $\mathbb{C}(q)$-vector space generated by $F(\vec{k})$.
\end{definition}

\begin{example}\label{ex:untagged web}
Let $n=4$ and $\vec{k}=(1,1,3,3)$. The web in Figure~\ref{fig:ex untagged web} is an element of $F(\vec{k})$. We use it as a running example throughout the paper.
\begin{figure}[h]
 \begin{center}
\raisebox{3pt}{\begin{tikzpicture}[scale=1]

\draw[style=dashed, <->] (0,0)--(7,0);

\begin{scope}[thick,decoration={
    markings,
    mark=at position 0.5 with {\arrow{<}}}
    ] 

\draw[postaction={decorate},style=thick] (1,0) to[out=270,in=180] (3,-2);
\draw[postaction={decorate},style=thick] (3,-2)--(3,-1);
\draw[postaction={decorate},style=thick] (3,-2)--(4,-2);
\draw[postaction={decorate},style=thick] (5,0) to[out=270,in=0] (4,-1);
\draw[postaction={decorate},style=thick] (4,-1)--(3,-1);
\draw[postaction={decorate},style=thick] (3,-1) to[out=180,in=270] (2,0);
\end{scope}

\begin{scope}[thick,decoration={
    markings,
    mark=at position 0.5 with {\arrow{>}}}
    ] 

\draw[postaction={decorate},style=thick] (6,0) to[out=270,in=0] (4,-2);
\draw[postaction={decorate},style=thick] (4,-2)--(4,-1);
\end{scope}

\draw[radius=.05, fill=black](1,0)circle;
\draw[radius=.05, fill=black](3,-1)circle;
\draw[radius=.05, fill=black](4,-1)circle;
\draw[radius=.05, fill=black](3,-2)circle;
\draw[radius=.05, fill=black](4,-2)circle;
\draw[radius=.05, fill=black](2,0)circle;
\draw[radius=.05, fill=black](5,0)circle;
\draw[radius=.05, fill=black](6,0)circle;

\node at (.75,-.25) {\tiny{$1$}};
\node at (1.75,-.25) {\tiny{$3$}};
\node at (3.5,-.75) {\tiny{$1$}};
\node at (3.5,-2.25) {\tiny{$3$}};
\node at (2.75, -1.5) {\tiny{$2$}};
\node at (4.25, -1.5) {\tiny{$2$}};
\node at (6.25, -.25) {\tiny{$1$}};
\node at (5.25, -.25) {\tiny{$3$}};

    \end{tikzpicture}}
    \end{center}

    \caption{An untagged $\mathfrak{sl}_4$ web graph}\label{fig:ex untagged web}
\end{figure}
\end{example}

\begin{remark}
Fontaine labels edges by fundamental representations rather than by integers \cite{FontaineThesis,FON}. In that language, the interior vertex compatibility condition says that the tensor product of the three edge labels around each trivalent vertex must have a nonzero $\uq$-invariant subspace. This is equivalent to Definition~\ref{definition: untagged web}.
\end{remark}

\begin{remark}
If an untagged web has no boundary vertices, its boundary weight vector is $\vec{k}=\emptyset$.

If two boundary vertices $v_{i_1}$ and $v_{i_2}$ are joined by an edge $e:v_{i_1}\stackrel{\ell}{\mapsto}v_{i_2}$, then $e=e_{i_1}=e_{i_2}$ and
     $k_{i_1}=n-\ell$ while $k_{i_2}=\ell$. 
\end{remark}

\subsection{Edge flips and vertex types}

A basic operation on untagged webs is to reverse an edge while replacing its weight by a complementary weight.

\begin{definition}\label{definition: flip edge in web graph}
Let $G\in F(\vec{k})$. For an edge $u\stackrel{\ell}{\mapsto}v$ of $G$, define its \emph{edge flip} by
$$
\varphi\!\left(u\stackrel{\ell}{\mapsto}v\right)=v\stackrel{n-\ell}{\longmapsto}u.
$$
Thus $\varphi$ reverses orientation and replaces weight $\ell$ by $n-\ell$.

If $\mathcal{E}\subseteq E(G)$, let $G_{\varphi(\mathcal{E})}$ be the graph obtained from $G$ by flipping every edge in $\mathcal{E}$ and leaving all other data unchanged.
\end{definition}

\begin{example}\label{ex:edge flip}
Figure~\ref{fig:ex edge flip} shows the web from Example~\ref{ex:untagged web}, together with a chosen set of edges in violet and the result of flipping those edges.

\begin{figure}[h]
 \begin{center}
\raisebox{3pt}{\begin{tikzpicture}[scale=.75]

\draw[style=dashed, <->] (0,0)--(7,0);

\begin{scope}[thick,decoration={
    markings,
    mark=at position 0.5 with {\arrow{<}}}
    ] 

\draw[postaction={decorate},style=thick] (1,0) to[out=270,in=180] (3,-2);
\draw[postaction={decorate},style=thick] (3,-2)--(3,-1);
\draw[postaction={decorate},style=thick] (5,0) to[out=270,in=0] (4,-1);

\draw[postaction={decorate},style=thick, violet] (3,-2)--(4,-2);
\draw[postaction={decorate},style=thick, violet] (4,-1)--(3,-1);
\draw[postaction={decorate},style=thick, violet] (3,-1) to[out=180,in=270] (2,0);
\end{scope}

\begin{scope}[thick,decoration={
    markings,
    mark=at position 0.5 with {\arrow{>}}}
    ] 

\draw[postaction={decorate},style=thick, violet] (6,0) to[out=270,in=0] (4,-2);
\draw[postaction={decorate},style=thick, violet] (4,-2)--(4,-1);
\end{scope}

\draw[radius=.05, fill=black](1,0)circle;
\draw[radius=.05, fill=black](3,-1)circle;
\draw[radius=.05, fill=black](4,-1)circle;
\draw[radius=.05, fill=black](3,-2)circle;
\draw[radius=.05, fill=black](4,-2)circle;
\draw[radius=.05, fill=black](2,0)circle;
\draw[radius=.05, fill=black](5,0)circle;
\draw[radius=.05, fill=black](6,0)circle;

\node at (.75,-.25) {\tiny{$1$}};
\node at (1.75,-.25) {\textcolor{violet}{\tiny{$3$}}};
\node at (3.5,-.75) {\textcolor{violet}{\tiny{$1$}}};
\node at (3.5,-2.25) {\textcolor{violet}{\tiny{$3$}}};
\node at (2.75, -1.5) {\tiny{$2$}};
\node at (4.25, -1.5) {\textcolor{violet}{\tiny{$2$}}};
\node at (6.25, -.25) {\textcolor{violet}{\tiny{$1$}}};
\node at (5.25, -.25) {\tiny{$3$}};

    \end{tikzpicture}}
    \hspace{.25in}
\raisebox{3pt}{\begin{tikzpicture}[scale=.75]

\draw[style=dashed, <->] (0,0)--(7,0);

\begin{scope}[thick,decoration={
    markings,
    mark=at position 0.5 with {\arrow{<}}}
    ] 

\draw[postaction={decorate},style=thick] (1,0) to[out=270,in=180] (3,-2);
\draw[postaction={decorate},style=thick] (3,-2)--(3,-1);
\draw[postaction={decorate},style=thick] (5,0) to[out=270,in=0] (4,-1);
\draw[postaction={decorate},style=thick, violet] (6,0) to[out=270,in=0] (4,-2);
\draw[postaction={decorate},style=thick, violet] (4,-2)--(4,-1);

\end{scope}

\begin{scope}[thick,decoration={
    markings,
    mark=at position 0.5 with {\arrow{>}}}
    ] 
\draw[postaction={decorate},style=thick, violet] (3,-2)--(4,-2);
\draw[postaction={decorate},style=thick, violet] (4,-1)--(3,-1);
\draw[postaction={decorate},style=thick, violet] (3,-1) to[out=180,in=270] (2,0);
\end{scope}

\draw[radius=.05, fill=black](1,0)circle;
\draw[radius=.05, fill=black](3,-1)circle;
\draw[radius=.05, fill=black](4,-1)circle;
\draw[radius=.05, fill=black](3,-2)circle;
\draw[radius=.05, fill=black](4,-2)circle;
\draw[radius=.05, fill=black](2,0)circle;
\draw[radius=.05, fill=black](5,0)circle;
\draw[radius=.05, fill=black](6,0)circle;

\node at (.75,-.25) {\tiny{$1$}};
\node at (1.75,-.25) {\textcolor{violet}{\tiny{$1$}}};
\node at (3.5,-.75) {\textcolor{violet}{\tiny{$3$}}};
\node at (3.5,-2.25) {\textcolor{violet}{\tiny{$1$}}};
\node at (2.75, -1.5) {\tiny{$2$}};
\node at (4.25, -1.5) {\textcolor{violet}{\tiny{$2$}}};
\node at (6.25, -.25) {\textcolor{violet}{\tiny{$3$}}};
\node at (5.25, -.25) {\tiny{$3$}};

    \end{tikzpicture}}
    \end{center}

    \caption{Flipping edges in an untagged web graph}\label{fig:ex edge flip}
\end{figure}
\end{example}

\begin{lemma}\label{lemma: flipping edges gives another web graph}
If $G\in F(\vec{k})$ and $\mathcal{E}\subseteq E(G)$, then $G_{\varphi(\mathcal{E})}\in F(\vec{k})$.
\end{lemma}

\begin{proof}
If $e=u\stackrel{\ell}{\mapsto}v$ is flipped, then its contribution to the flow at $u$ changes from $-\ell$ to $n-\ell$, and its contribution at $v$ changes from $\ell$ to $-(n-\ell)$. In each case the change is a multiple of $n$, so conservation of flow modulo $n$ at interior vertices is unchanged.

For a boundary edge $e_i$ of weight $\ell_i$, flipping reverses the sign of $\sigma_{v_i}(e_i)$ and replaces $\ell_i$ by $n-\ell_i$. These changes preserve the corresponding entry $k_i$ of the boundary weight vector. Hence $G_{\varphi(\mathcal{E})}\in F(\vec{k})$.
\end{proof}

This gives a classification of trivalent vertices into two types, a distinction that will be useful later when comparing untagged webs with the tagged CKM model.

\begin{definition}\label{definition: split and merge vertices}
Let $v$ be an interior vertex of an untagged web graph $G$ with incident edges $e_1,e_2,e_3$ of weights $\ell_1,\ell_2,\ell_3$.

If
$$
\sum_{i=1}^3 \sigma_v(e_i)\ell_i=0,
$$
then at least one edge is directed into $v$ and at least one is directed out of $v$. If exactly one edge is directed into $v$, we call $v$ a \emph{split vertex}. If exactly two edges are directed into $v$, we call $v$ a \emph{merge vertex}.

More generally, an interior vertex $v$ of $G$ is called a split vertex, respectively a merge vertex, if it has that form in $G_{\varphi(\mathcal{E})}$ for some $\mathcal{E}\subseteq E(G)$.
\end{definition}

For instance, the web in Example \ref{ex:untagged web} has two split vertices on the left and two merge vertices on the right. Note that every interior vertex is exactly one of these two types.

\begin{remark}
If all three edges at an interior vertex are directed out of the vertex, then the vertex is a split vertex exactly when the three edge weights sum to $n$, and a merge vertex exactly when they sum to $2n$.
\end{remark}

\subsection{Tripod decompositions}
Later arguments use convenient isotopic representatives of a web. In particular, we use a tripod decomposition similar to the one in \cite{KK}, splitting a web into a disjoint union of the three local pieces in Figure~\ref{fig:basic pieces}: cups, caps, and tripods.
 
\begin{figure}[h]
\begin{subfigure}[t]{0.2\textwidth}
\centering
\begin{tikzpicture}[scale=.5]
    \draw[red, dashed] (0,0)--(3,0);
    \draw[thick] (.5,0) to[out=270, in=180] (1.5,-1) to[out=0, in=270] (2.5,0);
\end{tikzpicture}
\subcaption{cup}
\end{subfigure}
\begin{subfigure}[t]{0.2\textwidth}
\centering
\begin{tikzpicture}[scale=.5]
    \draw[red, dashed] (0,0)--(3,0);
    \draw[thick] (.5,0) to[out=90, in=180] (1.5,1) to[out=0, in=90] (2.5,0);
\end{tikzpicture}
\subcaption{cap}
\end{subfigure}
\begin{subfigure}[t]{0.2\textwidth}
\centering
\begin{tikzpicture}[scale=.4]
    \draw[dashed, red] (0,0)--(5,0);
    \draw[radius=.12, fill=black](2.5,-1.5)circle;
    \draw[thick] (.5,0) to[out=270, in=180] (2.5,-1.5) to[out=0, in=270] (4.5,0);
    \draw[thick] (2.5,-1.5)--(2.5,0);
\end{tikzpicture}
\subcaption{tripod}
\end{subfigure}
\caption{Local pieces of a tripod decomposition}\label{fig:basic pieces}
\end{figure}

\begin{definition}\label{def:tripod decomp}
Let $G$ be a plane graph with univalent boundary vertices and trivalent interior vertices, embedded below its boundary axis. A graph $G'$ is a \emph{tripod decomposition} of $G$ if $G'$ is isotopic to $G$ relative to the boundary and there exists a horizontal axis, called the \emph{internal axis}, such that:
\begin{enumerate}
    \item below the internal axis, $G'$ is a disjoint union of cups and tripods with endpoints on the internal axis,
    \item above the internal axis, $G'$ is a disjoint union of caps with endpoints on the internal axis, and
    \item between the internal and boundary axes, $G'$ consists only of vertical segments.
\end{enumerate}
We draw the internal axis as a red dashed line. The intersection points of $G'$ with the internal axis are not vertices of $G'$.
\end{definition}

\begin{example}\label{ex:websplit}
Figure~\ref{fig: ex websplit} shows the underlying trivalent graph of Example~\ref{ex:untagged web} together with one choice of tripod decomposition.

\begin{figure}[ht]
 \begin{center}
\raisebox{3pt}{\begin{tikzpicture}[scale=.75]
  \draw[style=dashed, <->] (0,0)--(7,0);
   \draw[style=thick] (1,0) to[out=270,in=180] (3,-2)--(3,-1) to[out=180,in=270] (2,0);
      \draw[style=thick] (6,0) to[out=270,in=0] (4,-2)--(3,-2);
      \draw[style=thick] (5,0) to[out=270,in=0] (4,-1);
      \draw[style=thick] (4,-2)--(4,-1)--(3,-1);

\draw[radius=.05, fill=black](1,0)circle;
\draw[radius=.05, fill=black](3,-1)circle;
\draw[radius=.05, fill=black](4,-1)circle;
\draw[radius=.05, fill=black](3,-2)circle;
\draw[radius=.05, fill=black](4,-2)circle;
\draw[radius=.05, fill=black](2,0)circle;
\draw[radius=.05, fill=black](5,0)circle;
\draw[radius=.05, fill=black](6,0)circle;

    \end{tikzpicture}}
\hspace{.25in}
\raisebox{-15pt}{\scalebox{.6}{\begin{tikzpicture}[scale=.65]
  \draw[style=dashed, <->] (0,4)--(13,4);
  \draw[style=dashed, red] (0,0)--(13,0);
  \draw[style=thick] (2,0) to[out=90,in=180] (5.5, 2.5) to[out=0,in=90] (9,0);
 \draw[style=thick] (3,0) to[out=90,in=180] (3.5, 1) to[out=0,in=90] (4,0);
  \draw[style=thick] (5,0) to[out=90,in=180] (5.5, 1) to[out=0,in=90] (6,0);
    \draw[style=thick] (7,0) to[out=90,in=180] (7.5, 1) to[out=0,in=90] (8,0);
    \draw[style=thick] (1,4)--(1,0) to[out=270,in=180] (2,-1)--(2,0);
    \draw[style=thick] (2,-1) to[out=0,in=270] (3,0);
   \draw[style=thick] (4,0) to[out=270,in=180] (5,-3)--(5,0);
      \draw[style=thick] (5,-3) to[out=0,in=270] (12,0)--(12,4);
       \draw[style=thick] (6,0) to[out=270,in=180] (7,-2)--(7,0);
      \draw[style=thick] (7,-2) to[out=0,in=270] (11,0)--(11,4);
        \draw[style=thick] (8,0) to[out=270,in=180] (9,-1)--(9,0);
      \draw[style=thick] (9,-1) to[out=0,in=270] (10,0)--(10,4);

\draw[radius=.08, fill=black](1,4)circle;
\draw[radius=.08, fill=black](2,-1)circle;
\draw[radius=.08, fill=black](5,-3)circle;
\draw[radius=.08, fill=black](7,-2)circle;
\draw[radius=.08, fill=black](9,-1)circle;
\draw[radius=.08, fill=black](10,4)circle;
\draw[radius=.08, fill=black](11,4)circle;
\draw[radius=.08, fill=black](12,4)circle;

    \end{tikzpicture}}}
    \caption{A tripod decomposition for a trivalent graph with boundary}\label{fig: ex websplit}
    \end{center}
\end{figure}
\end{example}

\subsection{Representation theory of $\uq$}\label{sub:repn theory}

We now fix conventions for $U_q(\mathfrak{sl}_n)$ and its fundamental representations. For ease of later comparison with CKM, we adopt conventions compatible with theirs.

Given integers $k$ and $l$, define the quantum integer  $[k]_q = \frac{q^k-q^{-k}}{q-q^{-1}}.$ This definition gives $[0]_q=0$, $[-k]_q=-[k]_q$, and $$[k]_q=q^{k-1} + q^{k-3} + \cdots + q^{-k+3} + q^{-k+1}$$ for $k>0$. We define the quantum binomial coefficient by $\genfrac[]{0pt}{2}{k}{0}_q=[1]_q=1$ and, for a positive integer $l$,  $$\genfrac[]{0pt}{0}{k}{l}_q = \frac{[k]_q\cdots [k-l+1]_q}{[l]_q\cdots[1]_q}.$$  

The quantum group $\uq$ is the Hopf algebra over $\mathbb{C}(q)$ generated by $E_i,F_i,K_i$ for $1\le i<n$, with relations
$$
K_iK_j=K_jK_i,\qquad K_jE_iK_j^{-1}=q^{\langle i,j\rangle}E_i,\qquad K_jF_iK_j^{-1}=q^{-\langle i,j\rangle}F_i,
$$
$$
[E_i,F_j]=\delta_{ij}\frac{K_i-K_i^{-1}}{q-q^{-1}},
$$
$$
[2]_qE_iE_jE_i=E_i^2E_j+E_jE_i^2 \textup{ if } |i-j|=1,\qquad [E_i,E_j]=0 \textup{ if } |i-j|>1,
$$
where
$$
\langle i,j\rangle=
\begin{cases}
2 & \textup{if } i=j,\\
-1 & \textup{if } |i-j|=1,\\
0 & \textup{otherwise.}
\end{cases}
$$

The coproduct is
$$
\Delta(E_i)=E_i\otimes K_i+1\otimes E_i,\qquad
\Delta(F_i)=F_i\otimes 1+K_i^{-1}\otimes F_i,\qquad
\Delta(K_i)=K_i\otimes K_i.
$$
The antipode is
$$
S(E_i)=-E_iK_i^{-1},\qquad S(F_i)=-K_iF_i,\qquad S(K_i)=K_i^{-1},
$$
and the counit is
$$
\epsilon(E_i)=\epsilon(F_i)=0,\qquad \epsilon(K_i)=1.
$$

Let $\mathbb{C}^n_q$ be the $n$-dimensional vector space over $\mathbb{C}(q)$ with standard basis $x_1,\ldots,x_n$. Write $\wedge_q$ for the quantum wedge product, characterized by
$$
x_i\wedge_q x_j + qx_j\wedge_q x_i=0 \textup{ for } i<j,
\qquad
x_i\wedge_q x_i=0.
$$
For $1\le k<n$, the $k$th fundamental representation is
$$
V_k:={\bigwedge}_q^k \mathbb{C}^n_q.
$$
Let $V_k^*=Hom(V_k,\mathbb{C}(q))$ be its dual.

If $\vec{k}=(k_1,\ldots,k_m)$ with each $k_j\in\{\pm1,\ldots,\pm(n-1)\}$, define $V(\vec{k})$
to be the tensor product whose $j$th factor is $V_{k_j}$ if $k_j>0$ and $V_{|k_j|}^*$ if $k_j<0$.

For a subset $T=\{t_1>\cdots>t_k\}\subset\{1,\ldots,n\}$, define
$$
x_T:=x_{t_1}\wedge_q \cdots \wedge_q x_{t_k}\in V_k.
$$
Then
$$
\{x_T: T\subset\{1,\ldots,n\},\ |T|=k\}
$$
is a basis of $V_k$.

\begin{remark}
If $\vec{b}=b_1\cdots b_n$ is a binary vector, let
$$
T(\vec{b})=\{i: b_i=1\}.
$$
We write $|\vec{b}|=|T(\vec{b})|$ and $x_{\vec{b}}:=x_{T(\vec{b})}$. Thus the basis of $V_k$ can be written as
$$
\{x_{\vec{b}}: |\vec{b}|=k\}.
$$
For example, we have $|1010|=2$, and $x_{1010}=x_3\wedge_q x_1\in V_2$.
\end{remark}

The action of $\uq$ on $V_1=\mathbb{C}^n_q$ is given by
$$
E_i(x_j)=
\begin{cases}
x_{j-1} & \textup{if } j=i+1,\\
0 & \textup{otherwise,}
\end{cases}
\qquad
F_i(x_j)=
\begin{cases}
x_{j+1} & \textup{if } j=i,\\
0 & \textup{otherwise,}
\end{cases}
$$
and
$$
K_i(x_j)=
\begin{cases}
qx_j & \textup{if } j=i,\\
q^{-1}x_j & \textup{if } j=i+1,\\
x_j & \textup{otherwise.}
\end{cases}
$$

Using the coproduct, this induces a $\uq$-action on tensor products that descends to quantum exterior powers. For $T=\{t_1,\dots, t_k\}\subset\{1,\ldots,n\}$ and $1\leq i<n$, we write $s_iT=\{s_it_1,\ldots, s_it_k\}$, where $s_i$ is the simple transposition swapping $i$ and $i+1$.  With this notation, we describe the $\uq$-action on $V_k$ as follows:
$$
E_i(x_T)=
\begin{cases}
x_{s_iT} & \textup{if } i\notin T \textup{ and } i+1\in T,\\
0 & \textup{otherwise,}
\end{cases}
$$
$$
F_i(x_T)=
\begin{cases}
x_{s_iT} & \textup{if } i\in T \textup{ and } i+1\notin T,\\
0 & \textup{otherwise,}
\end{cases}
$$
$$
K_i(x_T)=
\begin{cases}
qx_T & \textup{if } i\in T \textup{ and } i+1\notin T,\\
q^{-1}x_T & \textup{if } i\notin T \textup{ and } i+1\in T,\\
x_T & \textup{otherwise.}
\end{cases}
$$

Using the antipode, one obtains the action on $V_k^*$ given by
$$
E_i(x_T^*)=
\begin{cases}
-qx_{s_iT}^* & \textup{if } i\in T \textup{ and } i+1\notin T,\\
0 & \textup{otherwise,}
\end{cases}
$$
$$
F_i(x_T^*)=
\begin{cases}
-q^{-1}x_{s_iT}^* & \textup{if } i\notin T \textup{ and } i+1\in T,\\
0 & \textup{otherwise,}
\end{cases}
$$
$$
K_i(x_T^*)=
\begin{cases}
q^{-1}x_T^* & \textup{if } i\in T \textup{ and } i+1\notin T,\\
qx_T^* & \textup{if } i\notin T \textup{ and } i+1\in T,\\
x_T^* & \textup{otherwise.}
\end{cases}
$$

\begin{definition}
Let $W$ be a $\uq$-representation. A vector $w\in W$ is a \emph{$\uq$-invariant vector} if
$$
E_i(w)=F_i(w)=0
\qquad\textup{and}\qquad
K_i(w)=w
$$
for all $1\le i<n$.
We write $Inv(W)$ for the subspace of all $\uq$-invariant vectors in $W$, and in the case that $W=V(\vec{k})$ we abbreviate this as $Inv(\vec{k})$.
\end{definition}

Since $\uq$ acts trivially on $\mathbb{C}(q)$, evaluation at $1$ gives an isomorphism
$$
Hom_{\uq}(\mathbb{C}(q),W)\cong Inv(W).
$$
We use this identification freely.

\begin{theorem}\label{thm: hom space dimension}
Let $\vec{k}=(k_1,\ldots,k_m)$ be a vector of nonnegative integers. Then $Inv(\vec{k})$ has $\mathbb{C}(q)$-dimension equal to the number of row-strict Young tableaux of shape
$$
n\times \left(\frac{\sum_j k_j}{n}\right)
$$
and content $\{1^{k_1},2^{k_2},\ldots,m^{k_m}\}$ when $\frac{\sum_j k_j}{n}\in \mathbb{Z}$, and dimension $0$ otherwise \cite{MR1321638}.

In particular, for $\vec{k}=(1,\ldots,1)$, the dimension of $Inv(\vec{k})$ is the number of standard Young tableaux of shape $n\times \frac{m}{n}$ if $n$ divides $m$, and is $0$ otherwise.
\end{theorem}

\subsection{Tagged webs}

CKM use a tagged web model in their work \cite{CKM}; this model includes bivalent tagged vertices and imposes an integer, rather than modulo $n$, flow condition. Sikora \cite{Sikora} uses an $n$-valent graph model where all edges carry the standard representation. CKM define a surjective map from spaces of tagged webs to their corresponding invariant spaces and identify a generating set of relations for the kernel of this map. In Section~\ref{section:tagless vs tagged}, we compare our untagged model to the CKM tagged model via explicit maps and use that comparison to transfer relations.

\section{Strandings and binary labelings}\label{section: strandings and binary labelings results}

We now introduce the combinatorial structure that underlies our construction of web vectors. A \emph{stranding} of a web is a system of directed colored paths running along its edges. After developing this global picture, we show that the same data is encoded locally by binary labelings. The equivalence between these two perspectives will be fundamental in the next section.

\begin{remark}
We use the term \emph{simple strand} for the basic colored directed paths appearing in a stranding. Later in the section we will also consider associated $(i,j)$-strands. This terminology distinguishes the primary paths in a stranding from the additional strand data built from them.
\end{remark}

\subsection{Valid strandings}

\begin{definition}\label{definition: combinatorial stranding}
Let $G$ be an $\mathfrak{sl}_n$ web graph. A \emph{stranding} $S$ of $G$ is a collection of colored directed paths in $G$ such that:
\begin{itemize}
    \item each path is assigned a color $c\in\{1,\ldots,n-1\}$;
    \item each path is either closed or has both endpoints on the boundary of $G$;
    \item paths of the same color are pairwise disjoint.
\end{itemize}
The individual paths in $S$ are called \emph{simple strands}.
\end{definition}

Let $e$ be an edge of $G$. Write
$$
S(e)=S_+(e)\sqcup S_-(e),
$$
where $S_+(e)$, respectively $S_-(e)$, is the set of colors of simple strands directed with, respectively against, the orientation of $e$. Define
$$
\alpha_c^\vee(e)=
\begin{cases}
1 & \textup{if } c\in S_+(e),\\
-1 & \textup{if } c\in S_-(e),\\
0 & \textup{if } c\notin S(e).
\end{cases}
$$
While the quantity $\alpha_c^\vee(e)$ depends on a choice of stranding,
we suppress this in the notation; see Remark~\ref{rem:two alphas agree}
for more on the notation $\alpha_c^\vee$.

\begin{definition}
A stranding $S$ is \emph{valid} if for every edge $e$ of weight $\ell$ the following two conditions hold:
\begin{enumerate}
    \item the nonzero values $\alpha_c^\vee(e)$ alternate in sign as $c$ increases;
    \item if $c_{\max}$ is the largest color in $S(e)$, then
    $$
    \sum_{c=1}^{n-1}\alpha_c^\vee(e)c=
    \begin{cases}
        \ell & \textup{if } c_{\max}\in S_+(e),\\
        \ell-n & \textup{if } c_{\max}\in S_-(e).
    \end{cases}
    $$
\end{enumerate}
We write $\mathcal{S}tr(G)$ for the set of valid strandings of $G$.
\end{definition}

\begin{remark}
If $S$ satisfies condition \textup{(1)} above and $S(e)=\{c_1<\cdots<c_r\}$, then condition \textup{(2)} can be rewritten as
$$
c_r-c_{r-1}+\cdots+(-1)^{r-1}c_1=
\begin{cases}
\ell & \textup{if } c_r\in S_+(e),\\
n-\ell & \textup{if } c_r\in S_-(e).
\end{cases}
$$
\end{remark}

\begin{example}\label{ex:stranding example}
Figure~\ref{fig: ex stranding example} shows the untagged $\mathfrak{sl}_4$ web $G$ from Example~\ref{ex:untagged web} together with a valid stranding. We depict simple strands of colors $1,2,3$ in blue, red, and green respectively.

\begin{figure}[h]
 \begin{center}
\raisebox{3pt}{\begin{tikzpicture}[scale=.75]

\draw[style=dashed, <->] (0,0)--(7,0);

\begin{scope}[thick,decoration={
    markings,
    mark=at position 0.5 with {\arrow{<}}}
    ] 

\draw[postaction={decorate},style=thick] (1,0) to[out=270,in=180] (3,-2);
\draw[postaction={decorate},style=thick] (3,-2)--(3,-1);
\draw[postaction={decorate},style=thick] (3,-2)--(4,-2);
\draw[postaction={decorate},style=thick] (5,0) to[out=270,in=0] (4,-1);
\draw[postaction={decorate},style=thick] (4,-1)--(3,-1);
\draw[postaction={decorate},style=thick] (3,-1) to[out=180,in=270] (2,0);
\end{scope}

\begin{scope}[thick,decoration={
    markings,
    mark=at position 0.5 with {\arrow{>}}}
    ] 
\draw[postaction={decorate},style=thick] (6,0) to[out=270,in=0] (4,-2);
\draw[postaction={decorate},style=thick] (4,-2)--(4,-1);
\end{scope}

\draw[radius=.08, fill=black](1,0)circle;
\draw[radius=.08, fill=black](3,-1)circle;
\draw[radius=.08, fill=black](4,-1)circle;
\draw[radius=.08, fill=black](3,-2)circle;
\draw[radius=.08, fill=black](4,-2)circle;
\draw[radius=.08, fill=black](2,0)circle;
\draw[radius=.08, fill=black](5,0)circle;
\draw[radius=.08, fill=black](6,0)circle;

\node at (.75,-.25) {\tiny{$1$}};
\node at (1.75,-.25) {\tiny{$3$}};
\node at (3.5,-.75) {\tiny{$1$}};
\node at (3.5,-2.25) {\tiny{$3$}};
\node at (2.75, -1.5) {\tiny{$2$}};
\node at (4.25, -1.5) {\tiny{$2$}};
\node at (6.25, -.25) {\tiny{$1$}};
\node at (5.25, -.25) {\tiny{$3$}};

    \end{tikzpicture}}
    \hspace{.25in}
 \raisebox{10pt}{\begin{tikzpicture}[scale=.75]

\draw[style=dashed, <->] (0,0)--(7,0);

\begin{scope}[thick,decoration={
    markings,
    mark=at position 0.6 with {\arrow{<}}}
    ] 

\draw[postaction={decorate},style=thick, blue] (3.05,-2)--(3.05,-1);
\draw[postaction={decorate},style=thick, green] (4,-1.05)--(3,-1.05);
\draw[postaction={decorate},style=thick, blue] (3,-1.05) to[out=180,in=270] (1.95,0);
\draw[postaction={decorate},style=thick, blue] (6.05,0) to[out=270,in=0] (4,-2.05);
\draw[postaction={decorate},style=thick, green] (5,0) to[out=270,in=0] (4,-1);

\end{scope}

\begin{scope}[thick,decoration={
    markings,
    mark=at position 0.4 with {\arrow{>}}}
    ] 
\draw[postaction={decorate},style=thick, green] (1,0) to[out=270,in=180] (3,-2);
\draw[postaction={decorate},style=thick, blue] (3,-2)--(4,-2);
\draw[postaction={decorate},style=thick, red] (4,-2)--(4,-1);
\draw[postaction={decorate},style=thick, red] (5.95,0) to[out=270,in=0] (4,-1.95);
\draw[postaction={decorate},style=thick, green] (2.95,-2)--(2.95,-1);
\draw[postaction={decorate},style=thick, red] (4,-.95)--(3,-.95);
\draw[postaction={decorate},style=thick, red] (3,-.95) to[out=180,in=270] (2.05,0);

\end{scope}

\draw[radius=.08, fill=black](1,0)circle;
\draw[radius=.08, fill=black](3,-1)circle;
\draw[radius=.08, fill=black](4,-1)circle;
\draw[radius=.08, fill=black](3,-2)circle;
\draw[radius=.08, fill=black](4,-2)circle;
\draw[radius=.08, fill=black](2,0)circle;
\draw[radius=.08, fill=black](5,0)circle;
\draw[radius=.08, fill=black](6,0)circle;

    \end{tikzpicture}}   
    \end{center}

    \caption{An untagged web graph and a valid stranding}\label{fig: ex stranding example}
\end{figure}
\end{example}

\begin{remark}
Gaetz et al.\ define \emph{trip strands} for $\mathfrak{sl}_4$ webs \cite[Definition~2.12]{GaetzetalRotation}, but these are different from our simple strands. For instance, in their setup every edge carries a trip strand of every type, whereas in ours an edge may carry only some simple strand colors.
\end{remark}

The next result translates the path-based definition of a stranding into a compatibility condition for edgewise strand data at interior vertices. This reformulation will play a central role in the arguments that follow.

\begin{theorem}\label{thm:net zero strand contribution at each vertex}
Let $G$ be a web graph. For each edge $e\in E(G)$, choose valid edgewise simple strand data $S(e)=S_+(e)\sqcup S_-(e)$. These choices yield a valid stranding of $G$ if and only if for each interior vertex $v$ with incident edges $e_1,e_2,e_3$ and each color $1\le c\le n-1$,
$$
\sum_{i=1}^3 \sigma_v(e_i)\alpha_c^\vee(e_i)=0.
$$
\end{theorem}

\begin{proof}
The chosen edgewise data define colored directed strand fragments along the edges of $G$. These glue to a global stranding precisely when, for every interior vertex $v$ and every color $c$, either no strand of color $c$ is incident to $v$, or exactly one such strand enters $v$ and exactly one leaves $v$.

For an incident edge $e_i$, the quantity $\sigma_v(e_i)\alpha_c^\vee(e_i)$ is $0$ if no simple strand of color $c$ runs along $e_i$, is $1$ if such a strand is directed into $v$, and is $-1$ if it is directed out of $v$. Since each term lies in $\{-1,0,1\}$, the sum
$$
\sum_{i=1}^3 \sigma_v(e_i)\alpha_c^\vee(e_i)
$$
vanishes if and only if the three terms are either all $0$ or are $1,-1,0$ in some order. This is exactly the required gluing condition.
\end{proof}

\subsection{Induced \texorpdfstring{$(i,j)$}{(i,j)}-strands}

A valid stranding determines a larger family of induced directed paths indexed by pairs $(i,j)$.

\begin{definition}
Let $G$ be an $\mathfrak{sl}_n$ web graph, let $S\in\mathcal{S}tr(G)$, and let $e\in E(G)$. Define
$$
L(e)=\left\{(i,j): 1\le i<j\le n,\ \sum_{c=i}^{j-1}\alpha_c^\vee(e)\neq 0\right\}.
$$
Write $L(e)=L_+(e)\sqcup L_-(e)$, where
$$
L_+(e)=\left\{(i,j): \sum_{c=i}^{j-1}\alpha_c^\vee(e)=1\right\},
\qquad
L_-(e)=\left\{(i,j): \sum_{c=i}^{j-1}\alpha_c^\vee(e)=-1\right\}.
$$

Define a directed graph $L_{(i,j)}(S)$ on the vertex set $V(G)$ by declaring that for an edge $e=u\mapsto v$ of $G$,
\begin{itemize}
    \item if $(i,j)\in L_+(e)$, then $u\mapsto v$ is an edge of $L_{(i,j)}(S)$;
    \item if $(i,j)\in L_-(e)$, then $v\mapsto u$ is an edge of $L_{(i,j)}(S)$;
    \item if $(i,j)\notin L(e)$, then no edge between $u$ and $v$ appears in $L_{(i,j)}(S)$.
\end{itemize}
We refer to each connected component of $L_{(i,j)}(S)$ as an \emph{$(i,j)$-strand} of $S$.
\end{definition}

\begin{remark}
The simple strands are exactly the $(c,c+1)$-strands. More generally, an edge carries an $(i,j)$-strand precisely when the numbers of simple strands of colors in the interval $[i,j-1]$ directed with and against the edge are unequal.
\end{remark}

\begin{lemma}\label{lem: strand graph}
Let $G$ be an $\mathfrak{sl}_n$ web graph, let $S\in\mathcal{S}tr(G)$, and let $1\le i<j\le n$. 
Each $(i,j)$-strand of $S$ is either a directed closed loop or a directed path with endpoints on the boundary.
\end{lemma}

\begin{proof}
Fix an interior vertex $v$ with incident edges $e_1,e_2,e_3$. By Theorem~\ref{thm:net zero strand contribution at each vertex},
$$
\sum_{r=1}^3 \sigma_v(e_r)\alpha_c^\vee(e_r)=0
\qquad\textup{for each } c.
$$
Summing over $c=i,\ldots,j-1$ gives
$$
\sum_{r=1}^3 \sigma_v(e_r)\left(\sum_{c=i}^{j-1}\alpha_c^\vee(e_r)\right)=0.
$$
Each summand lies in $\{-1,0,1\}$, so the same argument as in the proof of Theorem~\ref{thm:net zero strand contribution at each vertex} shows that at $v$ there are either no incident $(i,j)$-strands, or exactly one entering and one leaving. This proves the claim.
\end{proof}

\begin{example}\label{ex: strands}
Figure~\ref{fig: ex strands} shows the stranding from Example~\ref{ex:stranding example} together with the corresponding $(2,4)$-strands.

\begin{figure}[h]
   \begin{center}  
\raisebox{12pt}{\begin{tikzpicture}[scale=.75]

\draw[style=dashed, <->] (0,0)--(7,0);

\begin{scope}[thick,decoration={
    markings,
    mark=at position 0.6 with {\arrow{<}}}
    ] 

\draw[postaction={decorate},style=thick, blue] (3.05,-2)--(3.05,-1);
\draw[postaction={decorate},style=thick, green] (4,-1.05)--(3,-1.05);
\draw[postaction={decorate},style=thick, blue] (3,-1.05) to[out=180,in=270] (1.95,0);
\draw[postaction={decorate},style=thick, blue] (6.05,0) to[out=270,in=0] (4,-2.05);
\draw[postaction={decorate},style=thick, green] (5,0) to[out=270,in=0] (4,-1);

\end{scope}

\begin{scope}[thick,decoration={
    markings,
    mark=at position 0.4 with {\arrow{>}}}
    ] 
\draw[postaction={decorate},style=thick, green] (1,0) to[out=270,in=180] (3,-2);
\draw[postaction={decorate},style=thick, blue] (3,-2)--(4,-2);
\draw[postaction={decorate},style=thick, red] (4,-2)--(4,-1);
\draw[postaction={decorate},style=thick, red] (5.95,0) to[out=270,in=0] (4,-1.95);
\draw[postaction={decorate},style=thick, green] (2.95,-2)--(2.95,-1);
\draw[postaction={decorate},style=thick, red] (4,-.95)--(3,-.95);
\draw[postaction={decorate},style=thick, red] (3,-.95) to[out=180,in=270] (2.05,0);

\end{scope}

\draw[radius=.08, fill=black](1,0)circle;
\draw[radius=.08, fill=black](3,-1)circle;
\draw[radius=.08, fill=black](4,-1)circle;
\draw[radius=.08, fill=black](3,-2)circle;
\draw[radius=.08, fill=black](4,-2)circle;
\draw[radius=.08, fill=black](2,0)circle;
\draw[radius=.08, fill=black](5,0)circle;
\draw[radius=.08, fill=black](6,0)circle;

    \end{tikzpicture}}  
    \hspace{.25in}
   \raisebox{14pt}{\begin{tikzpicture}[scale=.75]

\draw[style=dashed, <->] (0,0)--(7,0);

\begin{scope}[thick,decoration={
    markings,
    mark=at position 0.6 with {\arrow{<}}}
    ] 
\draw[postaction={decorate},style=thick, violet] (5,0) to[out=270,in=0] (4,-1);
\end{scope}

\begin{scope}[thick,decoration={
    markings,
    mark=at position 0.4 with {\arrow{>}}}
    ] 
\draw[postaction={decorate},style=thick, violet] (1,0) to[out=270,in=180] (3,-2);
\draw[postaction={decorate},style=thick, violet] (4,-2)--(4,-1);
\draw[postaction={decorate},style=thick, violet] (6,0) to[out=270,in=0] (4,-2);
\draw[postaction={decorate},style=thick, violet] (3,-2)--(3,-1);
\draw[postaction={decorate},style=thick, violet] (3,-1) to[out=180,in=270] (2,0);
\end{scope}

\draw[radius=.08, fill=black](1,0)circle;
\draw[radius=.08, fill=black](3,-1)circle;
\draw[radius=.08, fill=black](4,-1)circle;
\draw[radius=.08, fill=black](3,-2)circle;
\draw[radius=.08, fill=black](4,-2)circle;
\draw[radius=.08, fill=black](2,0)circle;
\draw[radius=.08, fill=black](5,0)circle;
\draw[radius=.08, fill=black](6,0)circle;

    \end{tikzpicture}}   
    \end{center}

    \caption{The $(2,4)$-strands determined by a valid stranding}\label{fig: ex strands}
\end{figure}
\end{example}

\begin{remark}
     Our $(1,3)$-strands taken with opposite orientation are the same as the flow curves for  $\mathfrak{sl}_3$ webs in \cite{KK}.  The states of colored webs associated to a bicolor $\{i,j\}$ defined in \cite{Robert} are the same as our $(i,j)$ strands taken with opposite orientation. The same curves are the root-labeled strands of Bigelow's cobwebs \cite{Bigelow}. The band diagram for $\mathfrak{sl}_3$ webs in \cite{RTshadow} is the collection of our $(1,3)$ strands without orientation. Morrison uses the related concept of flow labels for webs in his thesis work though the setup is different \cite{MorrisonThesis}.
\end{remark}

\subsection{Edge flips}

Recall from Definition~\ref{definition: flip edge in web graph} that if $\mathcal{E}\subseteq E(G)$, then $G_{\varphi(\mathcal{E})}$ denotes the web obtained by flipping all edges in $\mathcal{E}$.

\begin{lemma}\label{lemma: same combo strandings on web graphs with edges flipped}
Let $G$ be a web graph and let $\mathcal{E}\subseteq E(G)$. Then
$$
\mathcal{S}tr(G)=\mathcal{S}tr\bigl(G_{\varphi(\mathcal{E})}\bigr).
$$
\end{lemma}

\begin{proof}
The graphs $G$ and $G_{\varphi(\mathcal{E})}$ have the same underlying unoriented graph, so any stranding of one is a stranding of the other. It therefore suffices to compare the validity conditions on a flipped edge.

Let $e=u\stackrel{\ell}{\mapsto}v$ be an edge of $\mathcal{E}$, and let $\varphi(e)=v\stackrel{n-\ell}{\longmapsto}u$. Flipping reverses which simple strands are directed with and against the edge, so each $\alpha_c^\vee$ changes sign. Hence
$$
\sum_{c=1}^{n-1}\alpha_c^\vee(\varphi(e))\,c
=
-\sum_{c=1}^{n-1}\alpha_c^\vee(e)\,c.
$$
If the largest-colored simple strand is directed with $e$, then it is directed against $\varphi(e)$. Whereas the validity condition on the original edge $e$ requires a sum of $\ell$, the condition on the flipped edge $\varphi(e)$ requires a sum of $(n-\ell)-n=-\ell$. Thus, validity is preserved on each flipped edge.
\end{proof}

\begin{corollary}
For a fixed stranding, the underlying set of $(i,j)$-strands on an edge is unchanged by edge flipping:
$$
L(e)=L(\varphi(e)).
$$
\end{corollary}

\begin{remark}
If $e$ is flipped, then $S_+(e)$ and $S_-(e)$ are exchanged, and similarly $L_+(e)$ and $L_-(e)$ are exchanged.
\end{remark}

\subsection{Binary labelings}

We now give an equivalent local description of valid strandings.

\begin{definition}
A \emph{binary labeling} of a web graph $G$ is an assignment of an $n$-bit binary vector $b(e)\in\{0,1\}^n$ to each edge $e$ such that:
\begin{itemize}
    \item if $e$ has weight $\ell$, then $b(e)$ has exactly $\ell$ ones;
    \item at each interior vertex $v$ with incident edges $e_1,e_2,e_3$,
    $$
    \sum_{i=1}^3 \sigma_v(e_i)b(e_i)\in \mathbb{Z}\vec{1},
    $$
    where $\vec{1}=(1,\ldots,1)$.
\end{itemize}
We denote the set of binary labelings of $G$ by $\mathcal{B}in(G)$.
\end{definition}

\begin{remark}
We use the terms \emph{weight}, \emph{color}, and \emph{binary label} distinctly: weights belong to edges of the web, colors belong to simple strands, and binary labels are vectors in $\{0,1\}^n$ assigned to edges.
\end{remark}

For $1\le i\le n-1$, let $\vec{\lambda}_i$ be the binary vector with $1$ in the first $i$ positions and $0$ elsewhere. For a binary vector $\vec{b}=(b_1,\ldots,b_n)$, define
$$
\alpha_i^\vee(\vec{b})=b_i-b_{i+1}.
$$
These functions are dual to the vectors $\vec{\lambda}_i$ in the sense that
$$
\alpha_i^\vee(\vec{\lambda}_j)=
\begin{cases}
1 & \textup{if } i=j,\\
0 & \textup{if } i\neq j.
\end{cases}
$$

\begin{lemma}\label{lem:second validity condition for strandings in terms of binary labels}
Let $\vec{b}$ be a binary vector with exactly $\ell$ ones, where $1\le \ell\le n-1$. Then
$$
\sum_{i=1}^{n-1}\alpha_i^\vee(\vec{b})\,i=
\begin{cases}
\ell & \textup{if } b_n=0,\\
\ell-n & \textup{if } b_n=1.
\end{cases}
$$
\end{lemma}

\begin{proof}
A short calculation verifies the formula below from which the lemma follows:
\[\vec{b} = \begin{cases}\sum_{i=1}^{n-1} \alpha^\vee_i(\vec{b}) \vec{\lambda}_i & \textup{if } b_n=0 \\\\
\vec{1}+\sum_{i=1}^{n-1} \alpha^\vee_i(\vec{b}) \vec{\lambda}_i & \textup{if } b_n=1.
\end{cases}\]
\end{proof}

\begin{remark}
The binary vectors $\vec{\lambda}_i$ are coset representatives for the fundamental weights in Lie theory, which form a basis for the quotient $\mathbb{Z}^n/I$ where $I$ is the principal ideal consisting of multiples of the constant vector.

The $\alpha^\vee_i$ are the simple coroots in Lie theory; they are dual to the fundamental weights and form a basis for the dual space of $\mathbb{Z}^n/I$. 
\end{remark}

Next, we relate valid strandings and binary labelings.
\begin{definition}\label{def:bijection between strandings and binary labelings}
Let $G$ be an untagged web and let $e\in E(G)$.

Given $b\in\mathcal{B}in(G)$, define a stranding $S_b$ by
$$
(S_b)_+(e)=\{c:\alpha_c^\vee(b(e))=1\},
\qquad
(S_b)_-(e)=\{c:\alpha_c^\vee(b(e))=-1\}.
$$

Conversely, given $S\in\mathcal{S}tr(G)$ with $c_{\max}$ the largest simple strand color appearing on $e$, define a binary vector $b_S(e)$ by setting
$$
b_S(e)_n=
\begin{cases}
0 & \textup{if } \alpha_{c_{\max}}^{\vee}(e)=1\\
1 & \textup{if } \alpha_{c_{\max}}^{\vee}(e)=-1
\end{cases}
$$
and requiring
\begin{equation}\label{equation: defining b_S(e)}
\alpha_c^\vee(b_S(e))=b_S(e)_c-b_S(e)_{c+1}=\alpha_c^\vee(e)
\qquad\textup{for } 1\le c\le n-1.
\end{equation}
\end{definition}

\begin{example}\label{ex:stranding and binary labeling bijection}
    On the left in Figure \ref{fig: ex stranding and binary labeling bijection} is the stranding $S$ from Example \ref{ex:stranding example}, and on the right is the binary labeling $b_S$ described in Definition \ref{def:bijection between strandings and binary labelings}. For instance, consider the rightmost boundary edge $e$ of the web which has weight 1 and is oriented out of the boundary. Note that $S_{+}(e)=\{2\}$ and $S_{-}(e)=\{1\}$, so $\alpha^\vee_2(e)=1$ while $\alpha^\vee_1(e)=-1$. By definition, we have $b_S(e)_4= 0$ and so $b_S(e)=0100$.

    \begin{figure}[h]
    \begin{center}
         \raisebox{12pt}{\begin{tikzpicture}[scale=.75]

\draw[style=dashed, <->] (0,0)--(7,0);

\begin{scope}[thick,decoration={
    markings,
    mark=at position 0.6 with {\arrow{<}}}
    ] 

\draw[postaction={decorate},style=thick, blue] (3.05,-2)--(3.05,-1);
\draw[postaction={decorate},style=thick, green] (4,-1.05)--(3,-1.05);
\draw[postaction={decorate},style=thick, blue] (3,-1.05) to[out=180,in=270] (1.95,0);
\draw[postaction={decorate},style=thick, blue] (6.05,0) to[out=270,in=0] (4,-2.05);

\draw[postaction={decorate},style=thick, green] (5,0) to[out=270,in=0] (4,-1);

\end{scope}

\begin{scope}[thick,decoration={
    markings,
    mark=at position 0.4 with {\arrow{>}}}
    ] 
\draw[postaction={decorate},style=thick, green] (1,0) to[out=270,in=180] (3,-2);
\draw[postaction={decorate},style=thick, blue] (3,-2)--(4,-2);
\draw[postaction={decorate},style=thick, red] (4,-2)--(4,-1);
\draw[postaction={decorate},style=thick, red] (5.95,0) to[out=270,in=0] (4,-1.95);
\draw[postaction={decorate},style=thick, green] (2.95,-2)--(2.95,-1);
\draw[postaction={decorate},style=thick, red] (4,-.95)--(3,-.95);
\draw[postaction={decorate},style=thick, red] (3,-.95) to[out=180,in=270] (2.05,0);

\end{scope}

\draw[radius=.08, fill=black](1,0)circle;
\draw[radius=.08, fill=black](3,-1)circle;
\draw[radius=.08, fill=black](4,-1)circle;
\draw[radius=.08, fill=black](3,-2)circle;
\draw[radius=.08, fill=black](4,-2)circle;
\draw[radius=.08, fill=black](2,0)circle;
\draw[radius=.08, fill=black](5,0)circle;
\draw[radius=.08, fill=black](6,0)circle;

    \end{tikzpicture}} 
    \hspace{.25in}
       \raisebox{3pt}{\begin{tikzpicture}[scale=.75]

\draw[style=dashed, <->] (0,0)--(7,0);

\begin{scope}[thick,decoration={
    markings,
    mark=at position 0.5 with {\arrow{<}}}
    ] 

\draw[postaction={decorate},style=thick] (1,0) to[out=270,in=180] (3,-2);
\draw[postaction={decorate},style=thick] (3,-2)--(3,-1);
\draw[postaction={decorate},style=thick] (3,-2)--(4,-2);
\draw[postaction={decorate},style=thick] (5,0) to[out=270,in=0] (4,-1);
\draw[postaction={decorate},style=thick] (4,-1)--(3,-1);
\draw[postaction={decorate},style=thick] (3,-1) to[out=180,in=270] (2,0);
\end{scope}

\begin{scope}[thick,decoration={
    markings,
    mark=at position 0.5 with {\arrow{>}}}
    ] 

\draw[postaction={decorate},style=thick] (6,0) to[out=270,in=0] (4,-2);
\draw[postaction={decorate},style=thick] (4,-2)--(4,-1);
\end{scope}

\draw[radius=.05, fill=black](1,0)circle;
\draw[radius=.05, fill=black](3,-1)circle;
\draw[radius=.05, fill=black](4,-1)circle;
\draw[radius=.05, fill=black](3,-2)circle;
\draw[radius=.05, fill=black](4,-2)circle;
\draw[radius=.05, fill=black](2,0)circle;
\draw[radius=.05, fill=black](5,0)circle;
\draw[radius=.05, fill=black](6,0)circle;

\node at (1.7,-.25) {\tiny{$1011$}};
\node at (.8,-1) {\tiny{$0001$}};
\node at (3.5,-.75) {\tiny{$0010$}};
\node at (3.5,-2.25) {\tiny{$0111$}};
\node at (2.6, -1.5) {\tiny{$1001$}};
\node at (4.4, -1.5) {\tiny{$1100$}};
\node at (6.2, -1) {\tiny{$0100$}};
\node at (5.3, -.25) {\tiny{$1110$}};
  
    \end{tikzpicture}}
    \end{center}
    \caption{A stranding $S$ and its corresponding binary labeling $b_S$}\label{fig: ex stranding and binary labeling bijection}
    \end{figure}
\end{example}

\begin{theorem}\label{thm: strandings are bijective with binary labelings}
For any $\mathfrak{sl}_n$ web graph $G$, the assignments
$$
b\longmapsto S_b,
\qquad
S\longmapsto b_S
$$
define a bijection
$$
\mathcal{B}in(G)\cong \mathcal{S}tr(G).
$$
\end{theorem}

\begin{proof}
We first check the constructions edgewise. Let $e$ be an edge of $G$.

Let $b\in\mathcal{B}in(G)$. By definition, the simple strands of $S_b(e)$ are determined by the signs of the differences $\alpha_c^\vee(b(e))=b_c-b_{c+1}$. These nonzero values necessarily alternate in sign, so the first validity condition holds. The second validity condition follows from Lemma~\ref{lem:second validity condition for strandings in terms of binary labels}.

Conversely, let $S\in\mathcal{S}tr(G)$. Equation~\eqref{equation: defining b_S(e)} determines $b_S(e)$ uniquely from the values $\alpha_c^\vee(e)$ together with the value of the last entry. Because the nonzero values $\alpha_c^\vee(e)$ alternate in sign, the vector $b_S(e)$ is binary. The second validity condition for $S$ implies that $b_S(e)$ has exactly as many ones as the weight of $e$.

It remains to check the vertex conditions. Let $v$ be an interior vertex with incident edges $e_1,e_2,e_3$. 

Suppose $S\in\mathcal{S}tr(G)$. By Theorem~\ref{thm:net zero strand contribution at each vertex},
$$
\sum_{i=1}^3 \sigma_v(e_i)\alpha_c^\vee(e_i)=0
\qquad\textup{for all } c.
$$
Applying the coroots $\alpha_c^\vee$ to the vector
$$
\sum_{i=1}^3 \sigma_v(e_i)b_S(e_i)
$$
therefore gives zero for every $c$, so this vector is a multiple of $\vec{1}$. Hence $b_S$ is a binary labeling.

Now let $b\in\mathcal{B}in(G)$.  By the definition of binary labeling, we have
$$
\sum_{i=1}^3 \sigma_v(e_i)b(e_i)\in \mathbb{Z}\vec{1}.
$$
 Therefore, applying $\alpha_c^\vee$ to this expression yields
$$
\sum_{i=1}^3 \sigma_v(e_i)\alpha_c^\vee(b(e_i))=0
\qquad\textup{for all } c.
$$
Since $\alpha_c^\vee(b(e_i))=\alpha_c^\vee(e_i)$ by construction of the stranding $S_b$, Theorem~\ref{thm:net zero strand contribution at each vertex} shows that the edgewise strand data glue to a valid stranding. Thus $S_b\in\mathcal{S}tr(G)$.

Finally, a straightforward calculation shows $(b_{S_b}(e))_n=b(e)_n$ for each edge $e$ of $G$. Since the two constructions recover the same values of all $\alpha_c^\vee$ on every edge, we conclude they are inverse to one another.
\end{proof}

The bijection gives a convenient description of $(i,j)$-strands in terms of binary labels.

\begin{corollary}\label{cor:strands and binary labeling}
Let $S\in\mathcal{S}tr(G)$ and let $b_S(e)=b_1\cdots b_n$ be the binary label on an edge $e$. Then
$$
L(e)=\{(i,j): i<j \textup{ and } b_i\neq b_j\}.
$$
Moreover,
$$
L_+(e)=\{(i,j)\in L(e): b_i-b_j=1\},
\qquad
L_-(e)=\{(i,j)\in L(e): b_i-b_j=-1\}.
$$
\end{corollary}

\begin{proof}
For $1\le i<j\le n$,
$$
b_i-b_j
=
\sum_{c=i}^{j-1}(b_c-b_{c+1})
=
\sum_{c=i}^{j-1}\alpha_c^\vee(e).
$$
The claim follows directly from the definition of $L(e)$, $L_+(e)$, and $L_-(e)$.
\end{proof}

\begin{remark}\label{rem:two alphas agree}
The symbol $\alpha_c^\vee$ is used in two places: as the coroot
$\alpha_c^\vee(\vec{b})=b_c-b_{c+1}$ on binary vectors, and as the edgewise
quantity $\alpha_c^\vee(e)$ attached to a stranding. Theorem~\ref{thm:
strandings are bijective with binary labelings} shows that, under the
bijection $\mathcal{B}in(G)\cong\mathcal{S}tr(G)$, the two agree:
$\alpha_c^\vee(e)=\alpha_c^\vee(b_S(e))$ for every edge $e$. In this sense,
the function $\alpha_c^\vee$ on stranded graphs  is the map induced on $G$
  by the coroot. The simple strands on an edge record the coordinates of the
  weight $b_S(e)$ in the basis of fundamental weights $\vec{\lambda}_c$, and
  these coordinates assemble into a coherent global structure across the
  entire web.

Binary labels are convenient for local, edgewise computations, but they
are sensitive to edge flipping and other notational conventions: under a
flip, $b(e)$ is replaced by $\vec{1}-b(e)$    (Corollary~\ref{cor:binary labels under flips}). By contrast, strands are
  indexed by roots and record the coordinates of an edge's weight. A flip
  negates these coordinates and reverses the edge, so the strands themselves,
  as directed colored paths, are unchanged
  (Lemma~\ref{lemma: same combo strandings on web graphs with edges flipped}).
\end{remark}

\begin{corollary}\label{cor:binary labels under flips}
Let $G$ be a web graph and let $\mathcal{E}\subseteq E(G)$. Then $\mathcal{B}in(G)$ and $\mathcal{B}in(G_{\varphi(\mathcal{E})})$ are in bijection. Under this bijection, a binary labeling $b$ maps to the labeling $b'$ defined by
$$
b'(e)=b(e)\quad\textup{for } e\notin\mathcal{E},
\qquad
b'(\varphi(e))=\vec{1}-b(e)\quad\textup{for } e\in\mathcal{E}.
$$
\end{corollary}

\begin{proof}
By Lemma~\ref{lemma: same combo strandings on web graphs with edges flipped}, a valid stranding on $G$ is the same thing as a valid stranding on $G_{\varphi(\mathcal{E})}$. Passing through the bijection of Theorem~\ref{thm: strandings are bijective with binary labelings} gives the corresponding bijection of binary labelings.

If $e$ is not flipped, its simple strand data are unchanged, so its binary label is unchanged. If $e$ is flipped, then $S_+(e)$ and $S_-(e)$ are exchanged, so every value $\alpha_c^\vee(e)$ changes sign. This replaces the binary vector $b(e)$ by its complement $\vec{1}-b(e)$.
\end{proof} 

 \section{Web vectors from strandings}\label{section:main construction}
 
We now construct, for each untagged web graph $G$, a vector $f(G)$ in the tensor product $V(\vec{k})$. The construction is a state-sum over the valid strandings of $G$, with no choice of decomposition or isotopy required. In this section, we prove $f(G)$ is $\uq$-invariant, so $f$ maps $\mathcal{F}(\vec{k})$ to $\textup{Inv}(\vec{k})$.
 
\subsection{The state-sum formula}
 
Let $G\in \mathcal{F}(\vec{k})$ have boundary vertices $v_1, \ldots, v_m$ ordered left to right, with corresponding boundary edges $e_j$ of weight $\ell_j$. Recall that
$$k_j=\begin{cases}
    \ell_j & \textup{ if } \sigma_{v_j}(e_j)=1, \\
    n-\ell_j & \textup{ if } \sigma_{v_j}(e_j)=-1,
\end{cases}.$$
For each $S\in\mathcal{S}tr(G)$, set
$$\hat{b}(e_j)=\begin{cases}
    b_S(e_j) & \textup{ if } \sigma_{v_j}(e_j)=1, \\
    b_S(\varphi(e_j)) & \textup{ if } \sigma_{v_j}(e_j)=-1,
\end{cases}$$
and let $x_S=x_{\hat{b}(e_1)}\otimes \cdots \otimes x_{\hat{b}(e_m)}$. Each $\hat{b}(e_j)$ has exactly $k_j$ ones, so $x_S\in V(\vec{k})$. For instance, the stranding in Example~\ref{ex:stranding and binary labeling bijection} has associated monomial
$$x_{0001}\otimes x_{0100}\otimes x_{1110}\otimes x_{1011}  \in V_1\otimes V_1\otimes V_3\otimes V_3.$$
 
The coefficients in $f(G)$ track clockwise and counterclockwise strands. With the planar embedding inherited from $G$, each $(i,j)$ strand is either clockwise or counterclockwise. For $S\in\mathcal{S}tr(G)$, let $x_{(i,j)}(S)$ be the number of \emph{closed} clockwise $(i,j)$ strands and let $y_{(i,j)}(S)$ be the \emph{total number} of counterclockwise $(i,j)$ strands, both closed and open. Set
$$x(S)=\sum_{1\leq i<j\leq n} x_{(i,j)}(S) \hspace{.25in} \textup{ and } \hspace{.25in} y(S)=\sum_{1\leq i<j\leq n} y_{(i,j)}(S).$$
Remark~\ref{rem:symmetry} discusses the asymmetry between closed and open components.
 
We define
\begin{equation}\label{equation: definition of invariant vector} f(G)=\sum_{S\in\mathcal{S}tr(G)} (-q)^{x(S)-y(S)} x_S\end{equation}
and call $x(S)-y(S)$ the \emph{strand exponent} of $S$. Extending linearly gives a map $f: \mathcal{F}(\vec{k})\rightarrow V(\vec{k})$.
 
\begin{remark}
When $n=3$, the map $f$ coincides with the one defined by Khovanov and Kuperberg in \cite{KK}, up to a global change of variable replacing our $q$ with $q^{1/2}$. The notational translations, listing theirs first, are $V^+=V_1$, $V^-=V_2$, $e^+_1=x_1$, $e^+_0=x_2$, $e^+_{-1}=x_3$, $e^-_1=x_2\wedge x_1$, $e^-_0=x_3\wedge x_1$, and $e^-_{-1}=x_3\wedge x_2$. The coefficient in their state-sum formula tracks local contributions coming from a decomposition of the web.   In \cite[Section 8]{KK} they also describe the weight of a state through its closed clockwise flow loops and U-turns, which is an informal global count.
\end{remark}
 
\begin{remark}
Robert \cite{Robert} assigns to each coloring of a closed web a degree defined in terms of rotation numbers of its bicolored cycles, and proves
that the resulting state sum agrees with the MOY evaluation. His bicolored
cycles are our closed $(i,j)$-strands, and for closed webs our strand exponent
agrees with his degree after a normalization. He also defines a degree
for webs with boundary \cite[Definitions 2.8--2.10]{Robert}, which he uses to
verify skein relations. Theorem~\ref{thm:invariant vector flow formula} shows
that, with the convention above for open strands, the coefficients form an
invariant vector.
\end{remark}
 
\begin{remark}\label{rem:symmetry}
While the coefficient $x(S)-y(S)$ in the formula for $f(G)$ appears asymmetrical, the formula is in fact symmetric in the broader setting of web graphs between a top and a bottom horizontal boundary axis. In that setting, let $x(S)$ count the clockwise top-anchored strands plus the clockwise closed strands, and let $y(S)$ count the counterclockwise bottom-anchored strands plus the counterclockwise closed strands. The single-axis formula~\eqref{equation: definition of invariant vector} is the special case in which all boundary edges meet a single axis (so no top-anchored strands arise), and the cap map of this section already implements this two-axis convention at the boundary it glues.

All constructions of Section~\ref{section:tagless vs tagged} --- the comparison map $\psi$, the surjectivity of $f$, and the generating relations for $\ker(f)$ --- extend without change to the two-axis setting. The applications of Section~\ref{section: applications} use boundary-face combinatorics specific to the single-axis setting, so we work there for the body of the paper.
\end{remark}
 
A first easy property of $f$ is that flipping edges leaves it unchanged.
 
\begin{lemma}\label{lem:invt and flips}
For any web graph $G$ and any $\mathcal{E}\subseteq E(G)$, $f(G)=f(G_{\varphi(\mathcal{E})})$.
\end{lemma}
\begin{proof}
By Lemma~\ref{lemma: same combo strandings on web graphs with edges flipped}, $\mathcal{S}tr(G)=\mathcal{S}tr(G_{\varphi(\mathcal{E})})$, and the boundary weight vector $\vec{k}$ is also unchanged. Each stranding therefore contributes the same monomial with the same coefficient to both sums.
\end{proof}
 
\subsection{Invariance}
 
We show $\textup{Im}(f)\subseteq \textup{Inv}(\vec{k})$ by considering webs with at most one interior vertex, then arguing inductively. If $G$ has no boundary, $f(G)\in\mathbb{C}(q)$ is trivially invariant, so the work is in the cup, tripod, and cap cases. The proofs of these are given in Appendix~\ref{appendix:invariance calculations}. 
 
\begin{lemma}\label{lem:cup invariant}
Let $G$ be a connected $\mathfrak{sl}_n$ web that is a cup graph. Then $f(G)$ is $\uq$-invariant.
\end{lemma}
\begin{proof}
See Appendix~\ref{appendix:invariance calculations}.
\end{proof}

\begin{lemma}\label{lem:tripod invariant}
Let $G$ be a connected $\mathfrak{sl}_n$ web that is a tripod graph. Then $f(G)$ is $\uq$-invariant.
\end{lemma}
\begin{proof}
See Appendix~\ref{appendix:invariance calculations}.
\end{proof}
 
A cap is not itself a web graph in our conventions, but it encodes a $\uq$-equivariant evaluation map that reduces the number of tensor factors in the ambient space.
 
Consider adjacent boundary vertices $v_1$ and $v_2$ with edges $e_1$ and $e_2$ of opposite orientations and weights $k$ and $n-k$ (read left to right). A stranding on $\{e_1, e_2\}$ \emph{glues at the boundary} if $S_+(e_1)=S_-(e_2)$ and $S_-(e_1)=S_+(e_2)$, equivalently if the binary labels satisfy $\hat{{b}}(e_2)=\vec{1}-\hat{{b}}(e_1)$. In this case, $x_{\hat{{b}}(e_1)}\otimes x_{\hat{{b}}(e_2)}\in V_k\otimes V_{n-k}$ \emph{induces a stranding} on the arc joining $v_1$ and $v_2$ above the boundary axis.
The induced stranding depends only on $\hat{b}(e_1)$; for any binary vector $\vec{b}$ with exactly $k$ ones, write $S_{\vec{b}}$ for this stranding and $z(S_{\vec{b}})$ for its number of clockwise $(i,j)$ strands. The cap map $C_k:V_k\otimes V_{n-k}\rightarrow \mathbb{C}(q)$ is
$$C_k(x_{\vec{b}_1}\otimes x_{\vec{b}_2})=\begin{cases}
    (-q)^{z(S_{\vec{b}_1})} & \textup{ if } \vec{b}_2=\vec{1}-\vec{b}_1, \\
    0 & \textup{ otherwise.}
\end{cases}$$
 
\begin{lemma}\label{lem:cap equivariant}
The cap map $C_k$ is $\uq$-equivariant.
\end{lemma}
\begin{proof}
See Appendix~\ref{appendix:invariance calculations}.
\end{proof}
 
Tensoring with identities, $C_k$ extends to act on any $\uq$-representation containing $V_k\otimes V_{n-k}$ as a tensor factor. Abusing notation, we continue to write $C_k$ for any such extension and call it a cap map.
 
\begin{corollary}\label{cor:capmap}
Let $W_1$ and $W_2$ be $\uq$-representations and $1\leq k<n$. Then $Id_{W_1}\otimes C_k\otimes Id_{W_2}$ is $\uq$-equivariant.
\end{corollary}
 
It suffices to prove $f(G)$ is invariant for connected $G$; for disconnected $G$, the components contribute invariant factors that tensor into a fixed position within an ambient invariant vector.
 
\begin{lemma}\label{lem:disjoint union invariant}
Let $G$ be a web graph with connected components $G_1, \ldots, G_m$ such that each $f(G_i)$ is invariant. Then $f(G)$ is invariant.
\end{lemma}
\begin{proof}
Induct on $m$. Some component of $G$ has consecutive endpoints; call it $G_1$, and write $G = G_1 \sqcup G'$. By induction, $f(G')$ is invariant. Each stranding $S$ of $G$ is a stranding $S'$ of $G'$ together with a stranding $S_1$ of $G_1$, so $x_S=x_{S',L}\otimes x_{S_1} \otimes x_{S',R}$ where $x_{S'}=x_{S',L} \otimes x_{S',R}$. Then
\begin{align*}
    f(G) &= \sum_{\substack{S'\in\mathcal{S}tr(G')\\S_1\in\mathcal{S}tr(G_1)}}(-q)^{x(S')+x(S_1)-(y(S')+y(S_1))}x_{S',L}\otimes x_{S_1} \otimes x_{S',R}\\
    &= \sum_{S'\in\mathcal{S}tr(G')}(-q)^{x(S')-y(S')}x_{S',L} \otimes f(G_1) \otimes x_{S',R}.
\end{align*}
Using $E_i(f(G_1))=0$ and $K_i(f(G_1))=f(G_1)$,
\begin{align*}
E_i(f(G))&= \hspace{-.2in}\sum_{S'\in\mathcal{S}tr(G')}(-q)^{x(S')-y(S')} \Bigl( E_i(x_{S',L}) \otimes f(G_1) \otimes K_i(x_{S',R})\\
&\hspace{.1in}
+x_{S',L} \otimes f(G_1) \otimes E_i(x_{S',R}) \Bigr).
\end{align*}
This is $f(G_1)$ tensored into each term of $E_i(f(G'))=0$, so $E_i(f(G))=0$. Analogous arguments give $F_i(f(G))=0$ and $K_i(f(G))=f(G)$.
\end{proof}
 
We now prove the main theorem of the section. The proof uses a tripod decomposition, but the result demonstrates the web vector is independent of choice of decomposition.
 
\begin{theorem}\label{thm:invariant vector flow formula}
For each $G\in F(\vec{k})$,
$$f(G)=\sum_{S\in\mathcal{S}tr(G)} (-q)^{x(S)-y(S)} x_S\in\textup{Inv}(\vec{k}),$$
and $f$ extends linearly to a map $\mathcal{F}(\vec{k})\rightarrow \textup{Inv}(\vec{k})$.
\end{theorem}
\begin{proof}
Let $G'$ be a tripod decomposition of $G$, and induct on the number of caps. If $G'$ has no caps, it is a disjoint union of cups and tripods, and Lemmas~\ref{lem:cup invariant}, \ref{lem:tripod invariant}, and~\ref{lem:disjoint union invariant} give the result.
 
Now suppose $G'$ has $t+1$ caps. After an isotopy that moves a topmost cap above all others (and does not change $f$), insert a horizontal axis below the boundary axis separating this cap from the rest, and let $\bar{G}$ be the web below the axis. Say $G'$ has boundary $v_1,\ldots,v_m$ with weight vector $\vec{k}=(k_1,\ldots,k_m)$. Then $\bar{G}$ has boundary $\bar{v}_1,\ldots,\bar{v}_i,\bar{v}_{\alpha},\bar{v}_{\beta},\bar{v}_{i+1},\ldots,\bar{v}_m$ with weight vector $\vec{\bar{k}}=(k_1,\ldots,k_i,k_{\alpha},k_{\beta},k_{i+1},\ldots,k_m)$; see Figure~\ref{fig:cut top cap}. The cap map $C_{k_{\alpha}}:V(\vec{\bar{k}})\rightarrow V(\vec{k})$ is $\uq$-equivariant by Corollary~\ref{cor:capmap}, and $f(\bar{G})\in\textup{Inv}(\vec{\bar{k}})$ by induction, so $C_{k_{\alpha}}(f(\bar{G}))\in\textup{Inv}(\vec{k})$. It remains to show $f(G')=C_{k_{\alpha}}(f(\bar{G}))$.
 
\begin{figure}[h]
\raisebox{-30pt}{\scalebox{.9}{\begin{tikzpicture}[scale=.65]
  \draw[style=dashed, <->] (0,4)--(13,4);
    \draw[style=dashed, <->] (0,2.6)--(13,2.6);
  \draw[style=dashed, red] (0,-.6)--(13,-.6);
   \draw[thick, black] (1,4)--(1,1.8);
 
   \draw[thick, black] (.5,.2)--(.5,-1);
   \draw[thick, black] (12.5,.2)--(12.5,-1);
 
      \node at (6.5,-.3) {$\cdots \hspace{.25in} \cdots \hspace{.25in} \cdots \hspace{.25in} \cdots \hspace{.25in} \cdots$};
 
   \node at (2.5,3) {$\cdots$};
   \node at (10.5,3) {$\cdots$};
    \draw[thick, black] (4,4)--(4,1.8);
    \draw[thick, black] (9,4)--(9,1.8);
         \draw[thick, black] (12,4)--(12,1.8);
    \draw[thick, black] (5.5,1.8) --(5.5,2.6) to[out=90,in=180] (6.5,3.25) to[out=0,in=90] (7.5,2.6)--(7.5,1.8);
    \draw[fill=gray!25, thick] (0,1.8)--(13,1.8)--(13,.1)--(0,.1)--(0,1.8);
    \node at (6.5, 1.1) {Caps};
 
     \draw[fill=gray!25, thick] (0,-1)--(13,-1)--(13,-3)--(0,-3)--(0,-1);
    \node at (6.5, -2) {Cups and tripods};
 
    \draw[radius=.08, fill=black](5.5,2.6)circle;
\draw[radius=.08, fill=black](7.5,2.6)circle;
 
\draw[radius=.08, fill=black](1,2.6)circle;
\draw[radius=.08, fill=black](4,2.6)circle;
\draw[radius=.08, fill=black](9,2.6)circle;
\draw[radius=.08, fill=black](12,2.6)circle;
 
\draw[radius=.08, fill=black](1,4)circle;
 
\node at (.75, 2.2) {\tiny{$\bar{v}_1$}};
\node at (3.75, 2.2) {\tiny{$\bar{v}_i$}};
\node at (5.25, 2.2) {\tiny{$\bar{v}_{\alpha}$}};
\node at (7.85, 2.2) {\tiny{$\bar{v}_{\beta}$}};
\node at (9.5, 2.2) {\tiny{$\bar{v}_{i+1}$}};
\node at (12.35, 2.2) {\tiny{$\bar{v}_m$}};
 
\node at (1, 4.3) {\tiny{$v_1$}};
\node at (3.75, 4.3) {\tiny{$v_i$}};
\node at (9.5, 4.3) {\tiny{$v_{i+1}$}};
\node at (12.35, 4.3) {\tiny{$v_m$}};
\draw[radius=.08, fill=black](4,4)circle;
\draw[radius=.08, fill=black](9,4)circle;
\draw[radius=.08, fill=black](12,4)circle;
 
\draw (13.5,2.6)--(14,2.6)--(14,-3)--(13.5,-3);
 
\node at (15.5,-.2) {$\bar{G}\in F(\vec{\bar{k}})$};
 
\draw (-.5,4)--(-1,4)--(-1,-3)--(-.5,-3);
 
\node at (-2.5,.5) {$G'\in F(\vec{k})$};
 
    \end{tikzpicture}}}
 
    \caption{Forming $\bar{G}$ from $G'$}\label{fig:cut top cap}
\end{figure}
 
Write $W_1=V_{k_1}\otimes \cdots \otimes V_{k_i}$ and $W_2=V_{k_{i+1}}\otimes \cdots \otimes V_{k_m}$, so $f(\bar{G})\in \textup{Inv}(W_1\otimes V_{k_{\alpha}}\otimes V_{k_{\beta}}\otimes W_2)$. For $\bar{S}\in\mathcal{S}tr(\bar{G})$, write $\hat{b}_{\alpha}=\hat{b}(\bar{e}_{\alpha})$ and $\hat{b}_{\beta}=\hat{b}(\bar{e}_{\beta})$, so $x_{\bar{S}}=x_{\bar{S}_1}\otimes x_{\hat{b}_\alpha}\otimes x_{\hat{b}_\beta}\otimes x_{\bar{S}_2}$ with $x_{\bar{S}_i}\in W_i$. The cap map kills any term whose stranding does not extend across the cap:
$$C_{k_{\alpha}}(x_{\bar{S}})= \begin{cases}
(-q)^{z(S_{\hat{b}_\alpha})} x_{\bar{S}_1}\otimes x_{\bar{S}_2} & \textup{ if } \hat{b}_\beta = \vec{1}-\hat{b}_\alpha, \\
0 & \textup{ otherwise. }
\end{cases}$$
Writing $S$ for the stranding of $G'$ obtained from $\bar{S}$ by gluing across the cap,
$$C_{k_{\alpha}}(f(\bar{G}))= \hspace{-.1in}\sum_{\substack{\bar{S}\in\mathcal{S}tr(\bar{G})\\\hat{b}_\beta=\vec{1}-\hat{b}_\alpha}}
 \hspace{-.1in}(-q)^{x(\bar{S})-y(\bar{S})+z(S_{\hat{b}_\alpha})} x_{\bar{S}_1}\otimes x_{\bar{S}_2}
 = \hspace{-.1in} \sum_{S\in\mathcal{S}tr(G')}
 \hspace{-.1in}(-q)^{x(\bar{S})-y(\bar{S})+z(S_{\hat{b}_\alpha})} x_{S}.$$
 
It remains to verify $x(\bar{S})-y(\bar{S})+z(S_{\hat{b}_\alpha})=x(S)-y(S)$ for each such $\bar{S}$. Set $z_{(i,j)}(S_{\hat{b}_\alpha})=1$ if $S_{\hat{b}_\alpha}$ has a clockwise $(i,j)$ strand across the cap and $0$ otherwise; then $z(S_{\hat{b}_\alpha})=\sum_{i<j}z_{(i,j)}(S_{\hat{b}_\alpha})$. We compare $(i,j)$-strand counts in $S$ and $\bar{S}$ case by case.
 
If the cap carries no $(i,j)$ strand, then $z_{(i,j)}(S_{\hat{b}_\alpha})=0$ and the strand counts in $S$ and $\bar{S}$ agree.
 
If the cap carries a counterclockwise $(i,j)$ strand, $z_{(i,j)}(S_{\hat{b}_\alpha})=0$ and the strand enters $\bar{v}_{\beta}$ and exits $\bar{v}_{\alpha}$ in $\bar{G}$. Noncrossingness forces the configuration to be one of those in Figure~\ref{fig:invariant counter cap}, and in each case $x_{(i,j)}(S)-y_{(i,j)}(S)=x_{(i,j)}(\bar{S})-y_{(i,j)}(\bar{S})$.
\begin{figure}[h]
 
\raisebox{14pt}{\begin{tikzpicture}[scale=.5]
  \draw[style=dashed, <->] (2,0)--(6,0);
 
  \draw[radius=.08, fill=black](3,0)circle;
   \draw[radius=.08, fill=black](5,0)circle;
      \node at (2.75, -.5) {\tiny{$\bar{v}_{\alpha}$}};
     \node at (5.5, -.5) {\tiny{$\bar{v}_{\beta}$}};
   \begin{scope}[thick,decoration={
    markings,
    mark=at position 0.5 with {\arrow{>}}}
    ]
   \draw[postaction={decorate}, red] (5,0) to[out=90, in=0] (4,1) to[out=180,in=90] (3,0);
    \draw[postaction={decorate}, dashed, red] (3,0) to[out=270, in=180] (4,-1) to[out=0,in=270] (5,0);
\end{scope}
  \end{tikzpicture}}
  \hspace{.25in}
  \begin{tikzpicture}[scale=.5]
  \draw[style=dashed, <->] (-2,0)--(6,0);
 
  \draw[radius=.08, fill=black](3,0)circle;
   \draw[radius=.08, fill=black](5,0)circle;
      \node at (3.5, -.5) {\tiny{$\bar{v}_{\alpha}$}};
     \node at (5.5, -.5) {\tiny{$\bar{v}_{\beta}$}};
   \begin{scope}[thick,decoration={
    markings,
    mark=at position 0.5 with {\arrow{>}}}
    ]
   \draw[postaction={decorate}, red] (5,0) to[out=90, in=0] (4,1) to[out=180,in=90] (3,0);
    \draw[postaction={decorate}, dashed, red] (3,0) to[out=270, in=0] (2,-1) to[out=180,in=270] (1,0);
    \draw[postaction={decorate}, dashed, red] (-1,0) to[out=270, in=180] (2,-2) to[out=0,in=270] (5,0);
\end{scope}
  \end{tikzpicture}
   \hspace{.25in}
 \raisebox{14pt}{\begin{tikzpicture}[scale=.5]
  \draw[style=dashed, <->] (0,0)--(8,0);
 
  \draw[radius=.08, fill=black](3,0)circle;
   \draw[radius=.08, fill=black](5,0)circle;
        \node at (3.4, -.5) {\tiny{$\bar{v}_{\alpha}$}};
     \node at (4.75, -.5) {\tiny{$\bar{v}_{\beta}$}};
   \begin{scope}[thick,decoration={
    markings,
    mark=at position 0.5 with {\arrow{>}}}
    ]
   \draw[postaction={decorate}, red] (5,0) to[out=90, in=0] (4,1) to[out=180,in=90] (3,0);
    \draw[postaction={decorate}, dashed, red] (3,0) to[out=270, in=0] (2,-1) to[out=180,in=270] (1,0);
    \draw[postaction={decorate}, dashed, red] (7,0) to[out=270, in=0] (6,-1) to[out=180,in=270] (5,0);
\end{scope}
  \end{tikzpicture}}
     \hspace{.25in}
  \begin{tikzpicture}[scale=.5]
  \draw[style=dashed, <->] (2,0)--(10,0);
 
  \draw[radius=.08, fill=black](3,0)circle;
   \draw[radius=.08, fill=black](5,0)circle;
        \node at (2.75, -.5) {\tiny{$\bar{v}_{\alpha}$}};
     \node at (4.75, -.5) {\tiny{$\bar{v}_{\beta}$}};
   \begin{scope}[thick,decoration={
    markings,
    mark=at position 0.5 with {\arrow{>}}}
    ]
   \draw[postaction={decorate}, red] (5,0) to[out=90, in=0] (4,1) to[out=180,in=90] (3,0);
    \draw[postaction={decorate}, dashed, red] (3,0) to[out=270, in=180] (6,-2) to[out=0,in=270] (9,0);
    \draw[postaction={decorate}, dashed, red] (7,0) to[out=270, in=0] (6,-1) to[out=180,in=270] (5,0);
\end{scope}
  \end{tikzpicture}
    \caption{For counterclockwise $(i,j)$ strands across the cap, $x_{(i,j)}(S)=x_{(i,j)}(\bar{S})$ and $y_{(i,j)}(S)=y_{(i,j)}(\bar{S})$.}\label{fig:invariant counter cap}
\end{figure}
 
If $z_{(i,j)}(S_{\hat{b}_\alpha})=1$, then $\bar{S}$ has an $(i,j)$ strand out of $\bar{v}_{\beta}$ and into $\bar{v}_{\alpha}$ in $\bar{G}$. The possible configurations are shown in Figure~\ref{fig:invariant clock cap}, and in each case $x_{(i,j)}(S)-y_{(i,j)}(S)=x_{(i,j)}(\bar{S})-y_{(i,j)}(\bar{S})+1$. 
 
\begin{figure}[h]
 
\raisebox{14pt}{\begin{tikzpicture}[scale=.5]
  \draw[style=dashed, <->] (2,0)--(6,0);
 
  \draw[radius=.08, fill=black](3,0)circle;
   \draw[radius=.08, fill=black](5,0)circle;
   \node at (2.75, -.5) {\tiny{$\bar{v}_{\alpha}$}};
     \node at (5.5, -.5) {\tiny{$\bar{v}_{\beta}$}};
   \begin{scope}[thick,decoration={
    markings,
    mark=at position 0.5 with {\arrow{<}}}
    ]
   \draw[postaction={decorate}, red] (5,0) to[out=90, in=0] (4,1) to[out=180,in=90] (3,0);
    \draw[postaction={decorate}, dashed, red] (3,0) to[out=270, in=180] (4,-1) to[out=0,in=270] (5,0);
\end{scope}
  \end{tikzpicture}}
  \hspace{.25in}
  \begin{tikzpicture}[scale=.5]
  \draw[style=dashed, <->] (-2,0)--(6,0);
 
  \draw[radius=.08, fill=black](3,0)circle;
   \draw[radius=.08, fill=black](5,0)circle;
     \node at (3.5, -.5) {\tiny{$\bar{v}_{\alpha}$}};
     \node at (5.5, -.5) {\tiny{$\bar{v}_{\beta}$}};
   \begin{scope}[thick,decoration={
    markings,
    mark=at position 0.5 with {\arrow{<}}}
    ]
   \draw[postaction={decorate}, red] (5,0) to[out=90, in=0] (4,1) to[out=180,in=90] (3,0);
    \draw[postaction={decorate}, dashed, red] (3,0) to[out=270, in=0] (2,-1) to[out=180,in=270] (1,0);
    \draw[postaction={decorate}, dashed, red] (-1,0) to[out=270, in=180] (2,-2) to[out=0,in=270] (5,0);
\end{scope}
  \end{tikzpicture}
   \hspace{.25in}
 \raisebox{14pt}{\begin{tikzpicture}[scale=.5]
  \draw[style=dashed, <->] (0,0)--(8,0);
 
  \draw[radius=.08, fill=black](3,0)circle;
   \draw[radius=.08, fill=black](5,0)circle;
     \node at (3.4, -.5) {\tiny{$\bar{v}_{\alpha}$}};
     \node at (4.75, -.5) {\tiny{$\bar{v}_{\beta}$}};
   \begin{scope}[thick,decoration={
    markings,
    mark=at position 0.5 with {\arrow{<}}}
    ]
   \draw[postaction={decorate}, red] (5,0) to[out=90, in=0] (4,1) to[out=180,in=90] (3,0);
    \draw[postaction={decorate}, dashed, red] (3,0) to[out=270, in=0] (2,-1) to[out=180,in=270] (1,0);
    \draw[postaction={decorate}, dashed, red] (7,0) to[out=270, in=0] (6,-1) to[out=180,in=270] (5,0);
\end{scope}
  \end{tikzpicture}}
     \hspace{.25in}
  \begin{tikzpicture}[scale=.5]
  \draw[style=dashed, <->] (2,0)--(10,0);
 
  \draw[radius=.08, fill=black](3,0)circle;
   \draw[radius=.08, fill=black](5,0)circle;
          \node at (2.75, -.5) {\tiny{$\bar{v}_{\alpha}$}};
     \node at (4.75, -.5) {\tiny{$\bar{v}_{\beta}$}};
   \begin{scope}[thick,decoration={
    markings,
    mark=at position 0.5 with {\arrow{<}}}
    ]
   \draw[postaction={decorate}, red] (5,0) to[out=90, in=0] (4,1) to[out=180,in=90] (3,0);
    \draw[postaction={decorate}, dashed, red] (3,0) to[out=270, in=180] (6,-2) to[out=0,in=270] (9,0);
    \draw[postaction={decorate}, dashed, red] (7,0) to[out=270, in=0] (6,-1) to[out=180,in=270] (5,0);
\end{scope}
  \end{tikzpicture}
    \caption{For clockwise $(i,j)$ strands across the cap, $x_{(i,j)}(S)-y_{(i,j)}(S)=x_{(i,j)}(\bar{S})-y_{(i,j)}(\bar{S})+1$.}\label{fig:invariant clock cap}
\end{figure}
Summing over $i<j$ gives $x(\bar{S})-y(\bar{S})+z(S_{\hat{b}_\alpha})=x(S)-y(S)$, as required.
\end{proof}
  
 \section{Applications of stranding}\label{section: applications}

We collect several applications of the stranding framework. The first two --- a nonvanishing criterion and the existence of a canonical base stranding --- together establish that every web vector $f(G)$ is nonzero, with a common nonzero monomial. The remaining subsections sketch how stranding interacts with web bases from standard Young tableaux and with Springer fibers; each draws on results developed elsewhere.
 
\subsection{A combinatorial condition for nonvanishing of terms in an invariant vector}\label{section: nonvanishing terms in web vector}
 
While a web may have multiple strandings with the same monomial, the corresponding terms in $f(G)$ cannot cancel. This gives a combinatorial criterion for which monomials appear in a web's invariant vector which is useful, in particular, for constructing web bases.
 
\begin{theorem}\label{thm: nonzero coefficient}
For any $S\in\mathcal{S}tr(G)$, the monomial $x_S$ has a nonzero coefficient in $f(G)$.
\end{theorem}
\begin{proof}
Let $S=S_1,\ldots,S_t$ be the strandings of $G$ whose associated monomial is $x_S$. The coefficient of $x_S$ in $f(G)$ is therefore
$$
\sum_{i=1}^t (-q)^{x(S_i)-y(S_i)}.
$$
This is a sum of Laurent monomials in $q$, each with coefficient $1$ or $-1$. Terms with different exponents are distinct Laurent monomials, so they cannot cancel, while terms with the same exponent are identical and hence combine to a nonzero multiple of that monomial. Therefore the coefficient of $x_S$ is nonzero.
\end{proof}

  \begin{remark}
  The absence of cancellation in Theorem~\ref{thm: nonzero coefficient} was known in two cases: for $n=3$ by Khovanov and Kuperberg \cite{KK}, and at $q=1$ by Fraser, Lam, and Le \cite[Lemma 5.4]{FraserLamLe}. Compare also the state sum of \cite[Theorem 2.9]{GaetzetalRotation}.
  \end{remark}
 
\subsection{A base stranding}\label{section: base stranding}
 
Every web graph $G\in F(\vec{k})$ has a canonical valid stranding, which we call the \emph{base stranding} of $G$ and denote $S_0^G$. We construct it now and record several consequences.
 
Let $G^*$ be the planar dual of $G$, with each $e^*\in E(G^*)$ assigned the weight of $e$ and oriented so that $e^*$ followed by $e$ satisfies the right-hand rule. For $A, B\in V(G^*)$ and a directed path $P$ from $A$ to $B$, define the \emph{modulo $n$ distance} $dist_n(A,B,P)$ to be the signed sum modulo $n$ of weights of edges along $P$, adding a weight if the edge runs forward in $P$ and subtracting if backward.
 
\begin{lemma}
For $A, B\in V(G^*)$, $dist_n(A,B,P)$ depends only on $A$ and $B$, not on $P$. We therefore write $dist_n(A,B)$.
\end{lemma}
\begin{proof}
First, $dist_n(A,A,P)=0$ for any closed path $P$: conservation of flow modulo $n$ at web vertices in $G$ forces this for paths enclosing a single face, and induction on the number of enclosed faces gives the general case.
 
Now let $P_1, P_2$ be paths from $A$ to $B$. Let $C_1$ be the first vertex where they diverge and $C_2$ the next vertex after $C_1$ at which they reunite. Concatenating the segment of $P_1$ from $C_1$ to $C_2$ with the reverse of the corresponding segment of $P_2$ gives a closed loop on which the modulo $n$ distance is zero. Hence $dist_n(A,C_2,P_1)=dist_n(A,C_2,P_2)$. Iterating to $B$ gives $dist_n(A,B,P_1)=dist_n(A,B,P_2)$.
\end{proof}
 
Let $U\in V(G^*)$ be the vertex of the unbounded face. The \emph{base stranding $S_0^G$} is defined as follows: for each face $A$ of $G$ with $c=dist_n(U,A)\neq 0$, insert a clockwise simple strand of color $c$ around the boundary of $A$. If $A$ is partly bounded by the horizontal axis, erase the portion of the strand on the axis. Faces with $dist_n(U,A)=0$ contribute no strands.
 
\begin{example}\label{ex:canonicalstranding example}
On the left in Figure~\ref{fig: ex canonical stranding} is the untagged web $G$ from Example~\ref{ex:untagged web}; on the right is the base stranding $S_0^G$, with simple strand colors $1,2,3$ in blue, red, green, and modulo $4$ distances labeled in violet. The associated monomial is
$$x_{S_0^G}=x_{1000}\otimes x_{0100}\otimes x_{1011}\otimes  x_{0111}.$$
 
\begin{figure}[h]
 \begin{center}
\raisebox{3pt}{\begin{tikzpicture}[scale=.75]
 
\draw[style=dashed, <->] (0,0)--(7,0);
 
\begin{scope}[thick,decoration={
    markings,
    mark=at position 0.5 with {\arrow{<}}}
    ]
\draw[postaction={decorate},style=thick] (1,0) to[out=270,in=180] (3,-2);
\draw[postaction={decorate},style=thick] (3,-2)--(3,-1);
\draw[postaction={decorate},style=thick] (3,-2)--(4,-2);
\draw[postaction={decorate},style=thick] (5,0) to[out=270,in=0] (4,-1);
\draw[postaction={decorate},style=thick] (4,-1)--(3,-1);
\draw[postaction={decorate},style=thick] (3,-1) to[out=180,in=270] (2,0);
\end{scope}
 
\begin{scope}[thick,decoration={
    markings,
    mark=at position 0.5 with {\arrow{>}}}
    ]
\draw[postaction={decorate},style=thick] (6,0) to[out=270,in=0] (4,-2);
\draw[postaction={decorate},style=thick] (4,-2)--(4,-1);
\end{scope}
 
\draw[radius=.08, fill=black](1,0)circle;
\draw[radius=.08, fill=black](3,-1)circle;
\draw[radius=.08, fill=black](4,-1)circle;
\draw[radius=.08, fill=black](3,-2)circle;
\draw[radius=.08, fill=black](4,-2)circle;
\draw[radius=.08, fill=black](2,0)circle;
\draw[radius=.08, fill=black](5,0)circle;
\draw[radius=.08, fill=black](6,0)circle;
 
\node at (.75,-.25) {\tiny{$1$}};
\node at (1.75,-.25) {\tiny{$3$}};
\node at (3.5,-.75) {\tiny{$1$}};
\node at (3.5,-2.25) {\tiny{$3$}};
\node at (2.75, -1.5) {\tiny{$2$}};
\node at (4.25, -1.5) {\tiny{$2$}};
\node at (6.25, -.25) {\tiny{$1$}};
\node at (5.25, -.25) {\tiny{$3$}};
 
    \end{tikzpicture}}
    \hspace{.25in}
 \raisebox{-5pt}{\begin{tikzpicture}[scale=.75]
 
\draw[style=dashed, <->] (0,0)--(7,0);
 
\begin{scope}[thick,decoration={
    markings,
    mark=at position 0.6 with {\arrow{>}}}
    ]
 \draw[thick, postaction={decorate}, blue]   (1.95,0) to[out=270, in=180] (2.95, -1.05);
  \draw[thick, postaction={decorate}, blue]  (2.95, -1.05)--(2.95,-2);
    \draw[thick, postaction={decorate}, blue]  (2.95,-2) to[out=180, in=270] (1,0);
 \draw[thick, postaction={decorate}, red]  (4.95, 0) to[out=270, in=0] (4,-.95);
  \draw[thick, postaction={decorate}, red] (4,-.95)--(3,-.95);
   \draw[thick, postaction={decorate}, red] (3,-.95) to[out=180, in=270] (2.05,0);
 
   \draw[thick, postaction={decorate}, blue] (6,0) to[out=270, in=0] (4.05,-2);
    \draw[thick, postaction={decorate}, blue](4.05,-2)--(4.05,-1.05);
    \draw[thick, postaction={decorate}, blue] (4.05,-1.05) to[out=0, in=270] (5.05,0);
\end{scope}
\begin{scope}[thick,decoration={
    markings,
    mark=at position 0.4 with {\arrow{<}}}
    ]
\draw[thick, postaction={decorate}, green]  (3.05,-1.05)--(3.05,-2);
\draw[thick, postaction={decorate}, green] (3.05,-2)--(3.95,-2);
\draw[thick, postaction={decorate}, green] (3.95, -2)--(3.95,-1.05);
\draw[thick, postaction={decorate}, green](3.95,-1.05)--(3.05,-1.05);
\end{scope}
 
\draw[radius=.08, fill=black](1,0)circle;
\draw[radius=.08, fill=black](3,-1)circle;
\draw[radius=.08, fill=black](4,-1)circle;
\draw[radius=.08, fill=black](3,-2)circle;
\draw[radius=.08, fill=black](4,-2)circle;
\draw[radius=.08, fill=black](2,0)circle;
\draw[radius=.08, fill=black](5,0)circle;
\draw[radius=.08, fill=black](6,0)circle;
 
\node at (2,-2.5) {\textcolor{violet}{$0$}};
\node at (2.5,-1.5) {\textcolor{violet}{$1$}};
\node at (4.5,-1.5) {\textcolor{violet}{$1$}};
\node at (3.5,-.5) {\textcolor{violet}{$2$}};
\node at (3.5,-1.5) {\textcolor{violet}{$3$}};
 
    \end{tikzpicture}}
    \end{center}
    \caption{The base stranding for an untagged web}\label{fig: ex canonical stranding}
\end{figure}
\end{example}
 
\begin{theorem}\label{thm: every web has a stranding}
$S_0^G$ is a valid stranding.
\end{theorem}
\begin{proof}
By construction, simple strands in $S_0^G$ are either closed or have endpoints on the axis. If $A, B$ are adjacent faces of $G$, then $dist_n(U,A)\neq dist_n(U,B)$, so the simple strands enclosing $A$ and $B$ (if any) have different colors. Each edge of $G$ carries either one strand or two oppositely oriented strands, so the first validity condition of Definition~\ref{definition: combinatorial stranding} holds.
 
For an edge $e$ of weight $\ell$, let $A$ and $B$ be the faces to the left and right (with $e$ oriented ``up''); then $dist_n(U,B)=dist_n(U,A)+\ell \mod n$, and a direct calculation verifies the second validity condition.
\end{proof}
 
Fix $G\in F(\vec{k})$ and read the boundary faces of $G$ left to right as $U=A_0, A_1, \ldots, A_{m-1}, A_m=U$. (Faces may repeat if $G$ is disconnected.) Setting $c_j=dist_n(U,A_j)$, we have $c_0=0$, and since $A_0, A_1, \ldots, A_j$ is a path in $G^*$ from $U$ to $A_j$,
$$c_j=\sum_{t=1}^j k_t \mod n.$$
In particular, we have $\sum_{t=1}^m k_t\equiv 0\pmod n$. Combined with the surjectivity of $f$ (proved in Section~\ref{section:tagless vs tagged}), this recovers the following representation-theoretic fact (also part of Theorem~\ref{thm: hom space dimension}).
 
\begin{corollary}\label{cor: numerical invariant space condition}
For fixed $n$ and $\vec{k}$, $\textup{Inv}(\vec{k})$ is nontrivial if and only if $n$ divides $\sum_{t=1}^m k_t$.
\end{corollary}
 
The base stranding has an explicit associated monomial. Define
$$x_{S_0}=x_{\vec{b}_1}\otimes\cdots\otimes x_{\vec{b}_m},\qquad \vec{b}_j=\begin{cases}
\vec{\lambda}_{c_j}-\vec{\lambda}_{c_{j-1}} & \textup{if } c_{j-1}<c_j,\\
\vec{1}+\vec{\lambda}_{c_j}-\vec{\lambda}_{c_{j-1}} & \textup{if } c_j<c_{j-1},
\end{cases}$$
where $\vec{\lambda}_0=\vec{0}$ by convention. This depends only on $\vec{k}$, so $x_{S_0}$ is unambiguous across $ F(\vec{k})$.
 
\begin{lemma}\label{lem:canonical monomial}
For each $G\in F(\vec{k})$, the base stranding $S_0^G$ has monomial $x_{S_0^G}=x_{S_0}$.
\end{lemma}
 
Concretely, $x_{S_0}$ is a scalar multiple of the monomial formed from the standard basis vectors $x_1,\ldots,x_n$ in ascending order, repeated in blocks of length $n$ and grouped into tensor factors according to $\vec{k}$. The scalar arises because our basis lists wedge factors in descending order. As an example, for $n=3$ and $\vec{k}=(1,2,2,1)$,
$$x_{S_0}=q^{-2}\,x_1\otimes x_2\wedge_q x_3\otimes x_1\wedge_q x_2\otimes x_3.$$
 
Combining Theorems~\ref{thm: nonzero coefficient} and~\ref{thm: every web has a stranding}:
 
\begin{corollary}\label{cor:terms can't cancel}
For every $G\in F(\vec{k})$, $f(G)\neq 0$; in particular, $x_{S_0}$ has nonzero coefficient in $f(G)$.
\end{corollary}
 
\subsection{Basis webs from standard rectangular Young tableaux}\label{section: basis webs}
 
Recall from Theorem~\ref{thm: hom space dimension} the number of standard Young tableaux on an $n$-row rectangle is the dimension of the $\uq$-invariant space $\textup{Inv}(\vec{1})$. When $n=2$ and $n=3$ there are bijections from standard Young tableaux to particular bases of reduced web graphs using noncrossing matchings. This extends straightforwardly to a basis of $\mathfrak{sl}_n$ webs though the resulting graphs are not generally reduced in any obvious sense.
 
\subsubsection{Constructing stranded web graphs from standard Young tableaux}\label{section: constructing basis webs from tableaux}
 
Let $T$ be an $n$-row rectangular standard Young tableau. The following four-step process generalizes the algorithms from \cite{RTSpringerRep} and \cite{TSimpleBij} to build an $\mathfrak{sl}_n$ web $G_T$ together with a stranding $S_T$.
\begin{enumerate}
\item Place boundary vertices at $1, 2, \ldots, m$. Mark vertex $i$ with $\ell$ if entry $i$ appears in row $\ell$ of the tableau.
 
\item For each $\ell\in\{1,\ldots,n-1\}$, build a noncrossing matching on the vertices marked $\ell$ or $\ell+1$: pair $i$ (marked $\ell$) with $j$ (marked $\ell+1$) whenever $i<j$ and no unpaired vertex marked $\ell$ or $\ell+1$ lies strictly between them; draw the arc pairing $i$ and $j$ below the axis and color it $\ell$. This produces a \emph{multicolored noncrossing matching}. Arcs of a single color are noncrossing by construction, but arcs of different colors may cross.
 
\item Perturb the arcs so each intersection point lies on exactly two arcs and any two arcs meet at most once.
 
\item Resolve the matching into an $\mathfrak{sl}_n$ web graph $G_T$ stranded by $S_T$:
\begin{enumerate}
\item Direct each arc clockwise.
\item At each boundary vertex $i$ lying on two arcs, insert a trivalent interior vertex below $i$ with an edge into the boundary of weight $1$. The boundary edge carries simple strands with the colors and directions of the incident arcs, and we extend the strands through the new interior vertex as in Figure~\ref{figure: resolving webs}.
  \begin{figure}[h]
\begin{tikzpicture}[scale=.75]
\draw[style=dashed, <->] (1,0)--(3,0);
\draw[style=thick, blue] (2,0) to[out=270,in=0] (1,-1);
\draw[style=thick, red] (3,-1) to[out=180,in=270] (2,0);
\draw[style=dashed, <->] (4,0)--(6,0);
\begin{scope}[thick,decoration={
    markings,
    mark=at position 0.5 with {\arrow{>}}}
    ]
\draw[style=thick, postaction={decorate}] (5,-.75) -- (5,0);
\draw[postaction={decorate}, blue] (4.9,-.65) -- (4,-.9);
\draw[blue] (4.9,-.65) -- (4.9,0);
\draw[postaction={decorate}, red] (6,-.9) -- (5.1,-.65);
\draw[red] (5.1,-.65) -- (5.1,0);
\end{scope}
\draw[style=thick] (5,-.75) -- (6,-1);
\draw[style=thick] (5,-.75) -- (4,-1);
\draw[radius=.05, fill=black](5,0)circle;
\draw[radius=.05, fill=black](5,-.75)circle;
\node at (5.3,-.4) {\tiny $1$};
\end{tikzpicture}
\hspace{.5in}
\begin{tikzpicture}[scale=.75]
\draw[style=thick, blue] (2,0.25) to[out=250,in=40] (1,-1);
\draw[style=thick, red] (2,-1) to[out=110,in=320] (1,0.25);
\begin{scope}[thick,decoration={
    markings,
    mark=at position 0.5 with {\arrow{>}}}
    ]
\draw[style=thick, postaction={decorate}] (5,-.75) -- (5,0);
\draw[postaction={decorate}, blue] (4.9,-.65) -- (4,-.9);
\draw[postaction={decorate}, red] (4.9,-.1)--(4,0.15);
\draw[blue] (4.9,-.65) -- (5.1,-.1);
\draw[postaction={decorate}, red] (6,-.9) -- (5.1,-.65);
\draw[postaction={decorate}, blue] (6,0.15)--(5.1,-.1);
\draw[red] (5.1,-.65) -- (4.9,-.1);
\end{scope}
\draw[style=thick] (5,-.75) -- (6,-1);
\draw[style=thick] (5,-.75) -- (4,-1);
\draw[radius=.05, fill=black](5,0)circle;
\draw[radius=.05, fill=black](5,-.75)circle;
\draw[style=thick] (5,0) -- (6,0.25);
\draw[style=thick] (5,0) -- (4,0.25);
\node at (3,-.3) {\large $\rightsquigarrow$};
\node at (5.6,-.4) {\tiny ${\color{red} \ell'} - {\color{blue} \ell}$};
\end{tikzpicture}
    \caption{Resolving multicolored noncrossing matchings into web graphs, with new interior edge directed and labeled for the case $\ell' > \ell$} \label{figure: resolving webs}
        \end{figure}
 
\item Replace each crossing of arcs of colors $\ell$ and $\ell'$ with two interior vertices joined by an edge directed with the larger color, weighted and stranded as in Figure~\ref{figure: resolving webs} (case $\ell'>\ell$).
\item Each remaining arc segment becomes an edge of the web graph that is weighted, directed, and stranded according to its arc.
\end{enumerate}
\end{enumerate}
Figure~\ref{figure: example of tableau and web from tableau} gives an example.
\begin{remark}
The web $G_T$ is not unique as it depends on the perturbation choices made in Step (3). One convenient convention: draw each arc as a polygonal curve with one minimum, with left slope $-h$ and right slope $h$ for some fixed $h>0$.

For any $G_T$ constructed as above, Theorem \ref{thm: nonzero coefficient} guarantees the monomial $x_{S_T}$, which has $x_{\ell}$ in the $i^{th}$ factor exactly when $i$ is in row $\ell$ of $T$, has a nonzero coefficient in the web vector $f(G_T)$.
\end{remark}

\begin{figure}[h]
\begin{center} $\begin{array}{|c|c|c|} \cline{1-3} 1 & 3 & 8 \\ \cline{1-3} 2 & 6 & 11 \\ \cline{1-3} 4 & 7 & 12 \\ \cline{1-3} 5 & 10 & 13 \\ \cline{1-3} 9 & 14 & 15 \\ \hline \end{array}$ \hspace{0.2in}
\begin{tikzpicture}[xscale=1/2]
\draw[style=thick, blue] (1,0) to[out=280,in=180] (1.5,-.4);
\draw[style=thick, blue] (1.5,-.4) to[out=0,in=260] (2,0);
\draw[style=thick, blue] (3,0) to[out=280,in=180] (4.5,-.8);
\draw[style=thick, blue] (4.5,-.8) to[out=0,in=260] (6,0);
\draw[style=thick, blue] (8,0) to[out=280,in=180] (9.5,-.8);
\draw[style=thick, blue] (9.5,-.8) to[out=0,in=260] (11,0);
\draw[style=thick, red] (2,0) to[out=280,in=180] (3,-.6);
\draw[style=thick, red] (3,-.6) to[out=0,in=260] (4,0);
\draw[style=thick, red] (6,0) to[out=280,in=180] (6.5,-.4);
\draw[style=thick, red] (6.5,-.4) to[out=0,in=260] (7,0);
\draw[style=thick, red] (11,0) to[out=280,in=180] (11.5,-.4);
\draw[style=thick, red] (11.5,-.4) to[out=0,in=260] (12,0);
\draw[style=thick, green] (4,0) to[out=280,in=180] (4.5,-.4);
\draw[style=thick, green] (4.5,-.4) to[out=0,in=260] (5,0);
\draw[style=thick, green] (7,0) to[out=280,in=180] (8.5,-.9);
\draw[style=thick, green] (8.5,-.9) to[out=0,in=260] (10,0);
\draw[style=thick, green] (12,0) to[out=280,in=180] (12.5,-.4);
\draw[style=thick, green] (12.5,-.4) to[out=0,in=260] (13,0);
\draw[style=thick, orange] (5,0) to[out=290,in=180] (7,-.8);
\draw[style=thick, orange] (7,-.8) to[out=0,in=250] (9,0);
\draw[style=thick, orange] (10,0) to[out=300,in=180] (12.5,-1);
\draw[style=thick, orange] (12.5,-1) to[out=0,in=240] (15,0);
\draw[style=thick, orange] (13,0) to[out=280,in=180] (13.5,-.4);
\draw[style=thick, orange] (13.5,-.4) to[out=0,in=260] (14,0);
\end{tikzpicture}
\raisebox{-1in}{
\begin{tikzpicture}[xscale=1/2]
\draw[style=dashed, <->] (0,0)--(16,0);
\begin{scope}[thick,decoration={
    markings,
    mark=at position 0.5 with {\arrow{<}}}
    ]
\draw[thick, postaction={decorate}] (1,0) to[out=270,in=180] (2,-.5);
\draw[thick, postaction={decorate}] (2,0) -- (2,-.5);
\draw[thick, postaction={decorate}] (4,0) -- (4,-.5);
\draw[thick, postaction={decorate}] (5,0) -- (5,-.5);
\draw[thick, postaction={decorate}] (6,0) -- (6,-.5);
\draw[thick, postaction={decorate}] (7,0) -- (7,-.5);
\draw[thick, postaction={decorate}] (10,0) -- (10,-.5);
\draw[thick, postaction={decorate}] (11,0) -- (11,-.5);
\draw[thick, postaction={decorate}] (12,0) -- (12,-.5);
\draw[thick, postaction={decorate}] (13,0) -- (13,-.5);
\draw[thick, postaction={decorate}] (3,0) to[out=270,in=160] (3.5,-.5);
\draw[thick, postaction={decorate}] (8,0) to[out=270,in=180] (8.5,-.5);
\draw[thick, postaction={decorate}] (2,-.5) to[out=270,in=240] (3.5,-1);
\draw[thick, postaction={decorate}] (3.5,-1) to[out=270,in=270] (5.5,-1);
\draw[thick, postaction={decorate}] (5.5,-1) to[out=270,in=240] (7.5,-1.5);
\draw[thick, postaction={decorate}] (7.5,-1.5) to[out=270,in=270] (9.5,-1.5);
\draw[thick, postaction={decorate}] (9.5,-1.5) to[out=290,in=270] (10.5,-1);
\draw[thick, postaction={decorate}] (10.5,-1) to[out=270,in=270] (15,0);
\draw[thick, postaction={decorate}] (3.5,-.5) to[out=0,in=180] (4,-.5);
\draw[thick, postaction={decorate}] (4,-.5) to[out=270,in=270] (5,-.5);
\draw[thick] (5.5,-.5) -- (5.5,-1);
\draw[thick, postaction={decorate}] (5.5,-.75) -- (5.5,-1);
\draw[thick] (10.5,-.75) -- (10.5,-1);
\draw[thick, postaction={decorate}] (10.5,-.5) -- (10.5,-1);
\draw[thick, postaction={decorate}] (8.5,-.5) to[out=0,in=270] (9,0);
\draw[thick, postaction={decorate}] (13,-.5) to[out=0,in=270] (14,0);
\draw[thick, postaction={decorate}] (4,-.5) to[out=270,in=270] (5,-.5);
\draw[thick, postaction={decorate}] (5,-.5) to[out=0,in=180] (5.5,-.5);
\draw[thick, postaction={decorate}] (5.5,-.5) -- (6,-.5);
\draw[thick, postaction={decorate}] (6,-.5) to[out=270,in=270] (7,-.5);
\draw[thick, postaction={decorate}] (7,-.5) to[out=270,in=180] (7.5,-1);
\draw[thick, postaction={decorate}] (7.5,-1) to[out=0,in=180] (8.5,-1);
\draw[thick, postaction={decorate}] (8.5,-1) to[out=0,in=180] (9.5,-1);
\draw[thick, postaction={decorate}] (9.5,-1) to[out=0,in=270] (10,-.5);
\draw[thick, postaction={decorate}] (10.25,-.5) -- (10.5,-.5);
\draw[thick, postaction={decorate}] (10.75,-.5) -- (11,-.5);
\draw[thick, postaction={decorate}] (11,-.5) to[out=0,in=180] (12,-.5);
\draw[thick, postaction={decorate}] (12,-.5) to[out=0,in=180] (13,-.5);
\end{scope}
\draw[thick] (10,-.5) to[out=0,in=180] (13,-.5);
\begin{scope}[thick,decoration={
    markings,
    mark=at position 0.5 with {\arrow{>}}}
    ]
\draw[thick] (3.5,-.5) -- (3.5,-1);
\draw[thick, postaction={decorate}] (3.5,-.75) -- (3.5,-1);
\draw[thick] (7.5,-1) -- (7.5,-1.5);
\draw[thick, postaction={decorate}] (7.5,-1.25) -- (7.5,-1.5);
\draw[thick] (8.5,-.5) -- (8.5,-1);
\draw[thick, postaction={decorate}] (8.5,-.75) -- (8.5,-1);
\draw[thick] (9.5,-1) -- (9.5,-1.5);
\draw[thick, postaction={decorate}] (9.5,-1.25) -- (9.5,-1.5);
\end{scope}
\node at (.8,-.5) {\small $\color{blue} 1$};
\node at (1.8,-1) {\small $\color{red} 2$};
\node at (3.75,-1.6) {\small $\color{blue} 1$};
\node at (6.5,-2.1) {\small $\color{orange} 4$};
\node at (7.9,-2.2) {\small $\color{green} 3$};
\node at (10.35,-1.75) {\small $\color{blue} 1$};
\node at (14,-1.8) {\small $\color{orange} 4$};
\node at (9.55,-.65) {\tiny $\color{green} 3$};
\node at (8.9,-.75) {\tiny $\color{blue} 1$};
\node at (7.9,-.75) {\tiny $\color{orange} 4$};
\node at (7.05,-1.05) {\tiny $\color{green} 3$};
\node at (6.5,-1.05) {\tiny $\color{red} 2$};
\node at (4.5,-1.05) {\tiny $\color{green} 3$};
\node at (5.25,-.3) {\tiny $\color{orange} 4$};
\node at (5.65,-.3) {\tiny $\color{blue} 1$};
\node at (3.65,-.3) {\tiny $\color{red} 2$};
\node at (2.7,-.4) {\small $\color{blue} 1$};
\node at (7.7,-.3) {\small $\color{blue} 1$};
\node at (8.7,-.15) {\tiny $\color{orange} 4$};
\node at (10.25,-.3) {\tiny $\color{orange} 4$};
\node at (10.65,-.3) {\tiny $\color{blue} 1$};
\node at (11.5,-.8) {\small $\color{red} 2$};
\node at (12.5,-.8) {\small $\color{green} 3$};
\node at (13.5,-.7) {\small $\color{orange} 4$};
\end{tikzpicture}
}
\end{center}
\caption{A standard Young tableau, its multicolored noncrossing matching, and a web graph, with edge weights colored according to the arcs that created them (unlabeled boundary edges weighted $1$)} \label{figure: example of tableau and web from tableau}
\end{figure}
 
\subsubsection{Coherent webs and a result of Fontaine}
 
Following Section~\ref{section: base stranding}, let $G^*$ be the planar dual of $G$ with $U$ the vertex of the unbounded face. In what follows, we work with fundamental weights $\lambda_j$ rather than coset representatives $\vec{\lambda}_j\in\mathbb{Z}^n$. Dominant weights are nonnegative integral linear combinations of fundamental weights, and carry a standard partial order (see, e.g., \cite{STEM}). Fontaine defines coherence as follows.
 
\begin{definition}\label{definition: coherent}
Suppose that $P$ is a path in the planar dual $G^*$ to the web graph $G$ with edges $e_1, e_2, \ldots, e_j$ of weights $\ell_1, \ell_2, \ldots, \ell_j$. The \emph{weight of $P$} is the dominant weight $wt(P)=\sum_{i=1}^j \lambda_{\widetilde{\ell_i}}$, where $\widetilde{\ell_i}=\ell_i$ if $e_i$ is directed towards $e_{i+1}$ and $n-\ell_i$ otherwise. A web graph $G$ is \emph{coherent} if:
\begin{enumerate}
\item For each $A\in V(G^*)$, any two minimal-length paths $P_1, P_2$ from $U$ to $A$ in the dominant-weight poset satisfy $wt(P_1)=wt(P_2)$.
\item For each $A\in V(G^*)$, at least one minimal-length path in $G^*$ from $U$ passes through $A$ to the boundary.
\item For each directed edge $B\mapsto A$ in $G^*$ of weight $\ell$, $wt(P_A)-wt(P_B)$ lies in the Weyl orbit of $\lambda_\ell$ for any minimal-length paths $P_A, P_B$ from $U$ to $A, B$.
\end{enumerate}
\end{definition}
 
For any coherent web, we therefore have a well-defined function $wt(A)$ on $V(G^*)$ given by  the common weight of any minimal-length path from $U$ to $A$.
 
For a web $G_T$ constructed as above, all minimal-weight paths in $G_T^*$ are directed in the graph-theoretic sense (each $e_i$ directed from $e_{i-1}$ to $e_{i+1}$). We exploit this property to show $G_T$ is coherent.
 
\begin{lemma}\label{lemma: web graph from tableau is coherent}
Let $G_T$ and $S_T$ be the web graph and stranding constructed from a rectangular standard Young tableau $T$ as above. Then $G_T$ is a coherent web. Moreover, for each face $A$ of $G_T$, if the simple strands of $S_T$ enclosing $A$ are colored $c_{j_1}, c_{j_2}, \ldots, c_{j_i}$, then
\[wt(A) = \lambda_{c_{j_1}} + \lambda_{c_{j_2}} + \cdots + \lambda_{c_{j_i}}.\]
\end{lemma}
\begin{proof}
Let $A$ be an interior face enclosed by simple strands of colors $c_{j_1},\ldots,c_{j_i}$ (with multiplicity). By construction, every path from $U$ to $A$ in $G_T^*$ has edges of these weights with multiplicity, and at least one path uses exactly these. Hence the unique minimal weight of a path from $U$ to $A$ is $\lambda_{c_{j_1}}+\cdots+\lambda_{c_{j_i}}$, establishing the first two coherence conditions and the formula for $wt(A)$.
 
For the third condition, suppose $A$ and $B$ share an edge. If the edge has one simple strand in $S_T$, the formula above gives the result. If the edge has two simple strands --- with color $i$ enclosing $A$ and $j$ enclosing $B$ --- all other simple strands enclose both $A$ and $B$ or neither, so $wt(A)-\lambda_i=wt(B)-\lambda_j$. Assuming $j>i$, $\lambda_j-\lambda_i$ lies in the Weyl orbit of $\lambda_{j-i}$.
\end{proof}
 
Fontaine, Kamnitzer, and Kuperberg \cite{FKK} construct a map from webs to invariants in the classical ($q=1$) setting via the geometric Satake correspondence. For each web, their map agrees up to nonzero scalar with the specialization of $f$ at $q=1$. Thus, we can rephrase the following basis theorem due to Fontaine \cite[Corollary 2.9]{FON} as follows.
 
\begin{proposition}[Fontaine]\label{prop: Fontaine basis}
Let $\mathcal{G}\subset F(\vec{k})$ be a set of coherent $\mathfrak{sl}_n$ web graphs. For $G\in\mathcal{G}$ with boundary faces $U=A_0, A_1, \ldots, A_{m-1}, A_m=U$ read left to right, set $wt(G)=(wt(A_0), \ldots, wt(A_{m}))$. Suppose the set $\{wt(G):G\in\mathcal{G}\}$ is in bijection with the set of row-strict $n$-row rectangular Young tableaux with content $\{1^{k_1},\ldots,m^{k_m}\}$ via the following rule. For each $i$, let $R_i \subseteq \{1, \ldots, n\}$ denote the unique $k_i$-element subset with
\[wt(A_i)-wt(A_{i-1}) = \sum_{\ell \in R_i}(\lambda_{\ell}-\lambda_{\ell-1}),\]
where $\lambda_0=\lambda_n=0$ by convention. Then $T_G$ has $i$ in row $\ell$ iff $\ell\in R_i$. Under this bijection, $\{f(G):G\in\mathcal{G}\}$ is a basis of $\textup{Inv}(\vec{k})$ at $q=1$.
\end{proposition}
 
By Lemma~\ref{lemma: web graph from tableau is coherent}, the graphs $G_T$ are coherent. The vector $wt(G_T)$ is determined by the colors of simple strands in $S_T$ enclosing each boundary face. Moreover, the tableau Proposition~\ref{prop: Fontaine basis} assigns to $G_T$ is precisely $T$. Since the coefficients of our web vectors are Laurent polynomials in $q$ and the classical and quantum invariant spaces have the same dimension, the basis property at $q=1$ in Proposition~\ref{prop: Fontaine basis} lifts to a basis over $\mathbb{C}(q)$.
 
\begin{theorem}\label{theorem: using Fontaine to prove basis}
Let $\mathcal{G}$ be any set of web graphs constructed from the set of standard Young tableaux of a fixed rectangular shape as in Section~\ref{section: constructing basis webs from tableaux}. Then $\{f(G):G\in\mathcal{G}\}$ is a basis of $\textup{Inv}(\vec{1})$.
\end{theorem}

\begin{remark}
Elias \cite{Eli15} constructs a basis of light ladders for morphism spaces of
CKM webs over $\mathbb{Z}[q,q^{-1}]$, building each web inductively from a
minuscule Littelmann path. For $\vec{k}=\vec{1}$, these paths
correspond to rectangular standard Young tableaux. He remarks that
his basis very likely agrees with Fontaine's \cite[Remark 1.2]{Eli15}. The
construction of
Section~\ref{section: constructing basis webs from tableaux} is not inductive:
it produces the web $G_T$ directly from a multicolored noncrossing matching,
together with a stranding $S_T$ whose monomial records $T$. We do not know
whether the two constructions always produce the same vectors.
\end{remark}
 
In \cite{BBCMMMMRSW} we use stranding to prove a basis result that does not rely on coherence. The monomial basis of $V(\vec{k})$ for $\vec{k}=(k_1,\ldots,k_m)$ is in bijection with strictly row-increasing tableaux of shape $n\times\frac{\sum k_i}{n}$ and content $\{1^{k_1},\ldots,m^{k_m}\}$, via the map sending entry $i$ in row $\ell$ to $x_\ell$ in the $i^{th}$ tensor factor. Place $q$-wedge terms in descending order (which merely rescales by a nonzero element of $\mathbb{C}(q)$), and order the resulting monomials lexicographically. The \emph{lex leading term} of a web is its lexicographically smallest term, denoted $lt(G)$.
 
\begin{theorem}[Bo, Burns, Chen, Marsho, Martin, Mawn, Mohren, Russell, Sales, Wong]\label{theorem: web basis without coherence}
The lex leading term for every web corresponds to a row-strict tableau. Moreover, any set of webs $\mathcal{G}\subset F(\vec{k})$ for which $\{lt(G):G\in\mathcal{G}\}$ is in bijection with the row-strict tableaux of shape $n\times\frac{\sum k_i}{n}$ and content $\{1^{k_1},\ldots,m^{k_m}\}$ yields a basis of $\textup{Inv}(\vec{k})$.
\end{theorem}
 
The construction of Section~\ref{section: constructing basis webs from tableaux} also generalizes to webs with arbitrary $\vec{k}$ in \cite{BBCMMMMRSW}, and the resulting webs satisfy Theorem~\ref{theorem: web basis without coherence}. The lex leading term is realized by the stranding $S_T$. 
 
\subsection{Stranding and tableau combinatorics}\label{section: stranding tableau combo}
 
The tableau-to-web construction of the previous subsection produces both a web graph and a stranding from each standard Young tableau. Two further applications use stranding to refine this picture: the first identifies lex leading terms for $\mathfrak{sl}_3$ webs via a notion of depth in the dual graph, and the second shows that stranding intertwines evacuation of tableaux with reflection of web graphs.
 
\subsubsection{Leading terms and depth for $\mathfrak{sl}_3$}
 
Define the \emph{depth} of a face $A$ of a web graph $G$ to be the number of edges in a shortest undirected path from $U$ to $A$ in $G^*$, equivalently the minimum number of edges of $G$ crossed by a path in the plane from $U$ to $A$ that avoids vertices of $G$. Denote this by $d(A)$. The following result from \cite{BBCMMMMRSW} uses stranding and depth to explicitly compute $lt(G)$ for a large class of $\mathfrak{sl}_3$ webs, including but not limited to coherent webs.
 
\begin{theorem}[Bo, Burns, Chen, Marsho, Martin, Mawn, Mohren, Russell, Sales, Wong]\label{thm: sl3 leading term stranding}
Let $G$ be a web graph for $\mathfrak{sl}_3$ such that two distinct depths occur around each trivalent vertex. After edge flips if needed, assume every edge of $G$ has weight $1$. For each oriented edge $e$ of $G$ with face $A$ to the left and $B$ to the right, set:
\begin{itemize}
\item $b_S(e)=100$ if $d(A)<d(B)$,
\item $b_S(e)=010$ if $d(A)=d(B)$,
\item $b_S(e)=001$ if $d(A)>d(B)$.
\end{itemize}
Then $S$ is a valid stranding, and $x_S$ is the monomial in the lex leading term of $f(G)$.
\end{theorem}
 
Figure~\ref{figure: example of sl3 stranding by depth} gives an example with all edges weighted $1$, as in the theorem.

\begin{figure}[h]
\begin{tikzpicture}
    \draw[style=dashed, <->] (0,0)--(10,0);

\draw (1,0) -- (1,-1);
\draw (1,-1) -- (2,-1);
\draw (2,0) -- (2,-3);
\draw (2,-3) -- (9,-3);
\draw (9,-3) -- (9, 0);

\draw (3,0) -- (3,-3);
\draw (3,-1) -- (4,-1);

\draw (4,0) -- (4,-1);
\draw (5,0) -- (5,-1);
\draw (6,0) -- (6,-1);
\draw (7,0) -- (7,-1);
\draw (8,0) -- (8,-1);

\draw (4,-1) -- (4,-2);
\draw (7.5,-1) -- (7.5,-2);
\draw (6,-1) -- (6,-2);

\draw (4,-2) -- (7.5,-2);

\draw (3.5,-1) -- (3.5,-3);

\draw (5.75,-1) -- (5.75,-2);

\draw (5.5,-2) -- (5.5,-3);

\draw (5,-1) -- (6,-1);
\draw (7,-1) -- (8,-1);

\draw[blue, thick] (2.05,0) -- (2.05,-2.9);
\draw[blue, thick] (2.05,-2.9) -- (3.05,-2.9);
\draw[blue, thick] (3.05,-2.9) -- (3.05,-.9);
\draw[blue, thick] (3.05,-.9) -- (3.45,-.9);
\draw[blue, thick] (3.45,-.9) -- (3.45,-2.9);
\draw[blue, thick] (3.45,-2.9) -- (5.3,-2.9);
\draw[blue, thick] (5.3,-2.9) -- (5.3,-1.9);
\draw[blue, thick] (5.3,-1.9) -- (5.45,-1.9);
\draw[blue, thick] (5.45,-1.9) -- (5.45,-2.9);
\draw[blue, thick] (5.45,-2.9) -- (9.05,-2.9);
\draw[blue, thick] (9.05,-2.9) -- (9.05,0);

\draw[blue, thick] (4.05,0) -- (4.05,-1.9);
\draw[blue, thick] (4.05,-1.9) -- (5.05,-1.9);
\draw[blue, thick] (5.05,-1.9) -- (5.05,-0.9);
\draw[blue, thick] (5.05,-0.9) -- (5.8,-0.9);
\draw[blue, thick] (5.8,-0.9) -- (5.8,-1.9);
\draw[blue, thick] (5.8,-1.9) -- (5.95,-1.9);
\draw[blue, thick] (5.95,-1.9) -- (5.95,0);

\draw[blue, thick] (7.05,0) -- (7.05,-.9);
\draw[blue, thick] (7.05,-.9) -- (7.95,-.9);
\draw[blue, thick] (7.95,-.9) -- (7.95,0);

\draw[red, thick] (.95,0) -- (.95,-1.1);
\draw[red, thick] (.95,-1.1) -- (1.95,-1.1);
\draw[red, thick] (1.95,-1.1) -- (1.95,0);

\draw[red, thick] (2.95,0) -- (2.95,-3.1);
\draw[red, thick] (2.95,-3.1) -- (3.55,-3.1);
\draw[red, thick] (3.55,-3.1) -- (3.55,-1.1);
\draw[red, thick] (3.55,-1.1) -- (3.95,-1.1);
\draw[red, thick] (3.95,-1.1) -- (3.95,0);

\draw[red, thick] (4.9,0) -- (4.9,-2.1);
\draw[red, thick] (4.9,-2.1) -- (5.2,-2.1);
\draw[red, thick] (5.2,-2.1) -- (5.2,-3.1);
\draw[red, thick] (5.2,-3.1) -- (5.55,-3.1);
\draw[red, thick] (5.55,-3.1) -- (5.55,-2.1);
\draw[red, thick] (5.55,-2.1) -- (5.7,-2.1);
\draw[red, thick] (5.7,-2.1) -- (5.7,-1.1);
\draw[red, thick] (5.7,-1.1) -- (6.05,-1.1);
\draw[red, thick] (6.05,-1.1) -- (6.05,-2.1);
\draw[red, thick] (6.05,-2.1) -- (7.45,-2.1);
\draw[red, thick] (7.45,-2.1) -- (7.45,-1.1);
\draw[red, thick] (7.45,-1.1) -- (6.95,-1.1);
\draw[red, thick] (6.95,-1.1) -- (6.95,0);

\begin{scope}[decoration={
    markings,
    mark=at position 0.5 with {\arrow{<}}}
    ] 
   \draw[postaction={decorate}] (3,-1) -- (3,-3);
   \draw[postaction={decorate}] (2,-1) -- (2,-3);
   \draw[postaction={decorate}] (4,-1) -- (4,-2);
   \draw[postaction={decorate}] (5,-1) -- (5,-2);
   \draw[postaction={decorate}] (5.25,-2) -- (5.25,-3);
   \draw[postaction={decorate}] (6,-1) -- (6,-2);
   \draw[postaction={decorate}] (7.5,-1) -- (7.5,-2);
 
   \draw[postaction={decorate}] (3,-1) -- (3.5,-1);
  \draw[postaction={decorate}] (5,-1) -- (5.75,-1);

   \draw[postaction={decorate}] (5.35,-2) -- (5.5,-2);
   \draw[postaction={decorate}] (5.85,-2) -- (6,-2);

  \draw[postaction={decorate}] (3.5,-3) -- (5.25,-3);

\end{scope}

\begin{scope}[decoration={
    markings,
    mark=at position 0.5 with {\arrow{>}}}
    ] 

    \draw[postaction={decorate}] (1,0) -- (1,-1);
    \draw[postaction={decorate}] (2,0) -- (2,-1);
    \draw[postaction={decorate}] (3,0) -- (3,-1);
    \draw[postaction={decorate}] (4,0) -- (4,-1);
    \draw[postaction={decorate}] (5,0) -- (5,-1);
    \draw[postaction={decorate}] (6,0) -- (6,-1);
    \draw[postaction={decorate}] (7,0) -- (7,-1);
    \draw[postaction={decorate}] (8,0) -- (8,-1);
    \draw[postaction={decorate}] (9,0) -- (9,-3);

   \draw[postaction={decorate}] (3.5,-1) -- (3.5,-3);
  \draw[postaction={decorate}] (5.5,-2) -- (5.5,-3);
   \draw[postaction={decorate}] (5.75,-1) -- (5.75,-2);

   \draw[postaction={decorate}] (3.5,-1) -- (4,-1);
   \draw[postaction={decorate}] (5.85,-1) -- (6,-1);

   \draw[postaction={decorate}] (5.10,-2) -- (5.25,-2);
   \draw[postaction={decorate}] (5.6,-2) -- (5.75,-2);

   \draw[postaction={decorate}] (3,-3) -- (3.5,-3);
  \draw[postaction={decorate}] (5.35,-3) -- (5.5,-3);

\end{scope}

\node at (1.5,-.5) {\small \color{gray} $1$};
\node at (2.5,-.5) {\small \color{gray} $1$};
\node at (3.5,-.5) {\small \color{gray} $2$};
\node at (4.5,-.5) {\small \color{gray} $2$};
\node at (5.5,-.5) {\small \color{gray} $3$};
\node at (6.5,-.5) {\small \color{gray} $2$};
\node at (7.5,-.5) {\small \color{gray} $2$};
\node at (8.5,-.5) {\small \color{gray} $1$};

\node at (3.2,-2) {\small \color{gray} $1$};
\node at (4.5,-2.5) {\small \color{gray} $1$};
\node at (5.35,-1.5) {\small \color{gray} $2$};
\node at (6,-2.5) {\color{gray} $\hookleftarrow 1$};
\node at (6.5,-1.5) {\color{gray} $\hookleftarrow 2$};
    
\end{tikzpicture}
\caption{The leading term stranding for a non-reduced $\mathfrak{sl}_3$ web obtained from depth} \label{figure: example of sl3 stranding by depth}
\end{figure}

\subsubsection{Stranding and evacuation}
 
Strandings also analyze combinatorial operations on tableaux. \emph{Promotion} of a tableau $T$ erases the top-left entry and repeatedly slides in the smaller of the right- or below-neighbor of the empty box until no further slides are possible. For $n=2$ and $n=3$, the tableau-to-web bijection satisfies
\[G_{P(T)}=\rho(G_T),\]
where $P$ is promotion and $\rho$ is rotation of the web (identifying $\pm\infty$ on the boundary line). Comparing promotion of tableaux to rotation of webs has driven considerable work, with a rotation-invariant basis known for $n=4$ \cite{GaetzetalRotation}; the general case is an open problem.
 
\emph{Evacuation} iterates promotion: at each step, promote and freeze one entry. Evacuation is an involution \cite{Sch63, Hai92, MalReu94} and appears across the RSK algorithm, symmetric-group representation theory, Schubert calculus, and symmetric functions \cite{Sta09}.
 
Using strandings, we show \cite{CEST} that evacuation of tableaux corresponds to reflection of web graphs (after edge flips); see Figure~\ref{figure: evacuation and reflection}. This generalizes the $\mathfrak{sl}_2$ and $\mathfrak{sl}_3$ results of Patrias and Pechenik \cite{PatPech}.
 
\begin{theorem}[Adams Cowan, Eilfort, Seekamp, Tymoczko]
Let $T$ be an $n\times\frac{m}{n}$ tableau, $E(T)$ its evacuation, and $\theta$ the reflection across the line $x=\frac{m-1}{2}$. Let $G_{E(T)}$ be the web graph from $E(T)$ via Section~\ref{section: constructing basis webs from tableaux}. Then $\theta(G_{E(T)})$ is obtained from $T$ via the same construction with all edges flipped.
\end{theorem}

\begin{figure}[h]
\begin{center}
$\begin{array}{|c|c|c|} \cline{1-3} 1 & 2 & 7 \\ \cline{1-3} 3 & 6 & 11 \\ \cline{1-3} 4 & 9 & 12 \\ \cline{1-3} 5 & 10 & 14 \\ \cline{1-3} 8 & 13 & 15 \\ \hline \end{array}$ \quad
 \quad
\scalebox{.9}{\raisebox{-.7in}{\begin{tikzpicture}[scale=.9]
\draw[style=dashed, <->] (16,0)--(0,0);
\begin{scope}[thick,decoration={
    markings,
    mark=at position 0.5 with {\arrow{<}}}
    ] 
\draw[thick, postaction={decorate}] (15,0) to[out=270,in=0] (14,-.5);
\draw[thick, postaction={decorate}] (14,0) -- (14,-.5);
\draw[thick, postaction={decorate}] (12,0) -- (12,-.5);
\draw[thick, postaction={decorate}] (11,0) -- (11,-.5);
\draw[thick, postaction={decorate}] (10,0) -- (10,-.5);
\draw[thick, postaction={decorate}] (9,0) -- (9,-.5);
\draw[thick, postaction={decorate}] (6,0) -- (6,-.5);
\draw[thick, postaction={decorate}] (5,0) -- (5,-.5);
\draw[thick, postaction={decorate}] (4,0) -- (4,-.5);
\draw[thick, postaction={decorate}] (3,0) -- (3,-.5);
\draw[thick, postaction={decorate}] (13,0) to[out=270,in=20] (12.5,-.5);
\draw[thick, postaction={decorate}] (8,0) to[out=270,in=0] (7.5,-.5);
\draw[thick, postaction={decorate}] (14,-.5) to[out=270,in=300] (12.5,-1);
\draw[thick, postaction={decorate}] (12.5,-1) to[out=270,in=270] (10.5,-1);
\draw[thick, postaction={decorate}] (10.5,-1) to[out=270,in=300] (8.5,-1.5);
\draw[thick, postaction={decorate}] (8.5,-1.5) to[out=270,in=270] (6.5,-1.5);
\draw[thick, postaction={decorate}] (6.5,-1.5) to[out=250,in=270] (5.5,-1);
\draw[thick, postaction={decorate}] (5.5,-1) to[out=270,in=270] (1,0);
\draw[thick, postaction={decorate}] (12.5,-.5) to[out=180,in=0] (12,-.5);
\draw[thick, postaction={decorate}] (12,-.5) to[out=270,in=270] (11,-.5);
\draw[thick] (10.5,-.5) -- (10.5,-1);
\draw[thick, postaction={decorate}] (10.5,-.75) -- (10.5,-1);
\draw[thick] (5.5,-.75) -- (5.5,-1);
\draw[thick, postaction={decorate}] (5.5,-.5) -- (5.5,-1);
\draw[thick, postaction={decorate}] (7.5,-.5) to[out=180,in=270] (7,0);
\draw[thick, postaction={decorate}] (3,-.5) to[out=180,in=270] (2,0);
\draw[thick, postaction={decorate}] (12,-.5) to[out=270,in=270] (11,-.5);
\draw[thick, postaction={decorate}] (11,-.5) -- (10.5,-.5);
\draw[thick, postaction={decorate}] (10.5,-.5) -- (10,-.5);
\draw[thick, postaction={decorate}] (10,-.5) -- (9,-.5);
\draw[thick, postaction={decorate}] (9,-.5) to[out=270,in=0] (8.5,-1);
\draw[thick, postaction={decorate}] (8.5,-1) -- (7.5,-1);
\draw[thick, postaction={decorate}] (7.5,-1) -- (6.5,-1);
\draw[thick, postaction={decorate}] (6.5,-1) to[out=180,in=270] (6,-.5);
\draw[thick, postaction={decorate}] (5.75,-.5) -- (5.5,-.5);
\draw[thick, postaction={decorate}] (5.25,-.5) -- (5,-.5);
\draw[thick, postaction={decorate}] (5,-.5) -- (4,-.5);
\draw[thick, postaction={decorate}] (4,-.5) -- (3,-.5);
\end{scope}

\draw[thick] (6,-.5) to[out=180,in=0] (3,-.5);

\begin{scope}[thick,decoration={
    markings,
    mark=at position 0.5 with {\arrow{>}}}
    ] 
\draw[thick] (12.5,-.5) -- (12.5,-1);
\draw[thick, postaction={decorate}] (12.5,-.75) -- (12.5,-1);
\draw[thick] (8.5,-1) -- (8.5,-1.5);
\draw[thick, postaction={decorate}] (8.5,-1.25) -- (8.5,-1.5);
\draw[thick] (7.5,-.5) -- (7.5,-1);
\draw[thick, postaction={decorate}] (7.5,-.75) -- (7.5,-1);
\draw[thick] (6.5,-1) -- (6.5,-1.5);
\draw[thick, postaction={decorate}] (6.5,-1.25) -- (6.5,-1.5);
\end{scope}

\node at (15.2,-.5) {\small $\color{blue} 1$};
\node at (14.2,-1) {\small $\color{red} 2$};
\node at (12.25,-1.6) {\small $\color{blue} 1$};
\node at (9.5,-2.1) {\small $\color{orange} 4$};
\node at (8.1,-2.2) {\small $\color{green} 3$};
\node at (5.65,-1.75) {\small $\color{blue} 1$};
\node at (2,-1.8) {\small $\color{orange} 4$};

\node at (6.45,-.65) {\tiny $\color{green} 3$};
\node at (7.1,-.75) {\tiny $\color{blue} 1$};
\node at (8.1,-.75) {\tiny $\color{orange} 4$};
\node at (8.95,-1.05) {\tiny $\color{green} 3$};
\node at (9.5,-1.05) {\tiny $\color{red} 2$};
\node at (11.5,-1.05) {\tiny $\color{green} 3$};
\node at (10.75,-.3) {\tiny $\color{orange} 4$};
\node at (10.35,-.3) {\tiny $\color{blue} 1$};
\node at (12.35,-.3) {\tiny $\color{red} 2$};
\node at (13.3,-.4) {\small $\color{blue} 1$};
\node at (8.3,-.3) {\small $\color{blue} 1$};
\node at (7.3,-.15) {\tiny $\color{orange} 4$};

\node at (5.75,-.3) {\tiny $\color{orange} 4$};
\node at (5.35,-.3) {\tiny $\color{blue} 1$};
\node at (4.5,-.8) {\small $\color{red} 2$};
\node at (3.5,-.8) {\small $\color{green} 3$};
\node at (2.5,-.7) {\small $\color{orange} 4$};
\end{tikzpicture}}}
\end{center}
\caption{The evacuation of the tableau in Figure~\ref{figure: example of tableau and web from tableau} and its associated web graph with all edges flipped} \label{figure: evacuation and reflection}
\end{figure}
 
\subsection{Webs and Springer fibers}\label{subsec:springer}
 
Springer fibers are subvarieties of the flag variety that play a fundamental role in the representation theory of the symmetric group $S_m$.
 
The flag variety of type $A_{m-1}$ is $GL_m(\mathbb{C})/B$, where $B$ is the Borel subgroup of upper-triangular invertible matrices. Each coset $gB$ has a canonical representative $g$ satisfying:
\begin{itemize}
\item the lowest nonzero entry in each column is $1$ (a \emph{pivot}), and
\item every entry to the right of a pivot in its row is $0$.
\end{itemize}
The pivots of $g$ form a permutation matrix $\sigma$, and the \emph{Schubert cell} $\mathcal{C}_\sigma$ is the set of representatives with pivot pattern $\sigma$; it is a complete set of representatives for $B\sigma B/B$.
 
The \emph{Springer fiber} of an $m\times m$ matrix $X$ is
\[\mathcal{S}_X=\{gB: g^{-1}Xg \text{ is upper triangular}\},\]
informally the flags ``fixed'' by $X$. The cohomology $H^*(\mathcal{S}_X)$ carries a graded $S_m$-representation, with $H^{top}(\mathcal{S}_X)$ irreducible and corresponding to the partition $\lambda$ of $m$ given by the Jordan-block structure of $X$. Varying $X$ over its conjugacy class recovers each irreducible $S_m$-representation; see \cite{Tym17} for a survey.
 
The intersections $\mathcal{S}_X\cap\mathcal{C}_\sigma$ are generally complicated, even for top-dimensional cells, and describing them carries information about the corresponding representation \cite{Fung, Shi80, Spa76, Tym17}. Since $\mathcal{S}_X\cong\mathcal{S}_{g^{-1}Xg}$, one can choose $X$ in its conjugacy class so that the intersections are well-behaved --- for instance, forming an affine paving \cite{Pre13, Tym06}. Stranded web graphs turn out to be particularly useful here.
 
In \cite{HLTT}, we construct a family of stranded $\mathfrak{sl}_3$ web graphs called \emph{Springer web graphs}. Built from standard Young tableaux but generally non-reduced (some interior faces are $4$-cycles), they arise from an algorithm that telescopically inserts squares into a reduced web graph to eliminate instances of a red strand passing below a blue strand. Each square moves the blue strand down and the red strand up. For $i=1,\ldots,m/3$, we assign free variables $a_i$ to the $i^{th}$ blue strand, $b_i$ to the $i^{th}$ red strand, and $c_i$ to the boundary face just right of the $i^{th}$ blue strand's start. We then assign a polynomial $p(A)$ in these variables to every other bounded face, in such a way that two faces $A, B$ separated only by edges with two strands have $p(A)-p(B)$ depending only on the $a_i, b_i$. Our main result is the following; see \cite{HLTT} for details.
 
\begin{theorem}[Hafken, Lang, Tashman, Tymoczko]
Let $X$ be the Jordan form of a nilpotent matrix with three Jordan blocks each of dimension $\frac{m}{3}$, and let $\mathcal{S}_X$ be the associated Springer fiber. For each standard Young tableau $T$ of shape $3\times\frac{m}{3}$, let $\sigma(T)$ be the permutation matrix whose $i^{th}$ column is $\vec{e}_{m-\frac{jm}{3}+k}$ if $i$ lies in entry $(j,k)$ of $T$. Then the nonzero entries of matrices in $\mathcal{S}_X\cap\mathcal{C}_{\sigma(T)}$ are computed from paths through the Springer web graph, recording variables on blue arcs in one block of the matrix, red arcs in another, and faces in a third.
\end{theorem}
 
\begin{figure}[h]
\begin{tikzpicture}
    \draw[style=dashed, <->] (0,0)--(10,0);

\draw (1,0) -- (1,-1);
\draw (1,-1) -- (2,-1);
\draw (2,0) -- (2,-3);
\draw (2,-3) -- (9,-3);
\draw (9,-3) -- (9, 0);

\draw (3,0) -- (3,-3);
\draw (3,-1) -- (4,-1);

\draw (4,0) -- (4,-1);
\draw (5,0) -- (5,-1);
\draw (6,0) -- (6,-1);
\draw (7,0) -- (7,-1);
\draw (8,0) -- (8,-1);

\draw (4,-1) -- (4,-2);
\draw (7.5,-1) -- (7.5,-2);
\draw (6,-1) -- (6,-2);

\draw (4,-2) -- (7.5,-2);

\draw (3.5,-1) -- (3.5,-3);

\draw (5.75,-1) -- (5.75,-2);

\draw (5.5,-2) -- (5.5,-3);

\draw (5,-1) -- (6,-1);
\draw (7,-1) -- (8,-1);

\draw[red, thick] (2.05,0) -- (2.05,-2.9);
\draw[red, thick] (2.05,-2.9) -- (3.05,-2.9);
\draw[red, thick] (3.05,-2.9) -- (3.05,-.9);
\draw[red, thick] (3.05,-.9) -- (3.45,-.9);
\draw[red, thick] (3.45,-.9) -- (3.45,-2.9);
\draw[red, thick] (3.45,-2.9) -- (5.3,-2.9);
\draw[red, thick] (5.3,-2.9) -- (5.3,-1.9);
\draw[red, thick] (5.3,-1.9) -- (5.45,-1.9);
\draw[red, thick] (5.45,-1.9) -- (5.45,-2.9);
\draw[red, thick] (5.45,-2.9) -- (9.05,-2.9);
\draw[red, thick] (9.05,-2.9) -- (9.05,0);

\draw[red, thick] (4.05,0) -- (4.05,-1.9);
\draw[red, thick] (4.05,-1.9) -- (5.05,-1.9);
\draw[red, thick] (5.05,-1.9) -- (5.05,-0.9);
\draw[red, thick] (5.05,-0.9) -- (5.8,-0.9);
\draw[red, thick] (5.8,-0.9) -- (5.8,-1.9);
\draw[red, thick] (5.8,-1.9) -- (5.95,-1.9);
\draw[red, thick] (5.95,-1.9) -- (5.95,0);

\draw[red, thick] (7.05,0) -- (7.05,-.9);
\draw[red, thick] (7.05,-.9) -- (7.95,-.9);
\draw[red, thick] (7.95,-.9) -- (7.95,0);

\draw[blue, thick] (.95,0) -- (.95,-1.1);
\draw[blue, thick] (.95,-1.1) -- (1.95,-1.1);
\draw[blue, thick] (1.95,-1.1) -- (1.95,0);

\draw[blue, thick] (2.95,0) -- (2.95,-3.1);
\draw[blue, thick] (2.95,-3.1) -- (3.55,-3.1);
\draw[blue, thick] (3.55,-3.1) -- (3.55,-1.1);
\draw[blue, thick] (3.55,-1.1) -- (3.95,-1.1);
\draw[blue, thick] (3.95,-1.1) -- (3.95,0);

\draw[blue, thick] (4.9,0) -- (4.9,-2.1);
\draw[blue, thick] (4.9,-2.1) -- (5.2,-2.1);
\draw[blue, thick] (5.2,-2.1) -- (5.2,-3.1);
\draw[blue, thick] (5.2,-3.1) -- (5.55,-3.1);
\draw[blue, thick] (5.55,-3.1) -- (5.55,-2.1);
\draw[blue, thick] (5.55,-2.1) -- (5.7,-2.1);
\draw[blue, thick] (5.7,-2.1) -- (5.7,-1.1);
\draw[blue, thick] (5.7,-1.1) -- (6.05,-1.1);
\draw[blue, thick] (6.05,-1.1) -- (6.05,-2.1);
\draw[blue, thick] (6.05,-2.1) -- (7.45,-2.1);
\draw[blue, thick] (7.45,-2.1) -- (7.45,-1.1);
\draw[blue, thick] (7.45,-1.1) -- (6.95,-1.1);
\draw[blue, thick] (6.95,-1.1) -- (6.95,0);

\begin{scope}[decoration={
    markings,
    mark=at position 0.5 with {\arrow{<}}}
    ] 
   \draw[postaction={decorate}] (3,-1) -- (3,-3);
   \draw[postaction={decorate}] (2,-1) -- (2,-3);
   \draw[postaction={decorate}] (4,-1) -- (4,-2);
   \draw[postaction={decorate}] (5,-1) -- (5,-2);
   \draw[postaction={decorate}] (5.25,-2) -- (5.25,-3);
   \draw[postaction={decorate}] (6,-1) -- (6,-2);
   \draw[postaction={decorate}] (7.5,-1) -- (7.5,-2);
 
   \draw[postaction={decorate}] (3,-1) -- (3.5,-1);
  \draw[postaction={decorate}] (5,-1) -- (5.75,-1);

   \draw[postaction={decorate}] (5.35,-2) -- (5.5,-2);
   \draw[postaction={decorate}] (5.85,-2) -- (6,-2);

  \draw[postaction={decorate}] (3.5,-3) -- (5.25,-3);

\end{scope}

\begin{scope}[decoration={
    markings,
    mark=at position 0.5 with {\arrow{>}}}
    ] 

    \draw[postaction={decorate}] (1,0) -- (1,-1);
    \draw[postaction={decorate}] (2,0) -- (2,-1);
    \draw[postaction={decorate}] (3,0) -- (3,-1);
    \draw[postaction={decorate}] (4,0) -- (4,-1);
    \draw[postaction={decorate}] (5,0) -- (5,-1);
    \draw[postaction={decorate}] (6,0) -- (6,-1);
    \draw[postaction={decorate}] (7,0) -- (7,-1);
    \draw[postaction={decorate}] (8,0) -- (8,-1);
    \draw[postaction={decorate}] (9,0) -- (9,-3);

   \draw[postaction={decorate}] (3.5,-1) -- (3.5,-3);
  \draw[postaction={decorate}] (5.5,-2) -- (5.5,-3);
   \draw[postaction={decorate}] (5.75,-1) -- (5.75,-2);

   \draw[postaction={decorate}] (3.5,-1) -- (4,-1);
   \draw[postaction={decorate}] (5.85,-1) -- (6,-1);

   \draw[postaction={decorate}] (5.10,-2) -- (5.25,-2);
   \draw[postaction={decorate}] (5.6,-2) -- (5.75,-2);

   \draw[postaction={decorate}] (3,-3) -- (3.5,-3);
  \draw[postaction={decorate}] (5.35,-3) -- (5.5,-3);

\end{scope}

\node at (0.75,-0.75) {\small \color{blue} $a_1$};
\node at (2.75,-0.5) {\small \color{blue} $a_2$};
\node at (7.8,-2.2) {\small \color{blue} $a_3$};

    \node at (1.75,-2.3) {\small \color{red} $b_1$};
\node at (4.3,-0.3) {\small \color{red} $b_2$};
\node at (7.7,-.6) {\small \color{red} $b_3$};

\node at (1.5,-.5) {\small \color{violet} $c_1$};
\node at (3.3,-2.35) {\small \color{violet} $p_1$};
\node at (3.5,-.5) {\small \color{violet} $c_2$};
\node at (5.5,-.5) {\small \color{violet} $c_3$};
\node at (5.4,-1.4) {\small \color{violet} $p_3$};
\node at (5.8,-2.3) {\small \color{violet} $\longleftarrow p_2$};

    \node at (-.45, -1.6) {${\color{violet} p_1} = {\small {\color{violet} c_1} - {\color{red} b_1}({\color{blue} a_2-a_1})}$};
      \node at (-.45, -2.1) {${\color{violet} p_2} =  {\small {\color{violet} c_1} - {\color{red} b_1}({\color{blue} a_3-a_1})}$};
          \node at (-.45, -2.6) {${\color{violet} p_3} = {\small {\color{violet} c_2} - {\color{red} b_2}({\color{blue} a_3-a_2})} $};
\end{tikzpicture} \hspace{0.25in}
\raisebox{0.6in}{\scalebox{0.8}{$\left(\begin{array}{ccccccccc} {\color{violet} c_1} & {\color{red} b_1} & {\color{violet} c_2} & {\color{red} b_2} &  {\color{violet} c_3} & 1 & 0 & 0 & 0 \\
0 & 0 &{\color{violet} p_1} & {\color{red} b_1} & {\color{violet} p_3}&0 & {\color{red} b_3} & 1&0 \\ 
0 & 0 & 0 & 0 & {{\color{violet} p_2}} &0 & {\color{red} b_1} & 0&1  \\ 
{\color{blue} a_1} & 1 & 0&0&0&0&0&0&0 \\ 
0 & 0 & {\color{blue} a_2} &1&0&0&0&0&0 \\ 
0 & 0 & 0 &0&{\color{blue} a_3}&0&1&0&0 \\ 
1 & 0 & 0 &0&0&0&0&0&0 \\
0 & 0 & 1 &0&0&0&0&0&0 \\
0 & 0 & 0 &0&1&0&0&0&0 \\
\end{array}\right)$}}
\caption{Example of a Springer web graph and Springer Schubert cell, with all arc labels shown in blue and red, and some face labels shown in violet} \label{figure: example of Springer web graph and Schubert cell}
\end{figure}

This generalizes the $n=2$ result of \cite{GNST}; in that case, complete information about all Springer Schubert cells (and their closures) is also available.

 \section{Tagged webs and relations}\label{section:tagless vs tagged}
 
We now relate our untagged web framework to the work of CKM \cite{CKM} on tagged webs. Their construction provides a surjective $\uq$-equivariant map $g:\mathcal{C}(\vec{k})\to\textup{Inv}(\vec{k})$ together with a generating set of relations for $\ker(g)$. We will give an explicit translation between untagged and tagged web graphs under which $f$ and $g$ produce the same invariant vector up to sign, and use this comparison to verify the surjectivity of $f$ and to transfer the CKM relations to untagged webs.   Fraser, Lam, and Le \cite{FraserLamLe} give relations and a canonical sign for the untagged model at $q=1$; this section is the quantum version.
 
\subsection{Tagged webs}\label{section: tagged webs}
 
We recall the tagged web model of CKM \cite{CKM}; see also \cite{MorrisonThesis}. The terminology \emph{tagged web} distinguishes these from the untagged webs that are the focus of this paper.
 
\begin{definition}\label{definition: tagged web}
A \emph{tagged web} is an $\mathfrak{sl}_n$ web graph whose interior vertices are either bivalent or trivalent. At each trivalent vertex $v$ with incident edges $e_1, e_2, e_3$ of weights $\ell_1, \ell_2, \ell_3$, edge weights satisfy the following \emph{integral conservation of flow} condition: 
$$\sum_{i=1}^3 \sigma_v(e_i)\ell_i = 0.$$
At each bivalent vertex $v$ with incident edges $e_1, e_2$ of weights $\ell_1, \ell_2$,
$$\ell_1+\ell_2=n \qquad\textup{and}\qquad \sigma_v(e_1)=\sigma_v(e_2).$$
Bivalent vertices carry \emph{tags}, drawn as short segments indicating one side of the edge.
 
For a tagged web with boundary vertices $v_1,\ldots,v_m$ (left to right) and corresponding boundary edges $e_i$ of weight $\ell_i$, the \emph{boundary weight vector} is $\vec{k}=(k_1,\ldots,k_m)$ where
$$k_i=\begin{cases}\ell_i & \textup{if } \sigma_{v_i}(e_i)=1,\\ -\ell_i & \textup{if } \sigma_{v_i}(e_i)=-1.\end{cases}$$
Let $C(\vec{k})$ be the set of tagged $\mathfrak{sl}_n$ webs with boundary weight vector $\vec{k}$, and $\mathcal{C}(\vec{k})$ the free $\mathbb{C}(q)$-vector space on $C(\vec{k})$.
\end{definition}

Note the sign convention for $\vec{k}$ here differs from the untagged case (Definition~\ref{definition: untagged web}). The CKM convention distinguishes $V^k$ from its isomorphic dual $V_{n-k}^*$ via signed boundary weights, while our untagged convention implicitly incorporates this isomorphism. 
 
For tagged webs, the integral conservation-of-flow condition at trivalent vertices forces at least one edge directed into the vertex and at least one directed out. As with untagged webs, we call $v$ a \emph{split vertex} when exactly one incident edge is directed inward and a \emph{merge vertex} when exactly two are.
 
\begin{example}\label{ex:tagged web}
Let $n=4$ and $\vec{k}=(1,1,3,3)$. The tagged web $H\in C(\vec{k})$ in Figure~\ref{fig: ex tagged web} has the same underlying graph as the untagged web of Example~\ref{ex:untagged web}. As in that example, the two left interior vertices are split vertices and the two right interior vertices are merge vertices.
 
\begin{figure}[h]
 \begin{center}
\raisebox{3pt}{\begin{tikzpicture}[scale=1]
 
\draw[style=dashed, <->] (0,0)--(7,0);
 
\begin{scope}[thick,decoration={
    markings,
    mark=at position 0.5 with {\arrow{<}}}
    ]
\draw[postaction={decorate},style=thick] (1,0) to[out=270,in=180] (3,-2);
\draw[postaction={decorate},style=thick] (3,-2)--(3,-1);
\draw[postaction={decorate},style=thick] (5,0) to[out=270,in=0] (4,-1);
\draw[postaction={decorate},style=thick] (6,0) to[out=270,in=0] (4,-2);
\draw[postaction={decorate},style=thick] (4,-2)--(3,-2);
\end{scope}
 
\begin{scope}[thick,decoration={
    markings,
    mark=at position 0.5 with {\arrow{>}}}
    ]
\draw[postaction={decorate},style=thick] (3.5,-1)--(3,-1);
\draw[postaction={decorate},style=thick] (3.5,-1)--(4,-1);
\draw[postaction={decorate},style=thick] (3,-1) to[out=180,in=270] (2,0);
\draw[postaction={decorate},style=thick] (4,-1.5)--(4,-1);
\draw[postaction={decorate},style=thick] (4,-1.5)--(4,-2);
\end{scope}
 
\draw[radius=.05, fill=black](1,0)circle;
\draw[radius=.05, fill=black](3,-1)circle;
\draw[radius=.05, fill=black](4,-1)circle;
\draw[radius=.05, fill=black](3,-2)circle;
\draw[radius=.05, fill=black](4,-2)circle;
\draw[radius=.05, fill=black](2,0)circle;
\draw[radius=.05, fill=black](5,0)circle;
\draw[radius=.05, fill=black](6,0)circle;
 
\draw[thick] (3.5,-1)--(3.5,-.75);
\draw[thick] (4,-1.5)--(3.75,-1.5);
 
\node at (.75,-.25) {\tiny{$1$}};
\node at (1.75,-.25) {\tiny{$1$}};
\node at (3.25,-.75) {\tiny{$3$}};
\node at (3.75,-.75) {\tiny{$1$}};
\node at (3.5,-2.25) {\tiny{$1$}};
\node at (2.75, -1.5) {\tiny{$2$}};
\node at (4.25, -1.25) {\tiny{$2$}};
\node at (4.25, -1.75) {\tiny{$2$}};
\node at (6.25, -.25) {\tiny{$3$}};
\node at (5.25, -.25) {\tiny{$3$}};
    \end{tikzpicture}}
    \end{center}
    \caption{A tagged $\mathfrak{sl}_4$ web graph}\label{fig: ex tagged web}
\end{figure}
\end{example}
 
\begin{remark}\label{rem:ignore bivalent}
We often ignore the bivalent vertices of a tagged web, treating each maximal chain of edges through bivalent vertices as a single edge equipped with a sequence of tags. The orientation and weight alternate between $\stackrel{\ell}{\leftarrow}$ and $\stackrel{n-\ell}{\longrightarrow}$ at each tag.
\end{remark}
 
CKM construct a $\uq$-equivariant map $g:\mathcal{C}(\vec{k})\to\textup{Inv}(\vec{k})$ via a Morse-style decomposition of tagged webs into local pieces (cups, caps, $\lambda$-pieces, Y-pieces). The full construction of $g$, including the explicit formulas for the local maps, is recorded in Appendix~\ref{appendix:tagged ckm background}. For the purposes of this section, we use only the following two facts from \cite{CKM}: $g$ is surjective, and $\ker(g)$ is generated by an explicit finite set of relations.
 
\subsection{Relating untagged and tagged web graphs}\label{section: untagged tagged relation}
 
Given a tripod decomposition $G'$ of an untagged web $G$, the proof of Theorem~\ref{thm:invariant vector flow formula} expresses $f(G)=f(G')$ as a composition of cap maps applied to the tensor product of invariants coming from the cups and tripods in $G'$. We now construct a parallel translation to tagged webs: for each cup, cap, and tripod, there is a tagged web fragment whose CKM vector equals the corresponding piece of $f$ up to sign. Stitching these fragments together and tracking the sign, we obtain a global comparison between $f$ and $g$.
 
Each cup, cap, and tripod of $G'$ is equivalent under edge flips to exactly one of the local pieces in the first column of Figure~\ref{fig:replaceCKM}. The corresponding tagged web fragments and signs appear in the second and third columns. See Appendix~\ref{appendix:tagged ckm background} for the composition of CKM maps corresponding to each fragment.

 \begin{figure}[h]
\scalebox{.85}{\begin{tabular}{|c|c|c|}
\hline
&& \\
$G'$ & $H'$ & $sgn(G')$   \\
&& \\ \hline
    &&\\
\raisebox{3pt}{\begin{tikzpicture}[scale=.9]
        \draw[dashed, red, <->] (0,0)--(3,0);
    \node at(.8,-.65) {\tiny{$k$}};
        \begin{scope}[thick,decoration={
    markings,
    mark=at position 0.5 with {\arrow{>}}}
    ] 
\draw[postaction={decorate}, thick] (1,0) to[out=270,in=180] (1.5,-1) to[out=0,in=270] (2,0);
    \end{scope}
    \end{tikzpicture}} &
\begin{tikzpicture}[scale=.9]
        \draw[dashed, red, <->] (0,0)--(3,0);
    \node at(.5,-.65) {\tiny{$n-k$}};
    \node at(2.2,-.65) {\tiny{$k$}};
    
      \draw[style=thick] (1.5,-1)--(1.5,-1.25);
        \begin{scope}[thick,decoration={
    markings,
    mark=at position 0.5 with {\arrow{>}}}
    ] 
\draw[postaction={decorate}, thick]  (1.5,-1) to[out=0,in=270] (2,0);
\draw[postaction={decorate}, thick] (1.5,-1) to[out=180,in=270] (1,0);
    \end{scope}
    \end{tikzpicture} & \raisebox{10pt}{1} \\
        && \\ \hline
&&\\
\raisebox{-10pt}{\scalebox{.9}{\begin{tikzpicture}[scale=1]
        \draw[dashed, red, <->] (0,0)--(4,0);
        \draw[radius=.08, fill=black](2,-1)circle;
         \begin{scope}[thick,decoration={
    markings,
    mark=at position 0.5 with {\arrow{>}}}
    ] 
\draw[postaction={decorate}, thick] (2,-1) to[out=180,in=270] (1,0);
\draw[postaction={decorate}, thick] (2,-1)--(2,0);
\draw[postaction={decorate}, thick] (2,-1) to[out=0,in=270](3,0);
    \end{scope}
    \node at (.9,-.5) {\tiny{$k$}};
     \node at (1.7,-.5) {\tiny{$l$}};
     \node at (3.2,-.5) {\tiny{$m$}};
    \end{tikzpicture}}} &
\raisebox{-20pt}{\scalebox{.9}{\begin{tikzpicture}[scale=1]
        \draw[dashed, red, <->] (0,0)--(4,0);
        \draw[radius=.08, fill=black](2,-1)circle;
\begin{scope}[thick,decoration={
    markings,
    mark=at position 0.5 with {\arrow{>}}}
    ] 
\draw[postaction={decorate}, thick] (2,-1) to[out=180,in=270] (1,0);
\draw[postaction={decorate}, thick] (2,-1)--(2,0);
\end{scope}
\begin{scope}[thick,decoration={
    markings,
    mark=at position 0.33 with {\arrow{<}}}
    ] 
\draw[postaction={decorate}, thick] (2,-1) to[out=0,in=270](3,0);
    \end{scope}
    \begin{scope}[thick,decoration={
    markings,
    mark=at position 0.7 with {\arrow{>}}}
    ] 
\draw[postaction={decorate}, thick] (2,-1) to[out=0,in=270](3,0);
    \end{scope}
    \node at (.9,-.5) {\tiny{$k$}};
     \node at (1.7,-.5) {\tiny{$l$}};
     \node at (2.4,-1.25) {\tiny{$n-m$}};
       \node at (3.2,-.5) {\tiny{$m$}};
 \draw[style=thick] (2.65,-.75)--(2.8,-.9);
    \end{tikzpicture}}} & $(-1)^{kl}$ \\
    ($k+l+m=n$)&& \\ \hline
    &&\\
\scalebox{.9}{\raisebox{11pt}{\begin{tikzpicture}[scale=1]
        \draw[dashed, red, <->] (0,0)--(4,0);
        \draw[radius=.08, fill=black](2,-1)circle;
         \begin{scope}[thick,decoration={
    markings,
    mark=at position 0.5 with {\arrow{>}}}
    ] 
\draw[postaction={decorate}, thick] (2,-1) to[out=180,in=270] (1,0);
\draw[postaction={decorate}, thick] (2,-1)--(2,0);
\draw[postaction={decorate}, thick] (2,-1) to[out=0,in=270](3,0);
    \end{scope}
    \node at (.6,-.3) {\tiny{$k$}};
     \node at (1.6,-.3) {\tiny{$l$}};
     \node at (3.3,-.3) {\tiny{$m$}};
    \end{tikzpicture}}} &
\scalebox{.9}{\begin{tikzpicture}[scale=1]
        \draw[dashed, red, <->] (0,0)--(4,0);
        \draw[radius=.08, fill=black](2,-1)circle;
\begin{scope}[thick,decoration={
    markings,
    mark=at position 0.5 with {\arrow{>}}}
    ] 
\draw[postaction={decorate}, thick] (2,-1) to[out=180,in=270] (1,0);
\end{scope}
\begin{scope}[thick,decoration={
    markings,
    mark=at position 0.3 with {\arrow{<}}}
    ] 
\draw[postaction={decorate}, thick] (2,-1) to[out=0,in=270](3,0);
\draw[postaction={decorate}, thick] (2,-1)--(2,0);
    \end{scope}
    \begin{scope}[thick,decoration={
    markings,
    mark=at position 0.7 with {\arrow{>}}}
    ] 
\draw[postaction={decorate}, thick] (2,-1) to[out=0,in=270](3,0);
\draw[postaction={decorate}, thick] (2,-1)--(2,0);
    \end{scope}
    \node at (.6,-.3) {\tiny{$k$}};
     \node at (1.55,-.25) {\tiny{$l$}};
     \node at (3.3,-.3) {\tiny{$m$}};
          \node at (1.55,-.55) {\tiny{$n-l$}};
     \node at (2.4,-1.25) {\tiny{$n-m$}};
 \draw[style=thick] (2.65,-.75)--(2.8,-.9);
  \draw[style=thick] (2,-.55)--(2.2,-.55);
    \end{tikzpicture}}& \raisebox{10pt}{1} \\
    ($k+l+m=2n$)&& \\ \hline
    &&\\
\begin{tikzpicture}[scale=.9]
        \draw[dashed, red, <->] (0,0)--(3,0);
        \begin{scope}[thick,decoration={
    markings,
    mark=at position 0.5 with {\arrow{>}}}
    ] 
\draw[postaction={decorate}, thick] (1,0) to[out=90,in=180] (1.5,1) to[out=0,in=90] (2,0);
    \end{scope}
    \node at(.8,.65) {\tiny{$k$}};
    \end{tikzpicture} &
\begin{tikzpicture}[scale=.9]
        \draw[dashed, red, <->] (0,0)--(3,0);
        \begin{scope}[thick,decoration={
    markings,
    mark=at position 0.5 with {\arrow{>}}}
    ] 
\draw[postaction={decorate}, thick] (1,0) to[out=90,in=180] (1.5,1);
\draw[postaction={decorate}, thick] (2,0) to[out=90,in=0] (1.5,1); 
    \end{scope}
    \node at(.8,.65) {\tiny{$k$}};
    \node at(2.5,.65) {\tiny{$n-k$}};
      \draw[style=thick] (1.5,1)--(1.5,1.25);
    \end{tikzpicture}& \raisebox{10pt}{1} \\
        && \\ \hline
    \end{tabular}}
    \caption{Local untagged pieces, corresponding tagged web fragments, and signs}\label{fig:replaceCKM}
\end{figure}
 
\begin{definition}\label{def: eta map}
Let $G\in F(\vec{k})$ with tripod decomposition $G'$. The \emph{tagged translation} $\eta(G')\in C(\vec{k})$ is the tagged web obtained from $G'$ by replacing each cup, cap, and tripod (locally flipped if necessary) with the corresponding tagged web fragment from Figure~\ref{fig:replaceCKM}. The \emph{sign} $\textup{sgn}(G')\in\{\pm 1\}$ is the product, over the local pieces of $G'$, of the signs in the third column of Figure~\ref{fig:replaceCKM}.
\end{definition}
 
\begin{example}\label{ex:etamap}
Figure~\ref{fig: ex etamap} shows a tripod decomposition $G'$ of the web from Example~\ref{ex:untagged web} (using the structure of Example~\ref{ex:websplit}) alongside its tagged translation $\eta(G')$. The decomposition has exactly two split vertices, so $\textup{sgn}(G')=(-1)^{1\cdot 2}\cdot(-1)^{1\cdot 2}=1$.
 
  \begin{figure}[h]
    \begin{center}
    \raisebox{-30pt}{\scalebox{.65}{\begin{tikzpicture}[scale=.65]
  \draw[style=dashed, <->] (0,4)--(13,4);
  \draw[style=dashed, red] (0,0)--(13,0);

\begin{scope}[thick,decoration={
    markings,
    mark=at position 0.5 with {\arrow{>}}}
    ] 

\draw[style=thick, postaction={decorate}] (5,0) to[out=90,in=180] (5.5, 1) to[out=0,in=90] (6,0);
\node at (5.5, 1.5) {$2$};

  \draw[style=thick, postaction={decorate}] (6,0) to[out=270,in=180] (7,-2)--(7,0);

  \draw[style=thick, postaction={decorate}] (7,-2) to[out=0,in=270] (11,0)--(11,4);
  \node at (11.25, 2) {$3$};
\end{scope}

\begin{scope}[thick,decoration={
    markings,
    mark=at position 0.5 with {\arrow{<}}}
    ] 
\draw[style=thick, postaction={decorate}] (2,0) to[out=90,in=180] (5.5, 2.5) to[out=0,in=90] (9,0);
\node at (5.5, 3) {$2$};

 \draw[style=thick, postaction={decorate}] (3,0) to[out=90,in=180] (3.5, 1) to[out=0,in=90] (4,0);
 \node at (3.5, 1.5) {$3$};
 
 \draw[style=thick, postaction={decorate}] (7,0) to[out=90,in=180] (7.5, 1) to[out=0,in=90] (8,0);
\node at (7.5, 1.5) {$1$};

\draw[style=thick, postaction={decorate}] (1,4)--(1,0) to[out=270,in=180] (2,-1)--(2,0);
\node at (.75,2) {$1$};

\draw[style=thick, postaction={decorate}] (2,-1) to[out=0,in=270] (3,0);

\draw[style=thick, postaction={decorate}] (4,0) to[out=270,in=180] (5,-3)--(5,0);

 \draw[style=thick, postaction={decorate}] (5,-3) to[out=0,in=270] (12,0)--(12,4);
 \node at (12.25, 2) {$1$};

\draw[style=thick, postaction={decorate}] (8,0) to[out=270,in=180] (9,-1)--(9,0);

\draw[style=thick, postaction={decorate}] (9,-1) to[out=0,in=270] (10,0)--(10,4);
  \node at (10.25, 2) {$3$};
\end{scope}
    
\draw[radius=.08, fill=black](1,4)circle;
\draw[radius=.08, fill=black](2,-1)circle;
\draw[radius=.08, fill=black](5,-3)circle;
\draw[radius=.08, fill=black](7,-2)circle;
\draw[radius=.08, fill=black](9,-1)circle;
\draw[radius=.08, fill=black](10,4)circle;
\draw[radius=.08, fill=black](11,4)circle;
\draw[radius=.08, fill=black](12,4)circle;
  
    \end{tikzpicture}}}
    \hspace{.25in}
    \raisebox{-36pt}{\scalebox{.65}{\begin{tikzpicture}[scale=.65]
  \draw[style=dashed, <->] (0,4)--(13,4);
  \draw[style=dashed, red] (0,0)--(13,0);

\begin{scope}[thick,decoration={
    markings,
    mark=at position 0.5 with {\arrow{>}}}
    ] 

\draw[style=thick, postaction={decorate}] (5,0) to[out=90,in=180] (5.5, 1);
\draw[style=thick, postaction={decorate}](6,0) to[out=90,in=0](5.5, 1) ;
\node at (5, 1.15) {$2$};
\node at (6, 1.15) {$2$};
\draw[style=thick] (5.5,1)--(5.5,1.25);

\draw[style=thick, postaction={decorate}] (3,0) to[out=90,in=180] (3.5, 1);
\draw[style=thick, postaction={decorate}](4,0)to[out=90,in=0](3.5, 1);
\node at (3, 1.15) {$1$};
\node at (4, 1.15) {$3$};
\draw[style=thick] (3.5,1)--(3.5,1.25);

\draw[style=thick, postaction={decorate}] (7,0) to[out=90,in=180] (7.5, 1);
\draw[style=thick, postaction={decorate}](8,0)to[out=90,in=0](7.5, 1);
\node at (7, 1.15) {$3$};
\node at (8, 1.15) {$1$};
\draw[style=thick] (7.5,1)--(7.5,1.25);

\draw[style=thick, postaction={decorate}] (2,0) to[out=90,in=180] (5.5, 2.5);
\draw[style=thick, postaction={decorate}] (9,0) to[out=90,in=0] (5.5, 2.5);
\node at (3.5, 2.5) {$2$};
\node at (7.5, 2.5) {$2$};
\draw[style=thick] (5.5,2.5)--(5.5,2.75);

 \draw[style=thick, postaction={decorate}] (12,0)--(12,4);
 \node at (12.35, 2) {$3$};

  \draw[style=thick, postaction={decorate}] (6.5, -3) to[out=0,in=270] (12,0);

 \draw[style=thick, postaction={decorate}](11,0)--(11,4);
  \node at (11.35, 2) {$3$};

  \draw[style=thick, postaction={decorate}] (10,0)--(10,4);
  \node at (10.3, 2) {$1$};
\end{scope}

\begin{scope}[thick,decoration={
    markings,
    mark=at position 0.5 with {\arrow{<}}}
    ] 

  \draw[style=thick, postaction={decorate}] (6,0) to[out=270,in=180] (7,-2);
   \draw[style=thick, postaction={decorate}] (7,0)--(7,-1);
  \draw[style=thick, postaction={decorate}] (7,-2)--(7,-1);
  \draw[thick] (7,-1)--(7.25,-1);
  \node at (6.65,-1.5) {$1$};

\draw[style=thick, postaction={decorate}] (1,4)--(1,0) to[out=270,in=180] (2,-1)--(2,0);
 \node at (.75,2) {$1$};

\draw[style=thick, postaction={decorate}] (2,-1) to[out=0,in=270] (3,0);
\draw[thick] (2.85,-.5)--(3.1,-.7);
\node at (2.35,-1.25) {$3$};

\draw[style=thick, postaction={decorate}] (4,0) to[out=270,in=180] (5,-3);
\draw[style=thick, postaction={decorate}](5,0)--(5,-1.5);
\draw[style=thick, postaction={decorate}](5,-3)--(5,-1.5);
\draw[thick] (5,-1.5)--(5.25,-1.5);
\node at (4.7,-2.5) {$2$};

 \draw[style=thick, postaction={decorate}] (5,-3) -- (6.5, -3);
 \draw[thick] (7.5,-2.95)--(7.55, -3.3);
 \node at (5.5, -3.35) {$1$};

   \draw[style=thick, postaction={decorate}] (7,-2)--(8.5, -2);
  \draw[style=thick, postaction={decorate}] (11,0)  to[out=270,in=0] (8.5, -2); 
 \draw[thick] (8.75,-1.95)--(8.8, -2.3);
 \node at (7.5, -2.35) {$1$};

\draw[style=thick, postaction={decorate}] (8,0) to[out=270,in=180] (9,-1)--(9,0);

\draw[style=thick, postaction={decorate}] (9,-1) to[out=0,in=270] (10,0);
\draw[thick] (9.85,-.55)--(10.1,-.8);
\node at (9.5, -1.25) {$3$};

\end{scope}
    
\draw[radius=.08, fill=black](1,4)circle;
\draw[radius=.08, fill=black](2,-1)circle;
\draw[radius=.08, fill=black](5,-3)circle;
\draw[radius=.08, fill=black](7,-2)circle;
\draw[radius=.08, fill=black](9,-1)circle;
\draw[radius=.08, fill=black](10,4)circle;
\draw[radius=.08, fill=black](11,4)circle;
\draw[radius=.08, fill=black](12,4)circle;
  
    \end{tikzpicture}}}
    
    \end{center}
    \caption{Constructing $\eta(G')$ in our running example using a tripod decomposition $G'$}\label{fig: ex etamap}
    \end{figure}
\end{example}
 
Note that $\eta$ and $\textup{sgn}$ depend on the choice of tripod decomposition, so neither is individually well-defined on $F(\vec{k})$. Their combined data, however, behaves well, as the next two lemmas show.
 
\begin{lemma}\label{lem:edge flip eta}
Let $G\in F(\vec{k})$ with tripod decomposition $G'$, and let $\mathcal{E}\subseteq E(G)=E(G')$. Then $G'_{\varphi(\mathcal{E})}$ is a tripod decomposition of $G_{\varphi(\mathcal{E})}$,
$$\eta(G')=\eta\bigl(G'_{\varphi(\mathcal{E})}\bigr),  \textup{ and }  \textup{sgn}(G')=\textup{sgn}\bigl(G'_{\varphi(\mathcal{E})}\bigr).$$
\end{lemma}
 
\begin{lemma}\label{lemma: f and g agree}
Let $G\in F(\vec{k})$ with tripod decomposition $G'$. Then
$$f(G)=\textup{sgn}(G')\,g(\eta(G')).$$
\end{lemma}
\begin{proof}
By Figure~\ref{fig:replaceCKM}, the CKM compositions recorded in Appendix~\ref{appendix:tagged ckm background}, and the local invariance proofs in Appendix~\ref{appendix:invariance calculations}, each local piece of $G'$ satisfies the lemma in the local form. In other words, for each cup, cap, and tripod, the value of $f$ equals the indicated sign times $g$ of the corresponding tagged fragment. Both $f$ and $g$ compose local pieces of a decomposition in the same way, so the global identity follows.
\end{proof}

\subsection{Quotient spaces, comparison, and surjectivity}\label{section: quotients}
 
To compare $f$ and $g$, we pass to quotients on each side by relations that preserve the underlying unweighted graph. On the untagged side, the relations are edge flips; on the tagged side, we have the CKM tag relations depicted in Figure~\ref{fig:tagrelations}. A linear isomorphism $\psi$ identifies the two quotient spaces. This is summarized by the commutative diagram below, whose vertical arrows are quotient maps and whose horizontal arrow $\psi$ is the isomorphism we construct.
 
\begin{center}
\begin{tikzcd}
\mathcal{F}(\vec{k}) \arrow[swap, dr, twoheadrightarrow, "f"] \arrow[dd]
& & \mathcal{C}(\vec{k}) \arrow[dl, twoheadrightarrow, "g"] \arrow[dd]  \\
& \textup{Inv}(\vec{k}) &\\
\widetilde{\mathcal{F}}(\vec{k}) \arrow[rr, Rightarrow, "\psi"] \arrow[swap, ur, twoheadrightarrow, "\tilde{f}"]
& & \widetilde{\mathcal{C}}(\vec{k}) \arrow[ul, twoheadrightarrow, "\tilde{g}"]
\end{tikzcd}
\end{center}
 
We follow the standard practice of expressing relations as equations: a \emph{relation in $\ker(f)$} is an equation $\mathbf{v}=\mathbf{w}$ with $f(\mathbf{v})=f(\mathbf{w})$, and a set of relations \emph{generates} $\ker(f)$ if the corresponding vectors $\mathbf{v}_i-\mathbf{w}_i$ span $\ker(f)$. The same terminology applies to $g$.
 
\subsubsection{Quotient spaces}
 
For $G\in F(\vec{k})$ and $\mathcal{E}\subseteq E(G)$, recall that $G_{\varphi(\mathcal{E})}$ is the web obtained by flipping each edge in $\mathcal{E}$. Let $\widetilde{\mathcal{F}}(\vec{k})$ be the quotient of $\mathcal{F}(\vec{k})$ by all \emph{edge-flip relations} $G=G_{\varphi(\mathcal{E})}$. Lemma~\ref{lem:invt and flips} gives:
 
\begin{lemma}
Edge-flip relations lie in $\ker(f)$, and $\tilde{f}:\widetilde{\mathcal{F}}(\vec{k})\to\textup{Inv}(\vec{k})$ given by $\tilde{f}([G])=f(G)$ is well-defined.
\end{lemma}
 
CKM prove that the \emph{tag switch}, \emph{cancellation}, and \emph{migration} relations in Figure~\ref{fig:tagrelations} lie in $\ker(g)$ \cite[Lemma 2.2.1, Theorem 3.2.1]{CKM}. Let $\widetilde{\mathcal{C}}(\vec{k})$ be the quotient of $\mathcal{C}(\vec{k})$ by these relations; then $\tilde{g}:\widetilde{\mathcal{C}}(\vec{k})\to\textup{Inv}(\vec{k})$ is well-defined.

\begin{figure}[h]
\begin{equation}\label{eqn:tag switch}
\raisebox{-15pt}{  \begin{tikzpicture}[scale=.65]
\node at (-3.5,2.5){\tiny{$k$}};
\node at (-3.7,3.5){\tiny{$n-k$}};
\begin{scope}[thick,decoration={
    markings,
    mark=at position 0.5 with {\arrow{>}}}
    ] 
\draw[style=thick, postaction={decorate}] (-3,2)--(-3,3);
\draw[style=thick, postaction={decorate}] (-3,4)--(-3,3);
\end{scope}
\draw[style=thick] (-3,3)--(-2.8,3);
\end{tikzpicture}} 
=
 \raisebox{-15pt}{ \begin{tikzpicture}[scale=.65]
\node at (-1.5,3){\tiny{$(-1)^{k(n-k)}$}};
\node at (.5,2.5){\tiny{$k$}};
\node at (.7,3.5){\tiny{$n-k$}};
\begin{scope}[thick,decoration={
    markings,
    mark=at position 0.5 with {\arrow{>}}}
    ] 
\draw[style=thick, postaction={decorate}] (0,2)--(0,3);
\draw[style=thick, postaction={decorate}] (0,4)--(0,3);
\end{scope}
 \draw[style=thick] (-.2,3)--(0,3);
\end{tikzpicture} }
\end{equation}
\begin{equation}\label{eqn:tag cancel}
\raisebox{-25pt}{  \begin{tikzpicture}[scale=.65]
\node at (-3.5,2.5){\tiny{$k$}};
\node at (-3.7,3.5){\tiny{$n-k$}};
\node at (-3.5,4.5){\tiny{$k$}};
\begin{scope}[thick,decoration={
    markings,
    mark=at position 0.5 with {\arrow{>}}}
    ]

\draw[style=thick, postaction={decorate}] (-3,2)--(-3,3);
\draw[style=thick, postaction={decorate}] (-3,4)--(-3,3);
\draw[style=thick,postaction={decorate}] (-3,4)--(-3,5);

\end{scope}
\draw[style=thick] (-3,3)--(-2.8,3);

\draw[style=thick] (-3,4)--(-3.2,4);
\end{tikzpicture}} 
=
 \raisebox{-25pt}{ \begin{tikzpicture}[scale=.65]
\node at (-1.5,3.5){\tiny{$k$}};
\begin{scope}[thick,decoration={
    markings,
    mark=at position 0.5 with {\arrow{>}}}
    ] 
\draw[style=thick, postaction={decorate}] (-1,2)--(-1,5);
\end{scope}
\end{tikzpicture}}
\end{equation}
\begin{equation}\label{eqn:tag migrate}
\raisebox{-25pt}{  
\begin{tikzpicture}[scale=.45]
\begin{scope}[thick,decoration={
    markings,
    mark=at position 0.5 with {\arrow{>}}}
    ] 
\draw[style=thick, postaction={decorate}] (0,0)--(0,1);
\draw[style=thick, postaction={decorate}] (0,2)--(0,1);
\draw[style=thick, postaction={decorate}] (-2,-2)--(0,0);
\draw[style=thick,postaction={decorate}] (2,-2)--(0,0);
\end{scope}

\draw[style=thick] (0,1)--(.2,1);
\node at (-1.75,1.75) {\tiny{$n-k-l$}};
\node at (-1.25,.5) {\tiny{$k+l$}};
\node at (-2,-1) {\tiny{$k$}};
\node at (1.8,-1) {\tiny{$l$}};
\end{tikzpicture}} 
=
 \raisebox{-25pt}{ 
 \begin{tikzpicture}[scale=.45]
\begin{scope}[thick,decoration={
    markings,
    mark=at position 0.5 with {\arrow{>}}}
    ] 
\draw[style=thick, postaction={decorate}] (6,2)--(6,0);
\draw[style=thick, postaction={decorate}] (4,-2)--(6,0);
\draw[style=thick,postaction={decorate}] (6,0)--(7,-1);
\draw[style=thick,postaction={decorate}] (8,-2)--(7,-1);

\end{scope}

\draw[style=thick] (7,-1)--(7.2,-.8);
\node at (7.75,1.25) {\tiny{$n-k-l$}};
\node at (4,-1) {\tiny{$k$}};
\node at (7.5,-.25) {\tiny{$n-l$}};
\node at (8,-1.25) {\tiny{$l$}};
\end{tikzpicture}}\hspace{.25in}
\raisebox{-25pt}{  
\begin{tikzpicture}[scale=.45]
\begin{scope}[thick,decoration={
    markings,
    mark=at position 0.5 with {\arrow{>}}}
    ] 
\draw[style=thick, postaction={decorate}] (0,0)--(0,1);
\draw[style=thick, postaction={decorate}] (0,2)--(0,1);
\draw[style=thick, postaction={decorate}] (-2,-2)--(0,0);
\draw[style=thick,postaction={decorate}] (0,0)--(2,-2);
\end{scope}
\node at (-1.25,1.75) {\tiny{$n-k$}};
\node at (-.75,.5) {\tiny{$k$}};
\node at (-2.25,-1) {\tiny{$k+l$}};
\node at (1.8,-1) {\tiny{$l$}};
\draw[style=thick] (0,1)--(-.2,1);
\end{tikzpicture}
} 
=
 \raisebox{-25pt}{ 
\begin{tikzpicture}[scale=.45]
\begin{scope}[thick,decoration={
    markings,
    mark=at position 0.5 with {\arrow{>}}}
    ] 
\draw[style=thick, postaction={decorate}] (6,2)--(6,0);
\draw[style=thick, postaction={decorate}] (4,-2)--(5,-1);
\draw[style=thick, postaction={decorate}] (6,0)--(5,-1);
\draw[style=thick,postaction={decorate}] (6,0)--(8,-2);

\end{scope}
\node at (7.25,1.25) {\tiny{$n-k$}};
\node at (3.25,-1.5) {\tiny{$k+l$}};
\node at (3.85,-.5) {\tiny{$n-k-l$}};
\node at (7.5,-1) {\tiny{$l$}};
\draw[style=thick] (5,-1)--(4.8,-.8);
\end{tikzpicture}}
\end{equation}

\begin{equation}\label{eqn:tag migrate weird}
\raisebox{-30pt}{ \begin{tikzpicture}[scale=.45]
\begin{scope}[thick,decoration={
    markings,
    mark=at position 0.5 with {\arrow{<}}}
    ] 
\draw[style=thick, postaction={decorate}] (0,0)--(0,1);
\draw[style=thick, postaction={decorate}] (0,2)--(0,1);
\draw[style=thick, postaction={decorate}] (-2,-2)--(0,0);
\draw[style=thick,postaction={decorate}] (2,-2)--(0,0);
\end{scope}

\draw[style=thick] (0,1)--(.2,1);
\node at (-1.75,1.75) {\tiny{$n-k-l$}};
\node at (-1.25,.5) {\tiny{$k+l$}};
\node at (-2,-1) {\tiny{$k$}};
\node at (1.8,-1) {\tiny{$l$}};
\end{tikzpicture}} 
=
 \raisebox{-25pt}{ 
\begin{tikzpicture}[scale=.45]
\begin{scope}[thick,decoration={
    markings,
    mark=at position 0.5 with {\arrow{<}}}
    ] 
\draw[style=thick, postaction={decorate}] (6,2)--(6,0);
\draw[style=thick, postaction={decorate}] (4,-2)--(5,-1);
\draw[style=thick, postaction={decorate}] (6,0)--(5,-1);
\draw[style=thick,postaction={decorate}] (8,-2)--(6,0);
\end{scope}
\node at (7.5,1.25) {\tiny{$n-k-l$}};
\node at (3.5,-1.5) {\tiny{$k$}};
\node at (4.25,-.5) {\tiny{$n-k$}};
\node at (7.5,-1) {\tiny{$l$}};
\draw[style=thick] (5,-1)--(5.2,-1.2);
\end{tikzpicture}}
\hspace{.25in}
\raisebox{-25pt}{  
\begin{tikzpicture}[scale=.45]
\begin{scope}[thick,decoration={
    markings,
    mark=at position 0.5 with {\arrow{<}}}
    ] 
\draw[style=thick, postaction={decorate}] (0,0)--(0,1);
\draw[style=thick, postaction={decorate}] (0,2)--(0,1);
\draw[style=thick, postaction={decorate}] (-2,-2)--(0,0);
\draw[style=thick,postaction={decorate}] (0,0)--(2,-2);
\end{scope}
\node at (-1.5,1.75) {\tiny{$n-k$}};
\node at (-1,.5) {\tiny{$k$}};
\node at (-2.25,-1) {\tiny{$k+l$}};
\node at (1.8,-1) {\tiny{$l$}};
\draw[style=thick] (0,1)--(.2,1);
\end{tikzpicture}
} 
=
 \raisebox{-25pt}{ 
\begin{tikzpicture}[scale=.45]
\begin{scope}[thick,decoration={
    markings,
    mark=at position 0.5 with {\arrow{<}}}
    ] 
\draw[style=thick, postaction={decorate}] (6,2)--(6,0);
\draw[style=thick, postaction={decorate}] (4,-2)--(5,-1);
\draw[style=thick, postaction={decorate}] (6,0)--(5,-1);
\draw[style=thick,postaction={decorate}] (6,0)--(8,-2);
\end{scope}
\node at (7.25,1.25) {\tiny{$n-k$}};
\node at (3.5,-1.5) {\tiny{$k+l$}};
\node at (3.75,-.5) {\tiny{$n-k-l$}};
\node at (7.5,-1) {\tiny{$l$}};
\draw[style=thick] (5,-1)--(5.2,-1.2);
\end{tikzpicture}}
\end{equation}
\caption{The tag switch (\ref{eqn:tag switch}), tag cancellation (\ref{eqn:tag cancel}), and tag migration (\ref{eqn:tag migrate}, \ref{eqn:tag migrate weird}) relations on $\mathcal{C}(\vec{k})$.}\label{fig:tagrelations}
\end{figure}
 
\subsubsection{Forgetting and adding tags}
 
\begin{definition}
For $H\in C(\vec{k})$ and $G\in F(\vec{k})$, we say $G$ \emph{forgets the tags in} $H$ (equivalently, $H$ \emph{adds tags to} $G$) if:
\begin{enumerate}
\item $G$ and $H$ have isotopic underlying unweighted, undirected plane graphs, and
\item under some isotopic identification, every portion of each tagged edge of $H$ agrees in orientation and label with either the corresponding edge $e$ of $G$ or $\varphi(e)$.
\end{enumerate}
\end{definition}
 Forgetting and adding tags preserves the split/merge type of each trivalent vertex.

 Given $H\in C(\vec{k})$, let $\pi(H)=[G]\in\widetilde{\mathcal{F}}(\vec{k})$, where $G$ is any untagged web that forgets the tags in $H$. By construction, $G_1$ and $G_2$ forget the tags in $H$ if and only if $[G_1]=[G_2]\in\widetilde{\mathcal{F}}(\vec{k})$, so $\pi$ is well-defined.
 
\begin{example}\label{ex:forget tags}
Figure~\ref{fig: ex forget tags} shows a tagged web (left) and an untagged web (right) that forgets its tags. These are the webs of Examples~\ref{ex:untagged web} and~\ref{ex:tagged web}.
 
\begin{figure}[h]
 \begin{center}
 \raisebox{3pt}{\begin{tikzpicture}[scale=.75]

\draw[style=dashed, <->] (0,0)--(7,0);

\begin{scope}[thick,decoration={
    markings,
    mark=at position 0.5 with {\arrow{<}}}
    ] 

\draw[postaction={decorate},style=thick] (1,0) to[out=270,in=180] (3,-2);
\draw[postaction={decorate},style=thick] (3,-2)--(3,-1);
\draw[postaction={decorate},style=thick] (5,0) to[out=270,in=0] (4,-1);

\draw[postaction={decorate},style=thick] (6,0) to[out=270,in=0] (4,-2);
\draw[postaction={decorate},style=thick] (4,-2)--(3,-2);
\end{scope}

\begin{scope}[thick,decoration={
    markings,
    mark=at position 0.5 with {\arrow{>}}}
    ] 

\draw[postaction={decorate},style=thick] (3.5,-1)--(3,-1);
\draw[postaction={decorate},style=thick] (3.5,-1)--(4,-1);
\draw[postaction={decorate},style=thick] (3,-1) to[out=180,in=270] (2,0);

\draw[postaction={decorate},style=thick] (4,-1.5)--(4,-1);
\draw[postaction={decorate},style=thick] (4,-1.5)--(4,-2);
\end{scope}

\draw[radius=.05, fill=black](1,0)circle;
\draw[radius=.05, fill=black](3,-1)circle;
\draw[radius=.05, fill=black](4,-1)circle;
\draw[radius=.05, fill=black](3,-2)circle;
\draw[radius=.05, fill=black](4,-2)circle;
\draw[radius=.05, fill=black](2,0)circle;
\draw[radius=.05, fill=black](5,0)circle;
\draw[radius=.05, fill=black](6,0)circle;

\draw[thick] (3.5,-1)--(3.5,-.75);
\draw[thick] (4,-1.5)--(3.75,-1.5);

\node at (.75,-.25) {\tiny{$1$}};
\node at (1.75,-.25) {\tiny{$1$}};
\node at (3.25,-.75) {\tiny{$3$}};
\node at (3.75,-.75) {\tiny{$1$}};
\node at (3.5,-2.25) {\tiny{$1$}};
\node at (2.75, -1.5) {\tiny{$2$}};
\node at (4.25, -1.25) {\tiny{$2$}};
\node at (4.25, -1.75) {\tiny{$2$}};
\node at (6.25, -.25) {\tiny{$3$}};
\node at (5.25, -.25) {\tiny{$3$}};
  
    \end{tikzpicture}}
\hspace{.15in} \raisebox{35pt}{\large{$\stackrel{\pi}{\longmapsto}$}}
\raisebox{3pt}{\begin{tikzpicture}[scale=.75]

\draw[style=dashed, <->] (0,0)--(7,0);

\begin{scope}[thick,decoration={
    markings,
    mark=at position 0.5 with {\arrow{<}}}
    ] 

\draw[postaction={decorate},style=thick] (1,0) to[out=270,in=180] (3,-2);
\draw[postaction={decorate},style=thick] (3,-2)--(3,-1);
\draw[postaction={decorate},style=thick] (3,-2)--(4,-2);
\draw[postaction={decorate},style=thick] (5,0) to[out=270,in=0] (4,-1);
\draw[postaction={decorate},style=thick] (4,-1)--(3,-1);
\draw[postaction={decorate},style=thick] (3,-1) to[out=180,in=270] (2,0);
\end{scope}

\begin{scope}[thick,decoration={
    markings,
    mark=at position 0.5 with {\arrow{>}}}
    ] 

\draw[postaction={decorate},style=thick] (6,0) to[out=270,in=0] (4,-2);
\draw[postaction={decorate},style=thick] (4,-2)--(4,-1);
\end{scope}

\draw[radius=.05, fill=black](1,0)circle;
\draw[radius=.05, fill=black](3,-1)circle;
\draw[radius=.05, fill=black](4,-1)circle;
\draw[radius=.05, fill=black](3,-2)circle;
\draw[radius=.05, fill=black](4,-2)circle;
\draw[radius=.05, fill=black](2,0)circle;
\draw[radius=.05, fill=black](5,0)circle;
\draw[radius=.05, fill=black](6,0)circle;

\node at (.75,-.25) {\tiny{$1$}};
\node at (1.75,-.25) {\tiny{$3$}};
\node at (3.5,-.75) {\tiny{$1$}};
\node at (3.5,-2.25) {\tiny{$3$}};
\node at (2.75, -1.5) {\tiny{$2$}};
\node at (4.25, -1.5) {\tiny{$2$}};
\node at (6.25, -.25) {\tiny{$1$}};
\node at (5.25, -.25) {\tiny{$3$}};
  
    \end{tikzpicture}}
    \end{center}

    \caption{Forgetting tags in a tagged web}\label{fig: ex forget tags}
\end{figure} 
\end{example}

The next lemma shows $\pi$ identifies tagged webs up to sign with their untagged counterparts in the quotient.

\begin{lemma}\label{lem:pi map}
Let $H, H'\in C(\vec{k})$. Then $[H]=\pm[H']\in\widetilde{\mathcal{C}}(\vec{k})$ if and only if $\pi(H)=\pi(H')\in\widetilde{\mathcal{F}}(\vec{k})$.
\end{lemma}
\begin{proof}
The forward direction is direct from Figure~\ref{fig:tagrelations} since each CKM tag relation preserves $\pi$. The converse requires showing that if $\pi(H)=\pi(H')$ then $H$ and $H'$ differ (up to sign) by tag relations. A vertex-by-vertex argument using tag cancellation and migration handles this. The details are in Appendix~\ref{appendix:relation arguments}.
\end{proof}
 
\subsubsection{The comparison isomorphism}
 
\begin{lemma}\label{lem:adding tags with sign is well-defined}
Let $G\in F(\vec{k})$ and let $G'_1, G'_2$ be tripod decompositions of $G$. Then
$$\textup{sgn}(G'_1)[\eta(G'_1)] = \textup{sgn}(G'_2)[\eta(G'_2)] \in \widetilde{\mathcal{C}}(\vec{k}).$$
\end{lemma}
\begin{proof}
By Lemma~\ref{lemma: f and g agree}, $\textup{sgn}(G'_1)g(\eta(G'_1)) = \textup{sgn}(G'_2)g(\eta(G'_2))$. Both $\eta(G'_1)$ and $\eta(G'_2)$ add tags to $G$, so $\pi(\eta(G'_1)) = [G] = \pi(\eta(G'_2))$, and Lemma~\ref{lem:pi map} gives $[\eta(G'_1)] = \pm[\eta(G'_2)]$. These observations together with the nonvanishing of $g(\eta(G'_1))$ shows  $\textup{sgn}(G'_1)[\eta(G'_1)] = \textup{sgn}(G'_2)[\eta(G'_2)]$.
\end{proof}
 
Lemmas~\ref{lem:edge flip eta} and~\ref{lem:adding tags with sign is well-defined} together show that $\textup{sgn}(G')[\eta(G')]\in\widetilde{\mathcal{C}}(\vec{k})$ depends only on $[G]\in\widetilde{\mathcal{F}}(\vec{k})$. We therefore have a well-defined linear map
$$\psi:\widetilde{\mathcal{F}}(\vec{k})\to\widetilde{\mathcal{C}}(\vec{k}), \qquad \psi([G]) = \textup{sgn}(G')[\eta(G')].$$
 
\begin{theorem}\label{thm:psi isomorphism}
The map $\psi:\widetilde{\mathcal{F}}(\vec{k})\to\widetilde{\mathcal{C}}(\vec{k})$ is an isomorphism, and $\tilde{f}=\tilde{g}\circ\psi$.
\end{theorem}
\begin{proof}
The identity $\tilde{f}=\tilde{g}\circ\psi$ is immediate from Lemma~\ref{lemma: f and g agree}.
 
Given $H\in C(\vec{k})$, let $G$ forget the tags in $H$ and let $G'$ be a tripod decomposition of $G$. Both $\eta(G')$ and $H$ add tags to $G$, so $\pi(\eta(G')) = [G] = \pi(H)$. Lemma~\ref{lem:pi map} gives $[H]=\psi(\pm[G])$, so $\psi$ is surjective. By Lemma~\ref{lem:pi map}, linearly independent classes $[G_i]\in\widetilde{\mathcal{F}}(\vec{k})$ map to linearly independent classes  $\pm[\eta(G_i')] \in \widetilde{\mathcal{C}}(\vec{k})$, so $\psi$ is injective.
\end{proof}
 
\begin{corollary}\label{cor:surjective psi}
The map $\tilde{f}:\widetilde{\mathcal{F}}(\vec{k})\to\textup{Inv}(\vec{k})$ is surjective, and $\ker(\tilde{f})=\psi^{-1}(\ker(\tilde{g}))$. Consequently, $f:\mathcal{F}(\vec{k})\to\textup{Inv}(\vec{k})$ is surjective.
\end{corollary}
\begin{proof}
Since $\tilde{g}$ is surjective (\cite{CKM}) and $\psi$ is an isomorphism, we have $\tilde{f}=\tilde{g}\circ\psi$ is surjective. Surjectivity of $f$ follows because the quotient map $\mathcal{F}(\vec{k})\twoheadrightarrow\widetilde{\mathcal{F}}(\vec{k})$ is surjective.
\end{proof}
 
\subsection{Relations on untagged webs}\label{section: untagged relations}

We can now transfer the CKM relations through $\psi$ to obtain a generating set of relations for $\ker(f)$.
 
CKM prove that the relations in Figure~\ref{fig:CKM relations}, together with the tag relations of Figure~\ref{fig:tagrelations}, generate $\ker(g)$ \cite{CKM}. The relations in Figure~\ref{fig:CKM relations} hold for all integer edge weights, with the conventions that an edge of weight $0$ or $n$ is erased and a trivalent vertex incident to an edge of weight $n$ is replaced by a tag (as illustrated in Figure~\ref{fig:CKM n to tag}) and any graph with an edge weight outside of the interval $[0,n]$ is deleted.
 
  \begin{figure}[h]

    \begin{tikzpicture}[scale=.75]
\draw[radius=.08, fill=black](0,0)circle;
\begin{scope}[thick,decoration={
    markings,
    mark=at position 0.5 with {\arrow{>}}}
    ] 
\draw[postaction={decorate}, thick] (0,-1)--(0,0);
\draw[postaction={decorate}, thick] (0,0)--(-1,1);
\draw[postaction={decorate}, thick] (0,0)--(1,1);

\draw[postaction={decorate}, thick] (4,-.5) to[out=180,in=270] (3,.5);
\draw[postaction={decorate}, thick] (4,-.5) to[out=0,in=270] (5,.5);
    \end{scope}
\draw[thick] (4,-.5)--(4,-.25);
    \node at (-.3,-.5) {\tiny{$n$}};

    \node at (2,0) {$=$};

    \end{tikzpicture}
    \hspace{1in}
       \begin{tikzpicture}[scale=.75]
\draw[radius=.08, fill=black](0,0)circle;
\begin{scope}[thick,decoration={
    markings,
    mark=at position 0.5 with {\arrow{<}}}
    ] 
\draw[postaction={decorate}, thick] (0,1)--(0,0);
\draw[postaction={decorate}, thick] (0,0)--(-1,-1);
\draw[postaction={decorate}, thick] (0,0)--(1,-1);
   \end{scope}
   \begin{scope}[thick,decoration={
    markings,
    mark=at position 0.5 with {\arrow{>}}}
    ] 
\draw[postaction={decorate}, thick] (3,-.5) to[out=90,in=180] (4,.5);
\draw[postaction={decorate}, thick] (5,-.5) to[out=90,in=0] (4,.5);
    \end{scope}
\draw[thick] (4,.5)--(4,.75);
    \node at (-.4,.5) {\tiny{$n$}};

    \node at (2,0) {$=$};

    \end{tikzpicture}
    
        \caption{Converting edges of weight $n$ to tags in CKM relations.}\label{fig:CKM n to tag}
    \end{figure}

\begin{theorem}[Cautis-Kamnitzer-Morrison]\label{thm:CKM relations}
The CKM relations in Figure~\ref{fig:CKM relations} (together with their mirror images and arrow reversals) and the tag relations of Figure~\ref{fig:tagrelations} generate $\ker(g)$.
\end{theorem}
 
\begin{figure}[h]

\begin{equation}\label{eqn:CKM bigon}
    \raisebox{-40pt}{\begin{tikzpicture}[scale=.5]
    \draw[radius=.12, fill=black](1,2)circle;
     \draw[radius=.12, fill=black](2,4)circle;
     \node at (-.75,1) {\tiny{$k+l$}};
     \node at (1.25,5) {\tiny{$k+l$}};
       \node at (0,3.75) {\tiny{$k$}};
     \node at (3,2.25) {\tiny{$l$}};
    \begin{scope}[thick,decoration={
    markings,
    mark=at position 0.5 with {\arrow{>}}}
    ] 
        \draw[postaction={decorate}, thick] (0,0)--(1,2);
        \draw[postaction={decorate}, thick] (2,4)--(3,6);
        \draw[postaction={decorate}, thick] (1,2)to[out=150, in=240] (.5, 3.5) to[out=60, in=165] (2,4);
        \draw[postaction={decorate}, thick] (1,2)to[out=-30, in=240] (2.5, 2.5) to[out=60, in=-15] (2,4);
        
        \end{scope}
    \end{tikzpicture}} = \genfrac[]{0pt}{0}{k+l}{l}_q \raisebox{-40pt}{\begin{tikzpicture}[scale=.5]
    \node at (1,4) {\tiny{$k+l$}};
      \begin{scope}[thick,decoration={
    markings,
    mark=at position 0.5 with {\arrow{>}}}
    ] 
        \draw[postaction={decorate}, thick] (0,0)--(3,6);
        \end{scope}
    \end{tikzpicture}}
\end{equation}

\begin{equation}\label{eqn:CKM IH}
  \raisebox{-30pt}{\begin{tikzpicture}[scale=.75]
    \draw[radius=.08, fill=black](1,2)circle;
     \draw[radius=.08, fill=black](.5,1)circle;
     \node at (0,-.25) {\tiny{$k$}};
     \node at (1,-.25) {\tiny{$l$}};
      \node at (2,-.25) {\tiny{$m$}};
       \node at (.2,1.65) {\tiny{$k+l$}};
     \node at (1,3) {\tiny{$k+l+m$}};
    \begin{scope}[thick,decoration={
    markings,
    mark=at position 0.5 with {\arrow{>}}}
    ] 
        \draw[postaction={decorate}, thick] (0,0)--(.5,1);
        \draw[postaction={decorate}, thick] (1,0)--(.5,1);
        \draw[postaction={decorate}, thick] (.5,1)--(1,2);
        \draw[postaction={decorate}, thick] (2,0)--(1,2);
        \draw[postaction={decorate}, thick] (1,2)--(1,2.75);
        
        \end{scope}
    \end{tikzpicture}} = 
      \raisebox{-30pt}{\begin{tikzpicture}[scale=.75]
    \draw[radius=.08, fill=black](1,2)circle;
     \draw[radius=.08, fill=black](1.5,1)circle;
     \node at (0,-.25) {\tiny{$k$}};
     \node at (1,-.25) {\tiny{$l$}};
      \node at (2,-.25) {\tiny{$m$}};
       \node at (1.8,1.65) {\tiny{$l+m$}};
     \node at (1,3) {\tiny{$k+l+m$}};
    \begin{scope}[thick,decoration={
    markings,
    mark=at position 0.5 with {\arrow{>}}}
    ] 
        \draw[postaction={decorate}, thick] (0,0)--(1,2);
        \draw[postaction={decorate}, thick] (1,0)--(1.5,1);
        \draw[postaction={decorate}, thick] (1.5,1)--(1,2);
        \draw[postaction={decorate}, thick] (2,0)--(1.5,1);
        \draw[postaction={decorate}, thick] (1,2)--(1,2.75);
        
        \end{scope}
    \end{tikzpicture}}
\end{equation}

\begin{equation}\label{eqn:CKM square removal}
  \raisebox{-50pt}{\begin{tikzpicture}[scale=.75]
    \draw[radius=.08, fill=black](0,0)circle;
     \draw[radius=.08, fill=black](0,2)circle;
        \draw[radius=.08, fill=black](2,0)circle;
     \draw[radius=.08, fill=black](2,2)circle;
     \node at (0,-1.25) {\tiny{$k$}};
     \node at (2,-1.25) {\tiny{$l$}};
      \node at (1,-.25) {\tiny{$s$}};
      \node at (1,2.25) {\tiny{$r$}};
       \node at (-.75,1) {\tiny{$k-s$}};
       \node at (2.75,1) {\tiny{$l+s$}};
  \node at (0,3.25) {\tiny{$k-r-s$}};
     \node at (2,3.25) {\tiny{$l+r+s$}};
    \begin{scope}[thick,decoration={
    markings,
    mark=at position 0.5 with {\arrow{>}}}
    ] 
        \draw[postaction={decorate}, thick] (0,-1)--(0,0);
        \draw[postaction={decorate}, thick] (0,0)--(0,2);
        \draw[postaction={decorate}, thick] (0,2)--(0,3);
         \draw[postaction={decorate}, thick] (2,-1)--(2,0);
        \draw[postaction={decorate}, thick] (2,0)--(2,2);
        \draw[postaction={decorate}, thick] (2,2)--(2,3);
        \draw[postaction={decorate}, thick] (0,0)--(2,0);
        \draw[postaction={decorate}, thick] (0,2)--(2,2);
        
        \end{scope}
    \end{tikzpicture}} = \genfrac[]{0pt}{0}{r+s}{r}_q \raisebox{-30pt}{\begin{tikzpicture}[scale=.75]
    \draw[radius=.08, fill=black](0,0)circle;
     \draw[radius=.08, fill=black](2,0)circle;
     \node at (0,-1.25) {\tiny{$k$}};
     \node at (2,-1.25) {\tiny{$l$}};
      \node at (1,-.25) {\tiny{$r+s$}};
  \node at (0,1.25) {\tiny{$k-r-s$}};
     \node at (2,1.25) {\tiny{$l+r+s$}};
    \begin{scope}[thick,decoration={
    markings,
    mark=at position 0.5 with {\arrow{>}}}
    ] 
        \draw[postaction={decorate}, thick] (0,-1)--(0,0);
        \draw[postaction={decorate}, thick] (0,0)--(0,1);
         \draw[postaction={decorate}, thick] (2,-1)--(2,0);
        \draw[postaction={decorate}, thick] (2,0)--(2,1);
        \draw[postaction={decorate}, thick] (0,0)--(2,0);
        
        \end{scope}
    \end{tikzpicture}}
\end{equation}

\begin{equation}\label{eqn:CKM square switch r=s=1}
\raisebox{-50pt}{\begin{tikzpicture}[scale=.75]
    \draw[radius=.08, fill=black](0,0)circle;
     \draw[radius=.08, fill=black](0,2)circle;
        \draw[radius=.08, fill=black](2,0)circle;
     \draw[radius=.08, fill=black](2,2)circle;
     \node at (0,-1.25) {\tiny{$k$}};
     \node at (2,-1.25) {\tiny{$l$}};
      \node at (1,-.25) {\tiny{$1$}};
      \node at (1,2.25) {\tiny{$1$}};
       \node at (-.75,1) {\tiny{$k-1$}};
       \node at (2.75,1) {\tiny{$l+1$}};
  \node at (0,3.25) {\tiny{$k$}};
     \node at (2,3.25) {\tiny{$l$}};
    \begin{scope}[thick,decoration={
    markings,
    mark=at position 0.5 with {\arrow{>}}}
    ] 
        \draw[postaction={decorate}, thick] (0,-1)--(0,0);
        \draw[postaction={decorate}, thick] (0,0)--(0,2);
        \draw[postaction={decorate}, thick] (0,2)--(0,3);
         \draw[postaction={decorate}, thick] (2,-1)--(2,0);
        \draw[postaction={decorate}, thick] (2,0)--(2,2);
        \draw[postaction={decorate}, thick] (2,2)--(2,3);
        \draw[postaction={decorate}, thick] (0,0)--(2,0);
        \draw[postaction={decorate}, thick] (2,2)--(0,2);
        
        \end{scope}
    \end{tikzpicture}} = 
        \raisebox{-50pt}{\begin{tikzpicture}[scale=.75]
    \draw[radius=.08, fill=black](0,0)circle;
     \draw[radius=.08, fill=black](0,2)circle;
        \draw[radius=.08, fill=black](2,0)circle;
     \draw[radius=.08, fill=black](2,2)circle;
     \node at (0,-1.25) {\tiny{$k$}};
     \node at (2,-1.25) {\tiny{$l$}};
      \node at (1,-.25) {\tiny{$1$}};
      \node at (1,2.25) {\tiny{$1$}};
       \node at (-.75,1) {\tiny{$k+1$}};
       \node at (2.75,1) {\tiny{$l-1$}};
  \node at (0,3.25) {\tiny{$k$}};
     \node at (2,3.25) {\tiny{$l$}};
    \begin{scope}[thick,decoration={
    markings,
    mark=at position 0.5 with {\arrow{>}}}
    ] 
        \draw[postaction={decorate}, thick] (0,-1)--(0,0);
        \draw[postaction={decorate}, thick] (0,0)--(0,2);
        \draw[postaction={decorate}, thick] (0,2)--(0,3);
         \draw[postaction={decorate}, thick] (2,-1)--(2,0);
        \draw[postaction={decorate}, thick] (2,0)--(2,2);
        \draw[postaction={decorate}, thick] (2,2)--(2,3);
        \draw[postaction={decorate}, thick] (2,0)--(0,0);
        \draw[postaction={decorate}, thick] (0,2)--(2,2);
        
        \end{scope}
    \end{tikzpicture}} + [k-l]_q \raisebox{-40pt}{\begin{tikzpicture}[scale=.75]
       \node at (0,3.25) {\tiny{$k$}};
       \node at (2,3.25) {\tiny{$l$}};

    \begin{scope}[thick,decoration={
    markings,
    mark=at position 0.5 with {\arrow{>}}}
    ] 
        \draw[postaction={decorate}, thick] (0,-1)--(0,3);
         \draw[postaction={decorate}, thick] (2,-1)--(2,3);
        \end{scope}
    \end{tikzpicture}}
\end{equation}

    \caption{CKM Bigon removal (\ref{eqn:CKM bigon}), `$I = H$' (\ref{eqn:CKM IH}), square removal (\ref{eqn:CKM square removal}), and square switch (\ref{eqn:CKM square switch r=s=1}) relations.  Relations for untagged webs are obtained by replacing $q$ by $(-q)$ in the quantum binomial coefficients, and including the edge flip relation $u \stackrel{\ell}{\rightarrow} v = u \stackrel{n-\ell}{\leftarrow} v$}\label{fig:CKM relations}
\end{figure}
 
\begin{remark}
This generating set differs slightly from \cite{CKM}: we have added tag cancellation, removed one of two bigon relations, and stated a restricted square switch (with horizontal edge weights $1$) from which the general form follows inductively. The generalized square switch on untagged webs appears in Corollary~\ref{cor:generalized square switch}.
\end{remark}
 
Applying the quotient map $\mathcal{C}(\vec{k})\twoheadrightarrow\widetilde{\mathcal{C}}(\vec{k})$ to the CKM relations gives a generating set for $\ker(\tilde{g})$. By Corollary~\ref{cor:surjective psi}, $\ker(\tilde{f}) = \psi^{-1}(\ker(\tilde{g}))$, and $\ker(f)$ is generated by edge-flip relations together with any choice of representatives in $\mathcal{F}(\vec{k})$ for the relations generating $\ker(\tilde{f})$.
 
\begin{theorem}\label{thm:Fontaine relations}
Edge-flip relations together with the diagrammatic relations of Figure~\ref{fig:CKM relations}, modified as follows, generate $\ker(f)$.
\begin{enumerate}
\item Quantum integers in the bigon, square removal, and square switch relations are evaluated at $-q$.
\item Edges of weight $0$ or $n$ are erased; a vertex incident to such an edge is removed, and the remaining two incident edges merge into a single edge whose orientation and weight match the original two (one of them if the erased edge had weight $n$, both if weight $0$).
\end{enumerate}
\end{theorem}
\begin{proof}
For the bigon, $I$=$H$, and square removal relations, there is a bijection between strandings of the two graphs that preserves directed strands up to isotopy, thereby confirming the corresponding relations.. This is proven in Appendix~\ref{appendix:relation arguments}. The square switch relation combines Lemma~\ref{lemma: overall square switch argument} (for summands where the parallel edges admit a stranding) with Lemma~\ref{lemma: bijection for relations with no interior faces} and the loop substitution~\eqref{eqn:Fontaine circle} (for the remaining summands). The statements and proofs of these lemmas are in Appendix~\ref{appendix:relation arguments}. When edges of weight $0$ or $n$ appear the arguments still apply, with the relevant strand contributions vacuous. Mirror images and arrow reversals follow by composing the listed relations with edge flips.
\end{proof}
 
For instance, when $k+l=n$ the bigon relation on untagged webs reduces (up to edge flipping) to the following \emph{circle removal} relation:
\begin{equation}\label{eqn:Fontaine circle}
\raisebox{-12pt}{\begin{tikzpicture}[scale=.5]
\node at (0,.85) {\tiny{$k$}};
\begin{scope}[thick,decoration={
    markings,
    mark=at position 0.5 with {\arrow{>}}}
    ]
\draw[postaction={decorate}, thick] (0,0) to[out=90, in=180] (1,1) to[out=0, in=90] (2,0) to[out=270, in=0] (1,-1) to[out=180,in=270] (0,0);
\end{scope}
\end{tikzpicture}}
= \genfrac[]{0pt}{1}{n}{k}_q\bigg\vert_{-q} = (-1)^{k(n-k)}\genfrac[]{0pt}{1}{n}{k}_q.
\end{equation}
 
The generalized square switch on untagged webs follows by induction from Theorem~\ref{thm:Fontaine relations}. It matches the CKM formula in the tagged setting except that the quantum integers are evaluated at $-q$, introducing a factor of $(-1)^{(k-l+r-s-1)t}$ in each summand.
 
\begin{corollary}\label{cor:generalized square switch}
The generalized square switch relation
\begin{equation}\label{eqn:Fontaine square switch general}
\raisebox{-50pt}{\begin{tikzpicture}[scale=.6]
    \draw[radius=.08, fill=black](0,0)circle;
     \draw[radius=.08, fill=black](0,2)circle;
        \draw[radius=.08, fill=black](2,0)circle;
     \draw[radius=.08, fill=black](2,2)circle;
     \node at (0,-1.25) {\tiny{$k$}};
     \node at (2,-1.25) {\tiny{$l$}};
      \node at (1,-.25) {\tiny{$s$}};
      \node at (1,2.25) {\tiny{$r$}};
       \node at (-.75,1) {\tiny{$k-s$}};
       \node at (2.75,1) {\tiny{$l+s$}};
  \node at (0,3.25) {\tiny{$k+r-s$}};
     \node at (2,3.25) {\tiny{$l-r+s$}};
    \begin{scope}[thick,decoration={
    markings,
    mark=at position 0.5 with {\arrow{>}}}
    ] 
        \draw[postaction={decorate}, thick] (0,-1)--(0,0);
        \draw[postaction={decorate}, thick] (0,0)--(0,2);
        \draw[postaction={decorate}, thick] (0,2)--(0,3);
         \draw[postaction={decorate}, thick] (2,-1)--(2,0);
        \draw[postaction={decorate}, thick] (2,0)--(2,2);
        \draw[postaction={decorate}, thick] (2,2)--(2,3);
        \draw[postaction={decorate}, thick] (0,0)--(2,0);
        \draw[postaction={decorate}, thick] (2,2)--(0,2);
        
        \end{scope}
    \end{tikzpicture}} = \sum_t (-1)^{(k-l+r-s-1)t}\genfrac[]{0pt}{0}{k-l+r-s}{t}_q
        \raisebox{-50pt}{\begin{tikzpicture}[scale=.6]
    \draw[radius=.08, fill=black](0,0)circle;
     \draw[radius=.08, fill=black](0,2)circle;
        \draw[radius=.08, fill=black](2,0)circle;
     \draw[radius=.08, fill=black](2,2)circle;
     \node at (0,-1.25) {\tiny{$k$}};
     \node at (2,-1.25) {\tiny{$l$}};
      \node at (1,-.25) {\tiny{$r-t$}};
      \node at (1,2.25) {\tiny{$s-t$}};
       \node at (-1,1) {\tiny{$k+r-t$}};
       \node at (3,1) {\tiny{$l-r+t$}};
  \node at (0,3.25) {\tiny{$k+r-s$}};
     \node at (2,3.25) {\tiny{$l-r+s$}};
    \begin{scope}[thick,decoration={
    markings,
    mark=at position 0.5 with {\arrow{>}}}
    ] 
        \draw[postaction={decorate}, thick] (0,-1)--(0,0);
        \draw[postaction={decorate}, thick] (0,0)--(0,2);
        \draw[postaction={decorate}, thick] (0,2)--(0,3);
         \draw[postaction={decorate}, thick] (2,-1)--(2,0);
        \draw[postaction={decorate}, thick] (2,0)--(2,2);
        \draw[postaction={decorate}, thick] (2,2)--(2,3);
        \draw[postaction={decorate}, thick] (2,0)--(0,0);
        \draw[postaction={decorate}, thick] (0,2)--(2,2);
        
        \end{scope}
    \end{tikzpicture}}
\end{equation}
lies in $\ker(f)$.
\end{corollary}
\begin{proof}
By induction on $r+s$, using the bigon, $I$=$H$, square removal, and ($r=s=1$) square switch relations of Theorem~\ref{thm:Fontaine relations}. The full inductive computation is in Appendix~\ref{appendix:relation arguments}.
\end{proof}

 \section{Open Questions}
We end this article with a list of open questions, many of which have appeared elsewhere.
\begin{enumerate}
    \item Is there an $\mathfrak{sl}_n$ web basis $\mathcal{B}$ that satisfies the \emph{identifiability property} (there is a simple deterministic test to see if an arbitrary web graph $G$ is in $\mathcal{B}$ or not)?
    \item Is there an $\mathfrak{sl}_n$ web basis $\mathcal{B}$ that satisfies the \emph{decomposition property} (there is a simple deterministic process for any web graph $G$ to find $\mathcal{G} \subseteq \mathcal{B}$ such that the web vectors satisfy $w_G = \sum_{G' \in \mathcal{G}} c_{G'} w_{G'}$ for some scalars $c_{G'} \in \mathbb{C}(q)$)?
    \item \emph{Rotation invariance}: Is there an $\mathfrak{sl}_n$ web basis that is invariant under the natural rotation action on bounded planar graphs?
    \item The \emph{nonvanishing} question: Given an $\mathfrak{sl}_n$ web graph $G$ and a standard basis vector $x_I$ in the tensor product $\otimes_i V_{k_i}$, is there a valid stranding $S$ of $G$ with $x_S = x_I$?  (By Theorem~\ref{thm: nonzero coefficient}, this implies the coefficient of $x_I$ in $w_G$ is nonvanishing.)  Given a web basis $\mathcal{B}$, the nonvanishing question is a classical question in combinatorial representation theory.
    \item The \emph{coefficient question}: Given a web vector $w_G$ and a valid stranding $S$ on $G$, what is a combinatorial formula for the coefficient of $x_S$ in $w_G$?  In other words, identify the $c_{G'}$ in the decomposition property.  Answering this question for each web graph $G$ in a web basis $\mathcal{B}$ is another classical problem in combinatorial representation theory.
    \item The \emph{matching} question: Given an $\mathfrak{sl}_n$ web graph $G$ and a (directed) multicolored noncrossing matching $\mathcal{M}$, is there a valid stranding $S_\mathcal{M}$ of $G$ whose non-closed strands agree with $\mathcal{M}$? Is there an $\mathfrak{sl}_n$ web basis $\mathcal{B}$ that is upper-triangular with respect to the partial order on multicolored noncrossing matchings, in the sense that the basis $\mathcal{B}$ can be indexed by multicolored noncrossing matchings $\mathcal{B} = \Big\{ G_\mathcal{M} = \sum_{\mathcal{M}'} c_{\mathcal{M}'}^{\mathcal{M}} x_{\mathcal{M}'}\Big\}$ so that if $\mathcal{M}' \prec \mathcal{M}$ then coefficient $c_{\mathcal{M}'}^{\mathcal{M}} = 0$ while the coefficient $c_{\mathcal{M}}^{\mathcal{M}} = 1$?
    \item The \emph{relations} question: Theorem~\ref{thm:Fontaine relations} proves a complete set of relations on untagged webs.  Other relations exist, for instance: the Kekule relations from Morrison's thesis \cite{MorrisonThesis}. How do we write these in terms of the relations presented in this paper?  Are there other useful relations?  This is related to a classical problem in commutative algebra.
\end{enumerate}

Of course, each of these questions is interesting in various special cases (for $\mathfrak{sl}_4$ or $\mathfrak{sl}_5$ webs, for certain classes of matchings $\mathcal{M}$ or coefficients $c_\mathcal{I}$, etc.) as well as in full generality.


 \appendix \renewcommand{\thesection}{\Alph{section}} 

 \section{Tagged webs and CKM technical background}\label{appendix:tagged ckm background}
 
This appendix collects the technical CKM machinery referenced in Section~\ref{section:tagless vs tagged} and Appendix~\ref{appendix:invariance calculations}: the Morse-style decomposition of a tagged web, the CKM maps on local pieces, the construction of $g:\mathcal{C}(\vec{k})\to\textup{Inv}(\vec{k})$, and the precise composition of CKM maps corresponding to the tagged web fragments of the tag translation map $\eta$.
 
\subsection{Morse-style decompositions}\label{appendix: morse}
 
A Morse-style decomposition slices a tagged web by horizontal axes so that each slice contains exactly one local piece (cup, cap, $\lambda$-piece, or Y-piece).
 
\begin{definition}\label{def:Morse decomp}
Let $G$ be a plane graph with univalent boundary vertices and trivalent internal vertices, embedded below its boundary axis. A graph $G''$ is a \emph{Morse-style decomposition} of $G$ if $G''$ is isotopic to $G$ relative to the boundary and a collection of horizontal axes can be superimposed on $G''$ so that between each pair of adjacent axes $G''$ consists of vertical segments together with exactly one cup, cap, $\lambda$-piece, or Y-piece. We draw these axes as red dashed lines; points where $G''$ meets these axes are not vertices.
\end{definition}

The cup and cap pieces are as in Figure~\ref{fig:basic pieces}. The two trivalent local pieces, $\lambda$ and Y, are shown in Figure~\ref{fig:morse pieces}.
 
\begin{figure}[h]
\begin{subfigure}[t]{0.3\textwidth}
\centering
\begin{tikzpicture}[scale=.5]
    \draw[red, dashed] (0,0)--(3,0);
    \draw[red, dashed] (0,-2)--(3,-2);
    \draw[radius=.12, fill=black](1.5,-1)circle;
    \draw[thick] (.5,-2)--(1.5,-1)--(2.5,-2);
    \draw[thick](1.5,-1)--(1.5,0);
\end{tikzpicture}
\subcaption{$\lambda$-piece}
\end{subfigure}
\hspace{.5in}
\begin{subfigure}[t]{0.3\textwidth}
\centering
\begin{tikzpicture}[scale=.5]
    \draw[dashed, red] (0,-2)--(3,-2);
    \draw[red, dashed] (0,0)--(3,0);
    \draw[radius=.12, fill=black](1.5,-1)circle;
    \draw[thick] (.5,0)--(1.5,-1)--(2.5,0);
    \draw[thick](1.5,-1)--(1.5,-2);
\end{tikzpicture}
\subcaption{Y-piece}
\end{subfigure}
\caption{Trivalent vertices in a Morse-style decomposition.}\label{fig:morse pieces}
\end{figure}
 
\begin{lemma}\label{lem:tripod to Morse}
Every tripod decomposition admits an associated Morse-style decomposition, obtained by isotoping each tripod into a $\lambda$-piece or Y-piece with two or one auxiliary cups, then vertically adjusting so that each cup, cap, $\lambda$-piece, and Y-piece sits at a distinct critical level.
\end{lemma}
 
\begin{example}\label{ex:tripodtoMorse}
Figure~\ref{fig: ex tripod to Morse} illustrates Lemma~\ref{lem:tripod to Morse}, producing a Morse-style decomposition from the tripod decomposition of Example~\ref{ex:websplit}.
 
 \begin{figure}[h]
\begin{center}
\scalebox{.7}{\begin{tikzpicture}[scale=.65]

\draw[style=dashed, <->] (0,4.5)--(13,4.5);
\draw[style=dashed, red] (0,0)--(13,0);
\draw[style=dashed, red] (0,2.5)--(13,2.5);
\draw[style=dashed, red] (0,1)--(13,1);
\draw[style=dashed, red] (0,1.75)--(13,1.75);
\draw[style=dashed, red] (0,-1)--(13,-1);
\draw[style=dashed, red] (0,-1.75)--(13,-1.75);

\draw[style=dashed, red] (0,-2.75)--(13,-2.75);
\draw[style=dashed, red] (0,-3.5)--(13,-3.5);
\draw[style=dashed, red] (0,-4.5)--(13,-4.5);
\draw[style=dashed, red] (0,-5)--(13,-5);
\draw[style=dashed, red] (0,-5.75)--(13,-5.75);
\draw[style=dashed, red] (0,-6.75)--(13,-6.75);
\draw[style=dashed, red] (0,-7.5)--(13,-7.5);

  \draw[style=thick] (2,0)--(2,2.5) to[out=90,in=180] (5.5, 4) to[out=0,in=90] (9,2.5)--(9,0);
 
 \draw[style=thick] (3,0) to[out=90,in=180] (3.5, .75) to[out=0,in=90] (4,0);
 
  \draw[style=thick] (5,0)--(5,1) to[out=90,in=180] (5.5, 1.5) to[out=0,in=90] (6,1)--(6,0);
  
    \draw[style=thick] (7,0) --(7,1.75) to[out=90,in=180] (7.5, 2.25) to[out=0,in=90] (8,1.75)--(8,0);

\draw[style=thick] (1,4.5)--(1,0)--(1.5,-.5)--(2,0);
\draw[style=thick] (1.5, -.5)--(1.5,-1) to[out=270,in=180] (2.25,-1.5) to[out=0, in=270] (3,-1)--(3,0);

    \draw[style=thick] (4,0)--(4,-6.25)--(3.5,-6.75) to[out=270,in=180] (7.75,-8.25) to[out=0,in=270] (12,-6.75)--(12,4.5);
 
    \draw[style=thick] (4,-6.25)--(4.5, -6.75) to[out=315,in=180] (4.75,-7) to[out=0, in=270] (5,-6.75)--(5,0);

    \draw[style=thick] (6,0)--(6,-4)--(5.5,-4.5) to[out=270,in=180] (8.25,-5.5) to[out=0,in=270] (11,-4.5)--(11,4.5);
    \draw[style=thick] (6,-4)--(6.5, -4.5) to[out=315,in=180] (6.75,-4.75) to[out=0, in=270] (7,-4.5)--(7,0);

        \draw[style=thick] (8,0)--(8,-1.75)--(8.5,-2.25)--(9,-1.75)--(9,0); 
        \draw[style=thick] (8.5, -2.25)--(8.5, -2.75) to[out=270,in=180] (9.25,-3.25) to[out=0, in=270] (10,-2.75)--(10,4.5);

\draw[radius=.08, fill=black](1,4.5)circle;
\draw[radius=.08, fill=black](1.5,-.5)circle;
\draw[radius=.08, fill=black](4,-6.25)circle;
\draw[radius=.08, fill=black](8.5,-2.25)circle;
\draw[radius=.08, fill=black](6,-4)circle;
\draw[radius=.08, fill=black](10,4.5)circle;
\draw[radius=.08, fill=black](11,4.5)circle;
\draw[radius=.08, fill=black](12,4.5)circle;
  
    \end{tikzpicture}}
    \caption{A Morse-style decomposition for a trivalent graph with boundary}\label{fig: ex tripod to Morse}
        \end{center}
    \end{figure}
\end{example}
 
\subsection{CKM maps}\label{appendix: ckm maps}
 
In the CKM construction, each elementary piece in the Morse-style decomposition of a tagged web prescribes one of the equivariant maps in Figure~\ref{fig:CKMmaps} \cite{CKM}. We call these CKM maps. Cups and caps correspond to co-evaluation and evaluation, a $\lambda$-piece to multiplication, a Y-piece to co-multiplication. Tags induce maps between $V_k$ and $V_{n-k}^*$, with the side of the edge determining a sign.

Coefficients for the CKM maps use the following function on binary vectors $\vec{b}=b_1\cdots b_n$ and $\vec{b}'=b'_1\cdots b'_n$.
$$\ell(\vec{b},\vec{b}') = |\{i<j : b_i = b'_j = 1\}|$$

\begin{figure}[!htbp]
\begin{center}
\scalebox{.75}{\begin{tabular}{ | c | c | c | } 
  \hline
  && \\
  Web & Map & CKM Map \\
  && \\
  \hline
    \begin{tikzpicture}[scale=.65]   
    \draw[radius=.08, fill=black](1,0)circle;
    \draw[radius=.08, fill=black](1,1)circle;
    \draw[radius=.08, fill=black](0,2)circle;
    \draw[radius=.08, fill=black](2,2)circle;

    \begin{scope}[thick,decoration={
        markings,
        mark=at position 0.5 with {\arrow{>}}}
        ] 
       \draw[postaction={decorate}] (1,0)--(1,1);
       \draw[postaction={decorate}] (1,1)--(0,2);
       \draw[postaction={decorate}] (1,1)--(2,2);

    \end{scope}

     \node at (1,-.5) {\small{\textcolor{black}{$k+l$}}};

   \node at (0,2.5) {\small{\textcolor{black}{$k$}}};

   \node at (2,2.5) {\small{\textcolor{black}{$l$}}};

    \end{tikzpicture}
   & \raisebox{.5in}{$M'_{k,l}: V_{k+l} \rightarrow V_{k}\otimes V_{l}$} & \raisebox{.5in}{$M'_{k, l}(x_{\vec{b}})=(-1)^{kl} \displaystyle{\sum_{{\vec{b}_1+\vec{b}_2=\vec{b}}\atop {|\vec{b}_1|=k, |\vec{b}_2|=l}}} (-q)^{-\ell(\vec{b}_2, \vec{b_1})} x_{\vec{b}_1} \otimes x_{\vec{b}_2}$} 
    \\ 
  \hline
  \begin{tikzpicture}[scale=.65]

    \draw[radius=.08, fill=black](0,0)circle;
    \draw[radius=.08, fill=black](2,0)circle;
    \draw[radius=.08, fill=black](1,1)circle;
    \draw[radius=.08, fill=black](1,2)circle;

    \begin{scope}[thick,decoration={
        markings,
        mark=at position 0.5 with {\arrow{>}}}
        ] 
       \draw[postaction={decorate}] (0,0)--(1,1);
       \draw[postaction={decorate}] (2,0)--(1,1);
       \draw[postaction={decorate}] (1,1)--(1,2);

    \end{scope}

    \node at (0,-0.3) {\small{\textcolor{black}{$k$}}};

    \node at (2,-0.3) {\small{\textcolor{black}{$l$}}};

    \node at (1,2.5) {\small{\textcolor{black}{$k+l$}}};

    \end{tikzpicture}
   & \raisebox{.5in}{$M_{k,l}:  V_{k}\otimes V_{l} \rightarrow V_{k+l} $} & \raisebox{.5in}{$M_{k, l}(x_{\vec{b}_1} \otimes x_{\vec{b}_2}) = 
    \begin{cases} 
      (-q)^{\ell(\vec{b}_1,\vec{b}_2)} x_{\vec{b}_1+\vec{b}_2} & \vec{b}_1 \cdot \vec{b}_2=0 \\
      0 & \text{otherwise}
   \end{cases}$}
   \\ 
  \hline
  \begin{tikzpicture}[scale=.65]

    \draw[radius=.08, fill=black](1,0)circle;
    \draw[radius=.08, fill=black](1,2)circle;
    \draw[thick] (0.8,1)--(1,1);

    \begin{scope}[thick,decoration={
        markings,
        mark=at position 0.5 with {\arrow{>}}}
        ] 
       \draw[postaction={decorate}] (1,0)--(1,1);
       \draw[postaction={decorate}] (1,2)--(1,1);

    \end{scope}

    \node at (1,-.5) {\small{\textcolor{black}{$k$}}};

    \node at (1,2.5) {\small{\textcolor{black}{$n-k$}}};

    \end{tikzpicture}
    & \raisebox{.5in}{$D_k: V_k \rightarrow (V_{n-k})^*$}
    & \raisebox{.5in}{$D_k(x_{\vec{b}}) = (-q)^{\ell(\vec{b},\vec{1}-\vec{b})}x_{\vec{1}-\vec{b}}^*$} \\
\hline
  \begin{tikzpicture}[scale=.65]

    \draw[radius=.08, fill=black](1,0)circle;
    \draw[radius=.08, fill=black](1,2)circle;
    \draw[thick] (1,1)--(1.2,1);

    \begin{scope}[thick,decoration={
        markings,
        mark=at position 0.5 with {\arrow{>}}}
        ] 
       \draw[postaction={decorate}] (1,0)--(1,1);
       \draw[postaction={decorate}] (1,2)--(1,1);

    \end{scope}

    \node at (1,-.5) {\small{\textcolor{black}{$k$}}};

    \node at (1,2.5) {\small{\textcolor{black}{$n-k$}}};

    \end{tikzpicture}
  & \raisebox{.5in}{$(-1)^{k(n-k)}D_k: V_k \rightarrow (V_{n-k})^*$}
    & \raisebox{.5in}{$(-1)^{k(n-k)}D_k(x_{\vec{b}})$} \\
\hline
  \begin{tikzpicture}[scale=.65]

    \draw[radius=.08, fill=black](1,0)circle;
    \draw[radius=.08, fill=black](1,2)circle;
    \draw[thick] (1,1)--(1.2,1);

    \begin{scope}[thick,decoration={
        markings,
        mark=at position 0.5 with {\arrow{>}}}
        ] 
       \draw[postaction={decorate}] (1,1)--(1,0);
       \draw[postaction={decorate}] (1,1)--(1,2);

    \end{scope}

     \node at (1,-.5) {\small{\textcolor{black}{$k$}}};

    \node at (1,2.5) {\small{\textcolor{black}{$n-k$}}};

    \end{tikzpicture}
    & \raisebox{.5in}{$D_{n-k}^{-1}: (V_k)^* \rightarrow V_{n-k}$}
    & \raisebox{.5in}{$D_{n-k}^{-1}(x_{\vec{b}}^*) = (-q)^{-\ell(\vec{1}-\vec{b},\vec{b})}x_{\vec{1}-\vec{b}}$} \\
\hline
  \begin{tikzpicture}[scale=.65]

    \draw[radius=.08, fill=black](1,0)circle;
    \draw[radius=.08, fill=black](1,2)circle;
    \draw[thick] (0.8,1)--(1,1);

    \begin{scope}[thick,decoration={
        markings,
        mark=at position 0.5 with {\arrow{>}}}
        ] 
       \draw[postaction={decorate}] (1,1)--(1,0);
       \draw[postaction={decorate}] (1,1)--(1,2);

    \end{scope}
      \node at (1,-.5) {\small{\textcolor{black}{$k$}}};

    \node at (1,2.5) {\small{\textcolor{black}{$n-k$}}};

    \end{tikzpicture}
    & \raisebox{.5in}{$(-1)^{k(n-k)}D_{n-k}^{-1}: (V_k)^* \rightarrow V_{n-k}$}
    & \raisebox{.5in}{$(-1)^{k(n-k)}D_{n-k}^{-1}(x_{\vec{b}}^*)$} \\
\hline
  \begin{tikzpicture}[scale=.65]

    \draw[radius=.08, fill=black](0,0)circle;
    \draw[radius=.08, fill=black](2,0)circle;

    \begin{scope}[thick,decoration={
        markings,
        mark=at position 0.5 with {\arrow{>}}}
        ] 
       \draw[postaction={decorate}] (2,0) to[out=90,in=0] (1,1) to[out=180,in=90] (0,0);

    \end{scope}

    \node at (1,1.4) {\small{\textcolor{black}{$k$}}};

    \end{tikzpicture}
    & \raisebox{.2in}{$C^{L,k}:(V_k)^*\otimes V_k \rightarrow \mathbb{C}(q)$}
    & \raisebox{.2in}{$C^{L,k}(x_{\vec{b}_1}^*\otimes x_{\vec{b}_2})=
    \begin{cases} 
      1 & \vec{b}_1=\vec{b}_2 \\
      0 & \text{otherwise}
   \end{cases}$} \\
\hline
&&\\
\begin{tikzpicture}[scale=.65]   

    \draw[radius=.08, fill=black](0,0)circle;
    \draw[radius=.08, fill=black](2,0)circle;

    \begin{scope}[thick,decoration={
        markings,
        mark=at position 0.5 with {\arrow{>}}}
        ] 
       \draw[postaction={decorate}] (0,0) to[out=270,in=180] (1,-1) to[out=0,in=270] (2,0);

    \end{scope}

    \node at (1,-1.4) {\small{\textcolor{black}{$k$}}};

    \end{tikzpicture}
    & \raisebox{.3in}{$C_{R,k}:\mathbb{C}(q)\rightarrow (V_k)^*\otimes V_k$}
    & \raisebox{.3in}{$C_{R,k}(1)=\displaystyle{\sum_{|\vec{b}| = k}} q^{k(n-k) - 2\ell(\vec{b}, \vec{1}-\vec{b})}x_{\vec{b}}^*\otimes x_{\vec{b}}$} \\
\hline
  \begin{tikzpicture}[scale=.65]

    \draw[radius=.08, fill=black](0,0)circle;
    \draw[radius=.08, fill=black](2,0)circle;

    \begin{scope}[thick,decoration={
        markings,
        mark=at position 0.5 with {\arrow{>}}}
        ] 
       \draw[postaction={decorate}] (0,0) to[out=90,in=180] (1,1) to[out=0,in=90] (2,0);

    \end{scope}

    \node at (1,1.4) {\small{\textcolor{black}{$k$}}};

    \end{tikzpicture}
    & \raisebox{.2in}{$C^{R,k}:V_k\otimes (V_k)^* \rightarrow \mathbb{C}(q)$}
    & \raisebox{.2in}{$C^{R,k}(x_{\vec{b}_1}\otimes x_{\vec{b}_2}^*)=
    \begin{cases} 
      q^{2\ell(\vec{b}_1,\vec{1}-\vec{b}_1)-k(n-k)} & \vec{b}_1=\vec{b}_2 \\
      0 & \text{otherwise}
   \end{cases}$} \\
\hline
&&\\
\begin{tikzpicture}[scale=.65]   

    \draw[radius=.08, fill=black](0,0)circle;
    \draw[radius=.08, fill=black](2,0)circle;

    \begin{scope}[thick,decoration={
        markings,
        mark=at position 0.5 with {\arrow{>}}}
        ] 
       \draw[postaction={decorate}] (2,0) to[out=270,in=0] (1,-1) to[out=180,in=270] (0,0);

    \end{scope}

    \node at (1,-1.4) {\small{\textcolor{black}{$k$}}};

    \end{tikzpicture}
    & \raisebox{.3in}{$C_{L,k}:\mathbb{C}(q)\rightarrow V_k\otimes (V_k)^*$}
    & \raisebox{.3in}{$C_{L,k}(1)=\displaystyle{\sum_{|\vec{b}| = k}} x_{\vec{b}}\otimes x_{\vec{b}}^*$} \\
\hline

\end{tabular}}
\end{center}
\caption{Maps from CKM webs to maps of representations, where for each pair $\vec{b}=b_{1}\cdots b_{n}$ and $\vec{b}'=b'_1\cdots b'_n$ the function $\ell(\vec{b}, \vec{b}')=|\{i<j: b_i=b'_j=1\}|$}\label{fig:CKMmaps}
\end{figure}

\begin{definition}\label{def: g via morse}
For $H\in C(\vec{k})$, choose a Morse-style decomposition $H''$ of $H$ in which:
\begin{itemize}
\item split and merge vertices appear as Y-pieces and $\lambda$-pieces respectively,
\item at each Y-piece and $\lambda$-piece the edges are oriented upward, and
\item tags occur on vertical segments at heights distinct from the Y-pieces and $\lambda$-pieces.
\end{itemize}
Then $g(H)\in\textup{Inv}(\vec{k})$ is the composition of the CKM maps from Figure~\ref{fig:CKMmaps} prescribed by the critical levels of $H''$, applied to $1\in\mathbb{C}(q)$.
\end{definition}
 
\begin{theorem}[Cautis-Kamnitzer-Morrison \cite{CKM}]\label{thm: g well-defined and surjective}
The map $g:\mathcal{C}(\vec{k})\to\textup{Inv}(\vec{k})$ is well-defined, $\uq$-equivariant, and surjective. Its kernel is generated by the relations of Figures~\ref{fig:tagrelations} and~\ref{fig:CKM relations}.
\end{theorem}
 
\begin{example}\label{ex:taggedex}
Figure~\ref{fig:taggedex} gives a Morse-style decomposition of a tagged tripod web and the corresponding composition of CKM maps.

\begin{figure}[h]
\raisebox{50pt}{\begin{large}$H'':$\end{large}}\hspace{.5in}
      \begin{tikzpicture}

        \draw[dashed, <->] (0,0)--(6,0);
        \draw[red, dashed] (0,-1.5)--(6,-1.5);
        \draw[red, dashed] (0,-2.5)--(6,-2.5);

        \node at (5,-3) {$C_{L, k+l}$};
        \node at (5, -2) {$Id\otimes D_{n-k-l}^{-1}$};
        \node at (5, -1) {$M_{k,l}'\otimes Id$};
        \draw[radius=.08, fill=black](1.5,-1)circle;
        \draw[radius=.08, fill=black](2,0)circle;
        \draw[radius=.08, fill=black](1,0)circle;
         \draw[radius=.08, fill=black](3,0)circle;
\begin{scope}[thick,decoration={
    markings,
    mark=at position 0.5 with {\arrow{>}}}
    ] 
\draw[postaction={decorate}, thick] (1.5,-1)--(1,0);
\draw[postaction={decorate}, thick] (1.5,-1)--(2,0);
\end{scope}
\begin{scope}[thick,decoration={
    markings,
    mark=at position 0.33 with {\arrow{<}}}
    ] 
\draw[postaction={decorate}, thick] (1.5,-1)--(1.5,-2.25); 
\draw[postaction={decorate}, thick] (1.5,-2.25) to[out=270,in=180](2.25, -3) to[out=0,in=270] (3,-2.25);
    \end{scope}
    \begin{scope}[thick,decoration={
    markings,
    mark=at position 0.7 with {\arrow{>}}}
    ] 
\draw[postaction={decorate}, thick] (3,-2.25) to[out=90,in=270](3,0);
    \end{scope}
    \node at (.9,-.5) {\tiny{$k$}};
     \node at (1.9,-.5) {\tiny{$l$}};
     \node at (1.9,-2) {\tiny{$k+l$}};
       \node at (3.7,-.5) {\tiny{$n-k-l$}};
 \draw[style=thick] (3,-2)--(3.2,-2);
    \end{tikzpicture}

     $$g(H'')=((M_{k,l}'\otimes Id)\;\circ (Id \; \otimes D_{n-k-l}^{-1})\; \circ (C_{L,k+l}))(1)\in \textup{Inv}(k,l,n-k-l)$$ 
    \caption{Splitting a tagged web into critical levels to compute its invariant vector}\label{fig:taggedex}
\end{figure}
\end{example}

\subsection{Tagged web fragments from the map $\eta$}\label{appendix: tagged fragments}

Below, we record the precise composition of CKM maps corresponding to each of the tagged web fragments from Figure \ref{fig:replaceCKM}. In Appendix~\ref{appendix:invariance calculations}, we show these composite maps align (up to a predictable sign) with the invariants/equivariants given by cups, caps, and tripods in our untagged framework. 
 
\begin{lemma}\label{lem:g on local pieces}
The map $g$ sends each of the following tagged web fragments to the indicated composition of CKM maps.

\smallskip
\noindent\textbf{Cup:}
\[\vcenter{\hbox{\begin{tikzpicture}[scale=.8]
\draw[dashed, red, <->] (0,0)--(3,0);
\node at(.5,-.65) {\tiny{$n-k$}};
\node at(2.2,-.65) {\tiny{$k$}};
\draw[style=thick] (1.5,-1)--(1.5,-1.25);
\begin{scope}[thick,decoration={
    markings,
    mark=at position 0.5 with {\arrow{>}}}]
\draw[postaction={decorate}, thick] (1.5,-1) to[out=0,in=270] (2,0);
\draw[postaction={decorate}, thick] (1.5,-1) to[out=180,in=270] (1,0);
\end{scope}
\end{tikzpicture}}} \quad\stackrel{g}{\longmapsto}\quad \bigl((Id\otimes D^{-1}_{k})\circ C_{L,n-k}\bigr)(1).\]

\smallskip
\noindent\textbf{Split vertex} $(k+l+m=n)$:
\[\vcenter{\hbox{\begin{tikzpicture}[scale=.8]
\draw[dashed, red, <->] (0,0)--(4,0);
\draw[radius=.08, fill=black](2,-1)circle;
\begin{scope}[thick,decoration={
    markings,
    mark=at position 0.5 with {\arrow{>}}}]
\draw[postaction={decorate}, thick] (2,-1) to[out=180,in=270] (1,0);
\draw[postaction={decorate}, thick] (2,-1)--(2,0);
\end{scope}
\begin{scope}[thick,decoration={
    markings,
    mark=at position 0.33 with {\arrow{<}}}]
\draw[postaction={decorate}, thick] (2,-1) to[out=0,in=270](3,0);
\end{scope}
\begin{scope}[thick,decoration={
    markings,
    mark=at position 0.7 with {\arrow{>}}}]
\draw[postaction={decorate}, thick] (2,-1) to[out=0,in=270](3,0);
\end{scope}
\node at (.9,-.5) {\tiny{$k$}};
\node at (1.7,-.5) {\tiny{$l$}};
\node at (2.4,-1.25) {\tiny{$n-m$}};
\node at (3.2,-.5) {\tiny{$m$}};
\draw[style=thick] (2.65,-.75)--(2.8,-.9);
\end{tikzpicture}}} \quad\stackrel{g}{\longmapsto}\quad \bigl((M'_{k,l}\otimes Id)\circ (Id\otimes D^{-1}_{m})\circ C_{L,k+l}\bigr)(1).\]

\smallskip
\noindent\textbf{Merge vertex} $(k+l+m=2n)$:
\begin{multline*}
\vcenter{\hbox{\begin{tikzpicture}[scale=.8]
\draw[dashed, red, <->] (0,0)--(4,0);
\draw[radius=.08, fill=black](2,-1)circle;
\begin{scope}[thick,decoration={
    markings,
    mark=at position 0.5 with {\arrow{>}}}]
\draw[postaction={decorate}, thick] (2,-1) to[out=180,in=270] (1,0);
\end{scope}
\begin{scope}[thick,decoration={
    markings,
    mark=at position 0.3 with {\arrow{<}}}]
\draw[postaction={decorate}, thick] (2,-1) to[out=0,in=270](3,0);
\draw[postaction={decorate}, thick] (2,-1)--(2,0);
\end{scope}
\begin{scope}[thick,decoration={
    markings,
    mark=at position 0.7 with {\arrow{>}}}]
\draw[postaction={decorate}, thick] (2,-1) to[out=0,in=270](3,0);
\draw[postaction={decorate}, thick] (2,-1)--(2,0);
\end{scope}
\node at (.6,-.3) {\tiny{$k$}};
\node at (1.55,-.25) {\tiny{$l$}};
\node at (3.3,-.3) {\tiny{$m$}};
\node at (1.55,-.55) {\tiny{$n-l$}};
\node at (2.4,-1.25) {\tiny{$n-m$}};
\draw[style=thick] (2.65,-.75)--(2.8,-.9);
\draw[style=thick] (2,-.55)--(2.2,-.55);
\end{tikzpicture}}} \quad\stackrel{g}{\longmapsto}\quad \bigl((M_{n-m,n-l}\otimes Id^{\otimes 2})\circ (Id^{\otimes 2}\otimes D^{-1}_l\otimes Id) \\
\circ (Id\otimes C_{L,n-l}\otimes Id)\circ (Id\otimes D^{-1}_m)\circ C_{L,n-m}\bigr)(1).
\end{multline*}

\smallskip
\noindent\textbf{Cap:}
\[\vcenter{\hbox{\begin{tikzpicture}[scale=.8]
\draw[dashed, red, <->] (0,0)--(3,0);
\begin{scope}[thick,decoration={
    markings,
    mark=at position 0.5 with {\arrow{>}}}]
\draw[postaction={decorate}, thick] (1,0) to[out=90,in=180] (1.5,1);
\draw[postaction={decorate}, thick] (2,0) to[out=90,in=0] (1.5,1);
\end{scope}
\node at(.8,.65) {\tiny{$k$}};
\node at(2.5,.65) {\tiny{$n-k$}};
\draw[style=thick] (1.5,1)--(1.5,1.25);
\end{tikzpicture}}} \quad\stackrel{g}{\longmapsto}\quad C^{L,n-k}\circ (D_k\otimes Id).\]
\end{lemma} 
\section{Local invariance calculations}\label{appendix:invariance calculations}
 
We prove Lemmas~\ref{lem:cup invariant},~\ref{lem:tripod invariant},  and ~\ref{lem:cap equivariant} by showing the composite maps in Lemma \ref{lem:g on local pieces} agree with the map $f$ on untagged webs with at most one vertex.

\begin{proof}[Proof of Lemma~\ref{lem:cup invariant}]
Say $G$ has vertices $v_1$ and $v_2$ ordered left to right with edge $e=v_1\stackrel{k}{\mapsto} v_2$. By Lemma~\ref{lem:invt and flips}, it suffices to show $f(G)\in\textup{Inv}(n-k,k)$.
 
Fix $1\leq i\leq n-1$. Since there are no closed strands,
    $$f(G) = \sum_{S\in\mathcal{S}tr(G)} (-q)^{-y(S)}x_S
    =\sum_{|\vec{b}|=k} (-q)^{-\ell(\vec{b}, \vec{1}-\vec{b})}x_{\vec{1}-\vec{b}}\otimes x_{\vec{b}}$$
where the second equality uses Corollary~\ref{cor:strands and binary labeling} (strands appear exactly when two entries of a binary label differ) together with the observation that each instance of a $1$ before a $0$ in $\vec{b}$ contributes a counterclockwise strand.

On the other hand, using the CKM maps from Figure \ref{fig:CKMmaps}, we have
\begin{align*}
    ((Id \otimes D_k^{-1})\circ C_{L,n-k})(1) &=(Id \otimes D_k^{-1})\left(\sum_{|\vec{b}|=k} x_{\vec{1}-\vec{b}}\otimes x_{\vec{1}-\vec{b}}^*\right) \\
    &=\sum_{|\vec{b}|=k} (-q)^{-\ell(\vec{b}, \vec{1}-\vec{b})}  x_{\vec{1}-\vec{b}}\otimes x_{\vec{b}}.  
\end{align*}
Since $f(G)$ is the image of the composition of CKM maps, it is $\uq$-invariant.
\end{proof}

The next lemma uses the structural properties of stranding at a vertex to put a condition on the binary labels that occur at that vertex, depending on whether it is a split or merge. 

\begin{lemma} \label{lemma: sink vertex split or merge and binary}
Suppose $S$ is a valid stranding on a web graph $G$ and $v$ is an interior vertex of $G$ incident to edges $e_1, e_2, e_3$ all directed out of $v$. Then the sum 
\[b_S(e_1)_j + b_S(e_2)_j+ b_S(e_3)_j\]
is the same for all $1\leq j\leq n$.  Moreover, this sum is $1$ if $v$ is a split vertex and $2$ if $v$ is a merge vertex.
\end{lemma}

\begin{proof}
Since $b_S$ is a binary labeling, the definition states that the vector $b_S(e_1)+b_S(e_2)+b_S(e_3) \in \mathbb{Z}^n$ is a nonnegative integer multiple of $\vec{1}$.  This proves the first part. The number of ones in each $b_S(e_i)$ is the same as the weight on edge $e_i$.  Weights are in the set $\{1, 2, \ldots, n-1\}$ so the sum of the entries in $b_S(e_1)+b_S(e_2)+b_S(e_3)$ is nonzero and at most $3n-1$.  Thus the vector sum is either $\vec{1}$ or $2 \cdot \vec{1}$ and the final claim follows by definition of split and merge vertices.
\end{proof}

\begin{proof}[Proof of Lemma~\ref{lem:tripod invariant}]
Let $G$ have boundary vertices $v_1, v_2, v_3$ (left to right) and interior vertex $v$. By Lemma~\ref{lem:invt and flips}, we may assume $G$ has edge set $\{e_1=v\stackrel{k}{\mapsto} v_1,\ e_2=v\stackrel{l}{\mapsto}v_2,\ e_3=v\stackrel{m}{\mapsto} v_3\}$. Fix $S\in\mathcal{S}tr(G)$ with associated binary labeling $b_S$, and write $\vec{b}_i=b_S(e_i)$. There are no closed faces in $G$, so $x(S)=0$. By Lemma~\ref{lemma: sink vertex split or merge and binary}, $\vec{b}_1 + \vec{b}_2 + \vec{b}_3$ equals $\vec{1}$ if $v$ is a split vertex and $\vec{2}$ if $v$ is a merge vertex.
 
If $v$ is a split vertex, $\vec{b}_3=\vec{1}-\vec{b}_1-\vec{b}_2$ and $\vec{b}_1\cdot \vec{b}_2=0$. Each counterclockwise strand is counted in exactly one of
\[\ell(\vec{b}_2, \vec{b}_1), \hspace{0.25in} \ell(\vec{1}-\vec{b}_1-\vec{b}_2, \vec{b}_1), \hspace{0.25in} \ell(\vec{1}-\vec{b}_1-\vec{b}_2, \vec{b}_2),\]
depending on which two edges it runs along, so $y(S)$ is their sum and
$$f(G)=\sum_{\substack{|\vec{b}_1|=k,\ |\vec{b}_2|=l\\\vec{b}_1\cdot \vec{b}_2=0}} (-q)^{-\ell(\vec{b}_2, \vec{b}_1)-\ell(\vec{1}-\vec{b}_1-\vec{b}_2, \vec{b}_1)-\ell(\vec{1}-\vec{b}_1-\vec{b}_2, \vec{b}_2)}\,x_{\vec{b}_1}\otimes x_{\vec{b}_2}\otimes x_{\vec{1}-\vec{b}_1-\vec{b}_2}.$$
Expanding the CKM maps of Figure~\ref{fig:CKMmaps}, we can show:
$$f(G)=(-1)^{kl}\bigl((M'_{k, l}\otimes Id)\circ (Id \otimes D^{-1}_{n-k-l})\circ C_{L, k+l}\bigr)(1)\in\textup{Inv}(k,l,m).$$
 
If $v$ is a merge vertex, $\vec{1}-\vec{b}_1=\vec{b}_2+\vec{b}_3-\vec{1}$ and $(\vec{1}-\vec{b}_2)\cdot (\vec{1}-\vec{b}_3)=0$. Each counterclockwise strand is counted in exactly one of
\[\ell(\vec{b}_2+\vec{b}_3-\vec{1}, \vec{1}-\vec{b}_2), \hspace{0.25in} \ell(\vec{b}_2+\vec{b}_3-\vec{1}, \vec{1}-\vec{b}_3), \hspace{0.25in} \ell(\vec{1}-\vec{b}_2, \vec{1}-\vec{b}_3).\]
Therefore, we have
$$f(G)=\hspace{-.35in}\sum_{\substack{|\vec{b}_2|=l,\ |\vec{b}_3|=m\\(\vec{1}-\vec{b}_2)\cdot (\vec{1}-\vec{b}_3)=0}}\hspace{-.25in}(-q)^{-\ell(\vec{b}_2+\vec{b}_3-\vec{1}, \vec{1}-\vec{b}_2)-\ell(\vec{b}_2+\vec{b}_3-\vec{1}, \vec{1}-\vec{b}_3)-\ell(\vec{1}-\vec{b}_2, \vec{1}-\vec{b}_3)}\,x_{\vec{2}-\vec{b}_2-\vec{b}_3}\otimes x_{\vec{b}_2}\otimes x_{\vec{b}_3},$$
Once again, expanding the CKM maps of Figure~\ref{fig:CKMmaps}, we can show:
\begin{align*}
    f(G)&=\bigl((M_{n-m,n-l}\otimes Id^{\otimes 2})\circ (Id^{\otimes 2} \otimes D^{-1}_{l}\otimes Id) \circ(Id\otimes C_{L,n-l}\otimes Id) \\
    &\hspace{.5in}\circ (Id \otimes D^{-1}_{m})\circ C_{L,n-m}\bigr)(1)\in\textup{Inv}(k,l,m).
\end{align*}
CKM maps are $\uq$-equivariant, so $f(G)$ is a $\uq$-invariant vector for both split and merge tripods.
\end{proof}
 
\begin{proof}[Proof of Lemma~\ref{lem:cap equivariant}]
Suppose $x_{\vec{b}_1}\otimes x_{\vec{b}_2}\in V_k\otimes V_{n-k}$ induces a stranding of the cap. By Lemma~\ref{lem: strand graph}, $S_{\vec{b}_1}$ has a clockwise $(i,j)$ strand exactly when the $i^{th}$ and $j^{th}$ entries of $\vec{b}_1$ are $1$ and $0$, so $z(S_{\vec{b}_1})=\ell(\vec{b}_1, \vec{1}-\vec{b}_1)$. Expanding $C^{L,n-k}\circ (D_k\otimes Id)$ via the CKM maps of Figure~\ref{fig:CKMmaps}, we can show
$$C_k=C^{L,n-k}\circ (D_k\otimes Id),$$
which is $\uq$-equivariant.
\end{proof} 
 \section{Relation arguments}\label{appendix:relation arguments}

\subsection{Tag migration and forgetting tags}

\begin{proof}[Proof of Lemma~\ref{lem:pi map}]
    First, say $[H]=\pm [H']$.  We can observe from Figure \ref{fig:tagrelations} that $\pi$ is the same on any pair of graphs that differ locally by CKM tag relations. Hence $\pi(H)=\pi(H')$.
    
    Now assume $\pi(H)=\pi(H')$. It is sufficient to consider connected graphs, so say $H$ and $H'$, which have the same underlying graph, are connected. A straightforward calculation shows if $H$ and $H'$ have no internal vertices (i.e. $H$ and $H'$ are either circles or cups) they differ up to sign by tag switching and cancellation, so $[H]=\pm [H']$. Next, say $H$ and $H'$ have internal vertices.   
    Because we are showing $[H]=\pm[H']$, we can ignore the side on which each tag appears. Implicitly applying tag switches as needed, we will show that $H$ can be transformed to $H'$ via tag cancellation and migration relations around each internal vertex. 
    
    Let $v$ be an internal vertex of the underlying graph common to $H$ and $H'$ with adjacent edges $e_1, e_2, e_3$. We will say $e_i$ is oriented into $v$ in $H$ if the entire edge (when there are no tags) or the portion of the edge adjacent to $v$ (when there are tags) corresponding to $e_i$ in $H$ points into $v$.  In this case, we write  $\sigma_v(e_i)=1$; otherwise, write $\sigma_v(e_i)=-1$. Let $\ell_i$ be the weight of the (possibly tagged) edge corresponding to $e_i$ at $v$. Define the quantities $\sigma'_v(e_i)$ and $\ell'_i$ analogously for the graph $H'$.  Since $\pi(H)=\pi(H')$, it follows that $\sigma_v(e_i)=\sigma'_v(e_i)$ if and only if $\ell_i=\ell'_i$. Moreover, whenever $\ell_i\neq \ell'_i$ we have $\ell_i=n-\ell'_i$. 
    
    The tagged web condition at trivalent vertices guarantees that  $$\sum_{i=1}^3\sigma_v(e_i)\ell_i=\sum_{i=1}^3\sigma_v'(e_i)\ell'_i=0 $$ where $\{\sigma_v(e_i):1\leq i\leq 3\}=\{\sigma'_v(e_i):1\leq i\leq 3\}=\{\pm1\}$. This means $\sigma_v(e_i)=\sigma'_v(e_i)$ for all $i$ or exactly one $i$. If $\sigma_v(e_i)=\sigma'_v(e_i)$ for all $i$, we will say that $H$ and $H'$ agree at $v$, and we do nothing. On the other hand, say without loss of generality $\sigma_v(e_i)=-\sigma'_v(e_i)$ for $i=1,2$ and $\sigma_v(e_3)=\sigma'_v(e_3)$. We also must have $\sigma_v(e_1)=-\sigma_v(e_2)$ and $\sigma'_v(e_1)=-\sigma'_v(e_2)$. The image on the left in Figure \ref{fig:agreement proof} illustrates this setup. (We omit the orientation of the third edge as well as edge labels for simplicity.)

    \begin{figure}[h]

    \begin{center}
    \begin{tikzpicture}[scale=.65]
\begin{scope}[thick,decoration={
    markings,
    mark=at position 0.5 with {\arrow{>}}}
    ] 

\draw[style=thick] (0,0)--(0,1);
\draw[style=thick, postaction={decorate}] (-2,-2)--(0,0);
\draw[style=thick,postaction={decorate}] (0,0)--(2,-2);

  \draw[radius=.08, fill=black](0,0)circle;
  \node at (-.5, .25) {\tiny{$v$}};
  \node at (0, 1.25) {\tiny{$e_3$}};
    \node at (-2.25, -2.25) {\tiny{$e_1$}};
       \node at (2.25, -2.25) {\tiny{$e_2$}};

   \draw[radius=.08, fill=black](6,0)circle;
  \node at (5.5, .25) {\tiny{$v$}};
    \node at (6, 1.25) {\tiny{$e_3$}};
    \node at (3.75, -2.25) {\tiny{$e_1$}};
       \node at (8.25, -2.25) {\tiny{$e_2$}};

   \draw[radius=.08, fill=black](12,0)circle;
  \node at (11.5, .25) {\tiny{$v$}};
     \node at (12, 1.25) {\tiny{$e_3$}};
    \node at (9.75, -2.25) {\tiny{$e_1$}};
       \node at (14.25, -2.25) {\tiny{$e_2$}};

\draw[style=thick] (6,1)--(6,0);
\draw[style=thick, postaction={decorate}] (4,-2)--(6,0);
\draw[style=thick, postaction={decorate}] (6,0)--(6.5,-.5);
\draw[style=thick,postaction={decorate}] (7.5,-1.5)--(6.5,-.5);
\draw[style=thick,postaction={decorate}] (7.5,-1.5)--(8,-2);

\draw[style=thick] (12,1)--(12,0);
\draw[style=thick, postaction={decorate}] (10,-2)--(11,-1);
\draw[style=thick, postaction={decorate}] (12,0)--(11,-1);
\draw[style=thick, postaction={decorate}] (13,-1)--(12,0);
\draw[style=thick,postaction={decorate}] (13,-1)--(14,-2);

\end{scope}

  \draw[style=thick] (6.5,-.5)--(6.3,-.7);
  \draw[style=thick] (7.7,-1.3)--(7.5,-1.5);
  \draw[style=thick] (11.2,-1.2)--(11,-1);
  \draw[style=thick] (13.2,-.8)--(13,-1);
\node at (3,0) {$\rightarrow$};
\node at (9,0) {$\rightarrow$};
\end{tikzpicture}
\end{center}
\caption{Locally modifying $H$ to agree with $H'$.}\label{fig:agreement proof}
    \end{figure}

As Figure \ref{fig:agreement proof} shows, we can introduce a pair of tags in $H$ using cancellation and move one tag using migration to form a new web that agrees with $H'$ at $v$. Let $\overline{H}$ be the CKM web obtained from $H$ by performing this process at all internal vertices. Then $[H]=\pm[\overline{H}]$. At every vertex, including boundary vertices, $\overline{H}$ and $H'$ have edges that agree in orientation and label but the number of tags within their edges may differ. Since the edge orientations and labels agree at both endpoints, though, for each edge the numbers of tags in $\overline{H}$ and $H'$ have the same parity. Hence $\overline{H}$ and $H'$ differ by tag cancellation and $[\overline{H}]=\pm[H']$ so that we have $[H]=\pm[H']$ as desired. 
\end{proof}

\subsection{Relations via stranding bijections}

In the remainder of this appendix, we use the interplay between stranding and binary labeling to prove local relations on untagged webs. We start with the following definition. Recall that $\sigma_v(e_i)=1$ if and only if $e_i$ is directed into $v$ respectively $-1$ out of $v$.

\begin{definition}
    Suppose $S$ is a valid stranding on a web graph $G\in F(\vec{k})$ and $v$ is an interior vertex of $G$ incident to edges $e_1, e_2, e_3$ weighted $\ell_1, \ell_2, \ell_3$. We say the stranding $S$ on $G$ \emph{conserves binary flow} at $v$ if, as $\mathbb{Z}^n$ vectors, we have 
\[\sum_{\sigma_v(e_i)=1} b_S(e_i)  = \sum_{\sigma_v(e_i)=-1} b_S(e_i).\]
\end{definition}

This leads to the next corollary, which implies that if $v$ is a split vertex then, after performing \emph{any} edge-flips that leave exactly one edge directed into $v$, the edge-weights of the graph conserve integer flow at $v$ (in the sense of Definition \ref{definition: tagged web}) and the binary labeling $b_S$ for \emph{any} valid stranding $S$ conserves binary flow (respectively merge and two edges directed in).

\begin{corollary} \label{corollary: relative position of omitted edge for binary}
    Suppose $S$ is a valid stranding on a web graph $G$ and $v$ is an interior vertex of $G$ incident to edges $e_1, e_2, e_3$ weighted $\ell_1, \ell_2, \ell_3$. Suppose further that $v$ is neither a source nor a sink, namely at least one edge is directed into $v$ and at least one edge is directed out of $v$.   

Then $G$ conserves integer flow at $v$ if and only if $S$ conserves binary flow at $v$.
\end{corollary}

\begin{proof}
The binary string $b_S(e_i)$ has exactly $\ell_i$ ones by definition of a binary labeling.  Thus the sum of the entries in the vector $\sum_{\sigma_v(e_i)=1} b_S(e_i)  - \sum_{\sigma_v(e_i)=-1} b_S(e_i)$ is precisely $\sum_{\sigma_v(e_i)=1} \ell_i  - \sum_{\sigma_v(e_i)=-1} \ell_i$.  This proves the claim.
\end{proof}  

In what follows, suppose $G_-$ and $G_+$ are web graphs that differ only in the web graph fragments shown in Figure~\ref{figure: equivalent stranding sets -- no interior faces}. Assume that the simple strands that enter and exit each fragment are fixed, or equivalently that the binary labels on the boundary edges are fixed.  We will explicitly provide natural bijections between the valid strandings on $G_-$ and $G_+$. The key observations are 1) that simple strands either pass through the graph or form a closed cycle inside the graph and 2) that these graphs conserve integer flow, which by Corollary~\ref{corollary: relative position of omitted edge for binary} is true if and only if they conserve binary flow.

\begin{figure}[h]
\scalebox{.8}{$\begin{array}{|c|c|l|}
\cline{1-3} G_- & G_+ & \textup{ Identities } \\

    \cline{1-3}   
     \raisebox{-50pt}{\begin{tikzpicture}[scale=.75]
    \draw[radius=.08, fill=black](1,2)circle;
     \draw[radius=.08, fill=black](.5,1)circle;
     \node[color=gray] at (0,-.25) {\tiny{$k$}};
     \node[color=gray] at (1,-.25) {\tiny{$l$}};
      \node[color=gray] at (2,-.25) {\tiny{$m$}};
       \node[color=gray] at (.3,1.95) {\tiny{$k+l$}};
     \node[color=gray] at (1,3) {\tiny{$k+l+m$}};

   \node at (0,-.8) {{$\vec{b}_1$}}; 
  \node at (1,-.8) {{$\vec{b}_2$}}; 
  \node at (2,-.8) {{$\vec{b}_3$}};  
\node at (1,3.5) {{$\vec{b}_4$}}; 
    \node at (0.1,1.5) {{$\vec{y}$}}; 
     
    \begin{scope}[thick,decoration={
    markings,
    mark=at position 0.5 with {\arrow{>}}}
    ] 
        \draw[postaction={decorate}, thick] (0,0)--(.5,1);
        \draw[postaction={decorate}, thick] (1,0)--(.5,1);
        \draw[postaction={decorate}, thick] (.5,1)--(1,2);
        \draw[postaction={decorate}, thick] (2,0)--(1,2);
        \draw[postaction={decorate}, thick] (1,2)--(1,2.75);
        
        \end{scope}
    \end{tikzpicture}} 
    &

      \raisebox{-50pt}{\begin{tikzpicture}[scale=.75]
    \draw[radius=.08, fill=black](1,2)circle;
     \draw[radius=.08, fill=black](1.5,1)circle;
     \node[color=gray] at (0,-.25) {\tiny{$k$}};
     \node[color=gray] at (1,-.25) {\tiny{$l$}};
      \node[color=gray] at (2,-.25) {\tiny{$m$}};
       \node[color=gray] at (1.75,1.85) {\tiny{$l+m$}};
     \node[color=gray] at (1,3) {\tiny{$k+l+m$}};

       \node at (0,-.8) {{$\vec{b}_1$}}; 
  \node at (1,-.8) {{$\vec{b}_2$}}; 
  \node at (2,-.8) {{$\vec{b}_3$}}; 
\node at (1,3.5) {{$\vec{b}_4$}}; 
    \node at (1.75,1.45) {{$\vec{x}$}}; 

    \begin{scope}[thick,decoration={
    markings,
    mark=at position 0.5 with {\arrow{>}}}
    ] 
        \draw[postaction={decorate}, thick] (0,0)--(1,2);
        \draw[postaction={decorate}, thick] (1,0)--(1.5,1);
        \draw[postaction={decorate}, thick] (1.5,1)--(1,2);
        \draw[postaction={decorate}, thick] (2,0)--(1.5,1);
        \draw[postaction={decorate}, thick] (1,2)--(1,2.75);
        
        \end{scope}
    \end{tikzpicture}} & 
     \begin{array}{lcl}
    \vec{0} &=& (\vec{b}_1 + \vec{b}_2 + \vec{b}_3)-\vec{b}_4 \\
    \vec{y} &=&\vec{b}_1+\vec{b}_2 \\
    \vec{x} &=&\vec{b}_2+\vec{b}_3 
    \end{array} \\
    
    \cline{1-3}
       
    \raisebox{-70pt}{\begin{tikzpicture}[scale=.75]
    \draw[radius=.08, fill=black](0,0)circle;
     \draw[radius=.08, fill=black](0,2)circle;
        \draw[radius=.08, fill=black](2,0)circle;
     \draw[radius=.08, fill=black](2,2)circle;
     \node[color=gray] at (0,-1.25) {\tiny{$k$}};
     \node[color=gray] at (2,-1.25) {\tiny{$l$}};
      \node[color=gray] at (1,-.25) {\tiny{$1$}};
      \node[color=gray] at (1,2.25) {\tiny{$1$}};
       \node[color=gray] at (-.75,1.25) {\tiny{$k-1$}};
       \node[color=gray] at (2.75,1.25) {\tiny{$l+1$}};
  \node[color=gray] at (0,3.25) {\tiny{$k$}};
     \node[color=gray] at (2,3.25) {\tiny{$l$}};

       \node at (0,-1.8) {{$\vec{b}_1$}}; 
\node at (2,-1.8) {{$\vec{b}_2$}}; 
       \node at (0,3.75) {{$\vec{b}_3$}}; 
\node at (2,3.75) {{$\vec{b}_4$}}; 
\node at (-.5,.65) {{$\vec{z}_1$}}; 
\node at (2.5,.65) {{$\vec{z}_2$}}; 
\node at (1.1,.45) {{$\vec{y}_1$}}; 
\node at (1.1,1.5) {{$\vec{y}_2$}}; 

    \begin{scope}[thick,decoration={
    markings,
    mark=at position 0.5 with {\arrow{>}}}
    ] 
        \draw[postaction={decorate}, thick] (0,-1)--(0,0);
        \draw[postaction={decorate}, thick] (0,0)--(0,2);
        \draw[postaction={decorate}, thick] (0,2)--(0,3);
         \draw[postaction={decorate}, thick] (2,-1)--(2,0);
        \draw[postaction={decorate}, thick] (2,0)--(2,2);
        \draw[postaction={decorate}, thick] (2,2)--(2,3);
        \draw[postaction={decorate}, thick] (0,0)--(2,0);
        \draw[postaction={decorate}, thick] (2,2)--(0,2);
        
        \end{scope}
    \end{tikzpicture}} 

    & 
           
    \raisebox{-70pt}{\begin{tikzpicture}[scale=.75]
    \draw[radius=.08, fill=black](0,0)circle;
     \draw[radius=.08, fill=black](0,2)circle;
        \draw[radius=.08, fill=black](2,0)circle;
     \draw[radius=.08, fill=black](2,2)circle;
     \node[color=gray] at (0,-1.25) {\tiny{$k$}};
     \node[color=gray] at (2,-1.25) {\tiny{$l$}};
      \node[color=gray] at (1,-.25) {\tiny{$1$}};
      \node[color=gray] at (1,2.25) {\tiny{$1$}};
       \node[color=gray] at (-.75,1.25) {\tiny{$k+1$}};
       \node[color=gray] at (2.75,1.25) {\tiny{$l-1$}};
  \node[color=gray] at (0,3.25) {\tiny{$k$}};
     \node[color=gray] at (2,3.25) {\tiny{$l$}};

       \node at (0,-1.8) {{$\vec{b}_1$}}; 
\node at (2,-1.8) {{$\vec{b}_2$}}; 
       \node at (0,3.75) {{$\vec{b}_3$}};  
\node at (2,3.75) {{$\vec{b}_4$}}; 
\node at (-.5,.65) {{$\vec{w}_1$}}; 
\node at (2.5,.65) {{$\vec{w}_2$}}; 
\node at (1.1,.45) {{$\vec{x}_1$}}; 
\node at (1.1,1.6) {{$\vec{x}_2$}}; 

    \begin{scope}[thick,decoration={
    markings,
    mark=at position 0.5 with {\arrow{>}}}
    ] 
        \draw[postaction={decorate}, thick] (0,-1)--(0,0);
        \draw[postaction={decorate}, thick] (0,0)--(0,2);
        \draw[postaction={decorate}, thick] (0,2)--(0,3);
         \draw[postaction={decorate}, thick] (2,-1)--(2,0);
        \draw[postaction={decorate}, thick] (2,0)--(2,2);
        \draw[postaction={decorate}, thick] (2,2)--(2,3);
        \draw[postaction={decorate}, thick] (0,2)--(2,2);
        \draw[postaction={decorate}, thick] (2,0)--(0,0);
        
        \end{scope}
    \end{tikzpicture}} 
    &      \begin{array}{lcl}
      \multicolumn{3}{c}{} \\
       \multicolumn{3}{c}{\textup{Assume: } \vec{b}_1 \neq \vec{b}_3} \\
    \multicolumn{3}{l}{} \\
    \vec{0} &=& (\vec{b}_1+\vec{b}_2) -(\vec{b}_3+\vec{b}_4) \\
   \vec{y}_1-\vec{y}_2 &=& \vec{b}_4 -\vec{b}_2 \\
    \vec{z}_1 + \vec{y}_1 &=& \vec{b}_1\\
    \vec{z}_2-\vec{y}_1 &=& \vec{b}_2 \\   
   \vec{x}_2-\vec{x}_1 &=& \vec{b}_4 -\vec{b}_2 \\
    \vec{w}_1 - \vec{x}_1 &=& \vec{b}_1\\
    \vec{w}_2+\vec{x}_1 &=& \vec{b}_2 \\ 
    \vec{y}_1-\vec{x}_2 &=& \vec{0} \\ 
    \vec{y}_2-\vec{x}_1 &=& \vec{0} \\ 
     & &  \vspace{0.05in}\\
    \end{array}
    \\
\hline
\end{array}$}
\caption{Constructing web graph fragment labelings compatible with a common choice of boundary labels} \label{figure: equivalent stranding sets -- no interior faces}
\end{figure}

\begin{lemma} \label{lemma: bijection for relations with no interior faces}
Suppose $G_-$ and $G_+$ are web graphs that differ only in the web graph fragments shown in Figure~\ref{figure: equivalent stranding sets -- no interior faces}, with boundary binary labels $\vec{b}_1, \vec{b}_2, \vec{b}_3, \vec{b}_4$ fixed. The web graphs $G_-$ and $G_+$ each have at most one valid stranding compatible with this choice of boundary.
Moreover there is a valid stranding $S_- \in \mathcal{S}tr(G_-)$ compatible with this boundary if and only if there is a valid stranding $S_+ \in \mathcal{S}tr(G_+)$ compatible with this boundary.  In that case, this pair of strandings shares the same set of $(i,j)$ strands for each $(i,j)$ , namely $L_{(i,j)}(S_-) = L_{(i,j)}(S_+)$.
\end{lemma}

\begin{proof}
For the graphs in the first row, we can see from the identities in the third column that strandings of $G_-$ and $G_+$ exist if and only if $\vec{b}_4=\vec{b}_1+\vec{b}_2+\vec{b}_3$, and in this case, the strandings compatible with these vectors are uniquely determined. Similarly for the second row, the identities in the third column imply that a stranding exists for $G_-$ if and only if one exists for $G_+$, and in this case, the strandings are uniquely determined. In particular, $\vec{y}_1=\vec{x}_2$, and $\vec{y}_2=\vec{x}_1$.

Now, assume a stranding compatible with the choice of boundary labels exists, and let $1\leq i<j\leq n$. Observe the graph fragments in the first row cannot have closed $(i,j)$ strands since they have no closed faces.  Any closed strand in the graph fragments on the second row of Figure~\ref{figure: equivalent stranding sets -- no interior faces} must pass over both horizontal edges. This implies $\vec{y}_1 + \vec{y}_2 = \vec{x}_1 + \vec{x}_2$ is nonzero precisely in entries $i$ and $j$. In this case, however, the $(i,j)$ strands along these two edges must point in the same direction (against one edge and with the other). This means the $(i,j)$ strands across these edges do not form a closed component. 

For each pair of graphs $G_-$ and $G_+$, the $(i,j)$ strands along the boundary edges are identical, and there can be no strands confined to the interior of the graph. In particular, if at most two edges of $G_-$ and $G_+$ carry $(i,j)$ strands then regardless of their path through each web graph fragment, the strands' orientation is the same in both graphs overall.  This case includes all $(i,j)$ strands in the graphs on the top row. The web graph fragments in the second row could support two $(i,j)$ strand components. In this case, since $\vec{y}_1 = \vec{x}_2$ and $\vec{y}_2 = \vec{x}_1$ either both or neither of $G_-$ and $G_+$ admit $(i,j)$ strands over the horizontal edges, respectively vertical. Thus $G_-$ and $G_+$ always have the same number and direction of $(i,j)$ strands for each pair $(i,j)$.
\end{proof}

The next examples are similar but include graphs with interior closed strands.  We treat the weight $l$ loop in the web graph fragments $G_+$ shown in  Figure~\ref{figure: strand-preserving bijection with loops} as  $\mathfrak{sl}_{k+l}$ web graphs. We define a map from valid strandings of this loop to the set of binary vectors of length $n$ that satisfy all constraints on $\vec{y}_1$ in $G_-$. 
We will show this map induces a bijection between the valid strandings of $G_-$ and $G_+$ that are consistent with the fixed choice of boundary vectors and moreover that this bijection preserves oriented closed strands.  

\begin{figure}[h]
\centering
\scalebox{.8}{
$\begin{array}{|c|c|l|}
\cline{1-3} G_- & G_+ & \textup{ Identities } \\
\cline{1-3}

 \raisebox{-45pt}{\begin{tikzpicture}[scale=.5]
    \draw[radius=.12, fill=black](1,2)circle;
     \draw[radius=.12, fill=black](2,4)circle;
     \node[color=gray] at (-.75,1) {\tiny{$k+l$}};
     \node[color=gray] at (1.25,5) {\tiny{$k+l$}};
       \node[color=gray] at (0,3.75) {\tiny{$k$}};
     \node[color=gray] at (3,2.25) {\tiny{$l$}};
     \node at (-.5,0) {{$\vec{b_1}$}}; 
     \node at (3.5,5.8) {{$\vec{b}_2$}}; 
    \node at (2.5,1.75) {{$\vec{y}_1$}}; 
    \node at (-.1,2.65) {{$\vec{y}_2$}}; 
  
    \begin{scope}[thick,decoration={
    markings,
    mark=at position 0.5 with {\arrow{>}}}
    ] 
        \draw[postaction={decorate}, thick] (0,0)--(1,2);
        \draw[postaction={decorate}, thick] (2,4)--(3,6);
        \draw[postaction={decorate}, thick] (1,2)to[out=150, in=240] (.5, 3.5) to[out=60, in=165] (2,4);
        \draw[postaction={decorate}, thick] (1,2)to[out=-30, in=240] (2.5, 2.5) to[out=60, in=-15] (2,4);
        
        \end{scope}
    \end{tikzpicture}} 
    & 
    \raisebox{-45pt}{\begin{tikzpicture}[scale=.5]
     \node[color=gray] at (1.85,2) {\tiny{$k+l$}};
     \node[color=gray] at (0.65,3.25) {\tiny{$l$}};
     \node at (-.5,0) {{$\vec{b}_1$}}; 
    \node at (-1,2.25) {{$\vec{x}$}}; 
  
    \begin{scope}[thick,decoration={
    markings,
    mark=at position 0.5 with {\arrow{>}}}
    ] 
        \draw[postaction={decorate}, thick] (0,0)--(3,6);
        \draw[dashed, postaction={decorate}, thick, radius=1] (0,3.5) circle;
     
        \end{scope}
    \end{tikzpicture}} 
    & 
    \begin{array}{lcl}
    \vec{b}_2 &=& \vec{b}_1 \\
    \vec{y}_2 &=& \vec{b}_1 - \vec{y}_1 \\
      \multicolumn{3}{c}{\hrulefill} \\
    \vec{y}_1 &=&\iota_{\vec{b}_1}(\vec{x}) \\
     & &\mathfrak{sl}_{k+l} \textup{ loop}
   \end{array} \\

    \cline{1-3}
\raisebox{-70pt}{\begin{tikzpicture}[scale=.75]
    \draw[radius=.08, fill=black](0,0)circle;
     \draw[radius=.08, fill=black](0,2)circle;
        \draw[radius=.08, fill=black](2,0)circle;
     \draw[radius=.08, fill=black](2,2)circle;
     \node[color=gray] at (0,-1.25) {\tiny{$r$}};
     \node[color=gray] at (2,-1.25) {\tiny{$s$}};
      \node[color=gray] at (1,-.25) {\tiny{$l$}};
      \node[color=gray] at (1,2.25) {\tiny{$k$}};
       \node[color=gray] at (-.75,1.25) {\tiny{$r-l$}};
       \node[color=gray] at (2.75,1.25) {\tiny{$s+l$}};
  \node[color=gray] at (0,3.25) {\tiny{$r-k-l$}};
     \node[color=gray] at (2,3.25) {\tiny{$s+k+l$}};

       \node at (0,-1.8) {{$\vec{b}_1$}}; 
\node at (2,-1.8) {{$\vec{b}_2$}}; 
\node at (-.5,.65) {{$\vec{z}_1$}}; 
\node at (2.5,.65) {{$\vec{z}_2$}}; 
\node at (1.1,.45) {{$\vec{y}_1$}}; 
\node at (1.1,1.6) {{$\vec{y}_2$}}; 
\node at (0,3.75) {{$\vec{b}_3$}}; 
\node at (2,3.75) {{$\vec{b}_4$}}; 

    \begin{scope}[thick,decoration={
    markings,
    mark=at position 0.5 with {\arrow{>}}}
    ] 
        \draw[postaction={decorate}, thick] (0,-1)--(0,0);
        \draw[postaction={decorate}, thick] (0,0)--(0,2);
        \draw[postaction={decorate}, thick] (0,2)--(0,3);
         \draw[postaction={decorate}, thick] (2,-1)--(2,0);
        \draw[postaction={decorate}, thick] (2,0)--(2,2);
        \draw[postaction={decorate}, thick] (2,2)--(2,3);
        \draw[postaction={decorate}, thick] (0,0)--(2,0);
        \draw[postaction={decorate}, thick] (0,2)--(2,2);
        
        \end{scope}
    \end{tikzpicture}}
    & 
\raisebox{-70pt}{\begin{tikzpicture}[scale=.75]
     \draw[radius=.08, fill=black](0,2)circle;
     \draw[radius=.08, fill=black](2,2)circle;
     \node[color=gray] at (0,-1.25) {\tiny{$r$}};
     \node[color=gray] at (2,-1.25) {\tiny{$s$}};
      \node[color=gray] at (1,2.25) {\tiny{$k+l$}};
   \node[color=gray] at (0,3.25) {\tiny{$r-k-l$}};
     \node[color=gray] at (2,3.25) {\tiny{$s+k+l$}};
     \node[color=gray] at (-.5,1) {\tiny{$l$}};

       \node at (0,-1.8) {{$\vec{b}_1$}}; 
\node at (2,-1.8) {{$\vec{b}_2$}}; 
\node at (1.1,1.6) {{$\vec{x}_1$}}; 
\node at (-1.5,.25) {{$\vec{x}$}}; 
\node at (0,3.75) {{$\vec{b}_3$}}; 
\node at (2,3.75) {{$\vec{b}_4$}}; 

    \begin{scope}[thick,decoration={
    markings,
    mark=at position 0.5 with {\arrow{>}}}
    ] 
        \draw[postaction={decorate}, thick] (0,-1)--(0,2);
        \draw[postaction={decorate}, thick] (0,2)--(0,3);
         \draw[postaction={decorate}, thick] (2,-1)--(2,2);
        \draw[postaction={decorate}, thick] (2,2)--(2,3);
        \draw[postaction={decorate}, thick] (0,2)--(2,2);

       \draw[dashed, postaction={decorate}, thick, radius=0.75] (-1,1.25) circle;

        \end{scope}
    \end{tikzpicture}}
    & 
     \begin{array}{lcl}
    \vec{b}_1 &=& \vec{b}_3 + \vec{x}_1 \\
    \vec{b}_4 &=& \vec{b}_2 + \vec{x}_1 \\
    \vec{z}_1 &=& \vec{b}_3 + \vec{y}_2 \\
    \vec{z}_1 &=& \vec{b}_1 + \vec{y}_1 \\
    \vec{z}_2 &=& \vec{b}_2 + \vec{y}_1 \\ 
    \vec{b}_4 &=& \vec{z}_2 + \vec{y}_2 \\
    \vec{x}_1 &=& \vec{y}_1 + \vec{y}_2 \\
      \multicolumn{3}{c}{\hrulefill} \\
    \vec{y}_1 &=&\iota_{\vec{x}_1}(\vec{x}) \\
    && \mathfrak{sl}_{k+l} \textup{ loop}
    \end{array} \\
\hline\end{array}$}
\caption{A strand-preserving bijection between strandings on $\mathfrak{sl}_n$ web graph fragments that extend fixed boundary edge labels, with dashed loops interpreted as $\mathfrak{sl}_{k+l}$ webs}\label{figure: strand-preserving bijection with loops}
\end{figure}

\begin{lemma} \label{lemma: easy relations}
Suppose $\vec{v} \in \{0,1\}^{n}$ has exactly $k+l$ ones.  Define the injection $\iota_{\vec{v}}: \{0, 1\}^{k+l} \rightarrow \{0,1\}^n$ by replacing the $k+l$ nonzero positions of $\vec{v}$ with the $k+l$ entries of $\vec{x} \in \{0,1\}^{k+l}$ inserted in order from left-to-right. If either 
    \begin{itemize}
        \item $\vec{v} = \vec{b}_1$ for $G_+$ in the top row, namely the map is $\vec{y}_1=\iota_{\vec{b}_1}(\vec{x})$ or
        \item $\vec{v} = \vec{x}_1$ for $G_+$  in the bottom row, namely the map is $\vec{y}_1=\iota_{\vec{x}_1}(\vec{x})$
    \end{itemize} 
    then the map $\iota_{\vec{v}}$ induces a bijection between valid strandings of the web graph $G_+$ and of $G_-$ in Figure~\ref{figure: strand-preserving bijection with loops}, both compatible with the given boundary edge labels.   
    
    Moreover, the map $\iota_{\vec{v}}$ preserves directions of strands in the following sense.  Suppose $S_+ \in \mathcal{S}tr(G_+)$, with $\vec{x}$ the associated binary label on the loop, and $S_- \in \mathcal{S}tr(G_-)$ is the stranding induced from $\vec{y}_1 = \iota_{\vec{v}}(\vec{x})$.   Then for each $(i,j)$, we have $L_{(i,j)}(G_-) = L_{(i,j)}(G_+)$ in the sense that strands enter and exit at the same boundary edges in $G_-$ as in $G_+$, and $G_-$ and $G_+$ have the same number of clockwise closed strands and counterclockwise strands. 
\end{lemma}

\begin{proof}
We mimic the argument in Lemma~\ref{lemma: bijection for relations with no interior faces}.  The identities in the third column  of Figure~\ref{figure: strand-preserving bijection with loops} show that a stranding exists in $G_-$ only if a stranding exists on the non-loop component in $G_+$ with $\vec{x}_1 = \vec{y}_1+\vec{y}_2$ for the graphs in the second row.  The non-loop component in $G_+$ can have at most one stranding compatible with the boundary labels $\vec{b}_1, \vec{b}_2, \vec{b}_3, \vec{b}_4$.  We now characterize the ways to choose $\vec{y}_1$ satisfying the identities in the third column.

Since the vector $\vec{x}$ has exactly $l$ ones, the vector $\iota_{\vec{v}}(\vec{x})$ has exactly $l$ ones also, all of which are in positions that are nonzero in $\vec{v}$.  So every vector $\vec{y}_1 \in Im(\iota_{\vec{v}})$ gives rise to a valid stranding of the graph $G_-$ compatible with the boundary labels $\vec{b}_1, \vec{b}_2, \vec{b}_3, \vec{b}_4$.  Conversely, fixing boundary labels on  $G_-$ determines $\vec{v}$.  All possible vectors with exactly $l$ ones amongst the $k+l$ nonzero positions of $\vec{v}$ are in $Im(\iota_{\vec{v}})$.  So the map $\iota_{\vec{v}}$ induces a bijection between strandings of the graphs $G_-$ and $G_+$ that are compatible with the boundary labels.

Finally, we prove that the map $\iota_{\vec{v}}$ induces a bijection of the individual directed strands within each stranding.  Suppose $S_+$ is a stranding of $G_+$ with loop edge labeled $\vec{x}$ and $S_-$ is the stranding induced from $\iota_{\vec{v}}(\vec{x}) = \vec{y}_1$.  Every closed strand in $G_-$ must run along the edge with binary label $\vec{y}_1$.  We consider the following cases:
\begin{itemize}
    \item Suppose entries $i,j$ are both nonzero in $\vec{v}$. Then there are no $(i,j)$ strands  along any boundary edges of the graph fragments and the $(i,j)$ strands in $G_-$ should match that of $G_+$. Moreover, entries $i,j$ in $\vec{y}_1$ contain the same values as the corresponding entries in $\vec{x}$, and so there is a closed strand in $L_{(i,j)}(G_-)$ along the edge with binary label $\vec{y}_1$ if and only if there is a corresponding strand in the same direction along the loop edge in $G_+$. 
    \item If  $i,j$ are both zero in $\vec{v}$ then by construction, we check there are no closed $(i,j)$ strands in $G_-$.  There are no $(i,j)$ strands in the examples in the first row of Figure~\ref{figure: strand-preserving bijection with loops}, and no $(i,j)$ strands over a horizontal edge in either $G_-$ or $G_+$ in the second row.  For the graphs in the second row, the $(i,j)$ strands run solely vertically compatible with the boundary edges, and thus have the same number and direction in $G_-$ as in $G_+$.
    \item Now suppose that exactly one of $i, j$ is zero in $\vec{v}$. Again, we confirm there are no closed $(i,j)$ strands in $G_-$. In this case, an $(i,j)$ strand in $G_-$ runs along exactly one of the edges labeled $\vec{y}_1$ and $\vec{y}_2$.  The $(i,j)$ strand is therefore not closed in the interior of $G_-$ so passes over both boundary edges for $G_-$ in the first row and at least two boundary edges for $G_-$ in the second row.  Since one of $i, j$ is nonzero in $\vec{v}$ there cannot be two $(i,j)$ strands passing through $G_-$.  Since strands are well-defined, however the $(i,j)$ strand runs through $G_-$ and $G_+$, it has the same endpoints (and thus direction) in both graphs.
\end{itemize}      
\end{proof}

\subsection{The square-switch relation}

Above, we proved relations on web graphs via explicit bijections between strandings of two different untagged web fragments that preserved the number and direction of strands. Now, we prove the square-switch relation which involves the web fragments shown in Figure~\ref{figure: strand switch relation}.  These graphs admit no strand-preserving bijection, so the approach is more complicated. The following lemma outlines our overall approach.

\begin{lemma} \label{lemma: overall square switch argument}
Suppose $G_-$, $G_+$, and $G_0$ are the web graph fragments shown in Figure~\ref{figure: strand switch relation} where the loop in $G_0$ is considered an $\mathfrak{sl}_{|k-l|}$ web and all other graph components are $\mathfrak{sl}_n$ webs. Fix boundary binary labels $\vec{b}_1$ and $\vec{b}_2$, and let $w_\partial(G_-)$ denote the expression
\[w_\partial(G_-) = \sum_{\footnotesize \begin{array}{c} S \in \mathcal{S}tr(G_-) \\ \partial(S)  =  \vec{b}_1 \sqcup \vec{b}_2 \end{array} } (-q)^{x_{G_-}(S) - y_{G_-}(S)}\]
and similarly for $G_+, G_0$. Then $w_\partial(G_-) = w_\partial(G_+)+w_\partial(G_0).$
\end{lemma}

\begin{figure}[h]
\[\begin{array}{|c|c|c|}
\cline{1-3} G_- & G_+ & G_0 \\

   \cline{1-3}

    \raisebox{-70pt}{\begin{tikzpicture}[scale=.75]
    \draw[radius=.08, fill=black](0,0)circle;
     \draw[radius=.08, fill=black](0,2)circle;
        \draw[radius=.08, fill=black](2,0)circle;
     \draw[radius=.08, fill=black](2,2)circle;
     \node[color=gray] at (0,-1.25) {\tiny{$k$}};
     \node[color=gray] at (2,-1.25) {\tiny{$l$}};
      \node[color=gray] at (1,-.25) {\tiny{$1$}};
      \node[color=gray] at (1,2.25) {\tiny{$1$}};
       \node[color=gray] at (-.75,1.25) {\tiny{$k-1$}};
       \node[color=gray] at (2.75,1.25) {\tiny{$l+1$}};
  \node[color=gray] at (0,3.25) {\tiny{$k$}};
     \node[color=gray] at (2,3.25) {\tiny{$l$}};

       \node at (0,-1.8) {{$\vec{b}_1$}}; 
\node at (2,-1.8) {{$\vec{b}_2$}}; 
       \node at (0,3.75) {{$\vec{b}_1$}}; 
\node at (2,3.75) {{$\vec{b}_2$}}; 
\node at (-.5,.65) {{$\vec{z}_1$}}; 
\node at (2.5,.65) {{$\vec{z}_2$}}; 
\node at (1.1,.45) {{$\vec{y}$}}; 
\node at (1.1,1.5) {{$\vec{y}$}}; 

    \begin{scope}[thick,decoration={
    markings,
    mark=at position 0.5 with {\arrow{>}}}
    ] 
        \draw[postaction={decorate}, thick] (0,-1)--(0,0);
        \draw[postaction={decorate}, thick] (0,0)--(0,2);
        \draw[postaction={decorate}, thick] (0,2)--(0,3);
         \draw[postaction={decorate}, thick] (2,-1)--(2,0);
        \draw[postaction={decorate}, thick] (2,0)--(2,2);
        \draw[postaction={decorate}, thick] (2,2)--(2,3);
        \draw[postaction={decorate}, thick] (0,0)--(2,0);
        \draw[postaction={decorate}, thick] (2,2)--(0,2);
        
        \end{scope}
    \end{tikzpicture}} 

 &
       \raisebox{-70pt}{\begin{tikzpicture}[scale=.75]
    \draw[radius=.08, fill=black](0,0)circle;
     \draw[radius=.08, fill=black](0,2)circle;
        \draw[radius=.08, fill=black](2,0)circle;
     \draw[radius=.08, fill=black](2,2)circle;
     \node[color=gray] at (0,-1.25) {\tiny{$k$}};
     \node[color=gray] at (2,-1.25) {\tiny{$l$}};
      \node[color=gray] at (1,-.25) {\tiny{$1$}};
      \node[color=gray] at (1,2.25) {\tiny{$1$}};
       \node[color=gray] at (-.75,1.25) {\tiny{$k+1$}};
       \node[color=gray] at (2.75,1.25) {\tiny{$l-1$}};
  \node[color=gray] at (0,3.25) {\tiny{$k$}};
     \node[color=gray] at (2,3.25) {\tiny{$l$}};

       \node at (0,-1.8) {{$\vec{b}_1$}}; 
\node at (2,-1.8) {{$\vec{b}_2$}}; 
       \node at (0,3.75) {{$\vec{b}_1$}}; 
\node at (2,3.75) {{$\vec{b}_2$}}; 
\node at (-.5,.65) {{$\vec{z}'_1$}}; 
\node at (2.5,.65) {{$\vec{z}'_2$}}; 
\node at (1.1,.45) {{$\vec{x}$}}; 
\node at (1.1,1.6) {{$\vec{x}$}}; 

    \begin{scope}[thick,decoration={
    markings,
    mark=at position 0.5 with {\arrow{>}}}
    ] 
        \draw[postaction={decorate}, thick] (0,-1)--(0,0);
        \draw[postaction={decorate}, thick] (0,0)--(0,2);
        \draw[postaction={decorate}, thick] (0,2)--(0,3);
         \draw[postaction={decorate}, thick] (2,-1)--(2,0);
        \draw[postaction={decorate}, thick] (2,0)--(2,2);
        \draw[postaction={decorate}, thick] (2,2)--(2,3);
        \draw[postaction={decorate}, thick] (0,2)--(2,2);
        \draw[postaction={decorate}, thick] (2,0)--(0,0);
        
        \end{scope}
    \end{tikzpicture}} 
      & 
    \raisebox{-70pt}{\begin{tikzpicture}[scale=.75]
    \node[color=gray] at (1,3.25) {\tiny{$k$}};
     \node[color=gray] at (2,3.25) {\tiny{$l$}};
   \node[color=gray] at (0,1) {\tiny{$1$}};

       \node at (1,-1.8) {{$\vec{b}_1$}}; 
\node at (2,-1.8) {{$\vec{b}_2$}}; 
\node at (-1,.25) {{$\vec{x}$}}; 

    \begin{scope}[thick,decoration={
    markings,
    mark=at position 0.5 with {\arrow{>}}}
    ] 
        \draw[postaction={decorate}, thick] (1,-1)--(1,3);
         \draw[postaction={decorate}, thick] (2,-1)--(2,3);

     \draw[dashed, postaction={decorate}, thick, radius=0.75] (-.5,1.25) circle;

        \end{scope}
    \end{tikzpicture}} 
    \\
\hline
\end{array}
\]
\caption{Square-switch example: the case when the same  simples boundary strands enter and leave each side, with loop graph read as $\mathfrak{sl}_{|k-l|}$ web} \label{figure: strand switch relation}
\end{figure}

First, we explicitly compute the invariant vector associated to a loop graph.

\begin{lemma} \label{lemma: loop graph calculation}
Suppose $G_{n,k}$ is the loop $\mathfrak{sl}_n$ web graph with weight $k$ shown below.  Then the associated invariant is a scalar $w(G_{n,k}) \in \mathbb{C}(q)$ that equals the quantum binomial coefficient $\genfrac[]{0pt}{0}{n}{k}_q$ evaluated at $-q$.  In the case $k=1$ we have $w(G_{n,1}) = (-q)^{n-1} + (-q)^{n-3} + \cdots + (-q)^{-n+3} + (-q)^{-n+1}$.
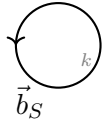
\begin{center} \raisebox{0pt}{\begin{tikzpicture}[scale=.75]
    \node[color=gray] at (0,1) {\tiny{$k$}};
\node at (-1,.25) {{$\vec{b}_S$}}; 

    \begin{scope}[thick,decoration={
    markings,
    mark=at position 0.5 with {\arrow{>}}}
    ] 
     \draw[postaction={decorate}, thick, radius=0.75] (-.5,1.25) circle;
        \end{scope}
    \end{tikzpicture}}
    \end{center}
\end{lemma}

\begin{proof}
The strandings of $G_{n,k}$ are bijective with $n$-bit binary strings with exactly $k$ ones.  All strands are closed so $x(S)$ counts the number of pairs $i<j$ for which the binary vector $\vec{b}_S$ has $(\vec{b}_S)_i (\vec{b}_S)_j = 01$ respectively $y(S)$ and $(\vec{b}_S)_i (\vec{b}_S)_j = 10$.  For instance, if $\vec{b}_S$ is the $n$-bit binary vector $\vec{e}_i$ with $1$ in entry $i$ and $0$ elsewhere then $x(S)-y(S) = (i-1)-(n-i) = 2i-n-1$.  This shows that for any $n$ and $k=1$ we have
\[w(G_{n,1}) = (-q)^{n-1} + (-q)^{n-3} + \cdots + (-q)^{-n+3} + (-q)^{-n+1}\]
which is $\genfrac[]{0pt}{0}{n}{1}_q$ evaluated at $-q$ as desired. Splitting the sum depending on the last bit of $\vec{b}_S$ gives the recurrence
\begin{align*}  
w(G_{n,k}) &= \sum_{S:(\vec{b}_S)_n=0} (-q)^{x(S)-y(S)} +\sum_{S:(\vec{b}_S)_n=1} (-q)^{x(S)-y(S)} \\
&= (-q)^{k} w(G_{n-1,k}) + (-q)^{-n+k}w(G_{n-1,k-1})
\end{align*}
A short calculation shows that quantum binomial coefficients satisfy the recurrence
$$
q^{k}\genfrac[]{0pt}{0}{n-1}{k}_q + q^{-n+k}\genfrac[]{0pt}{0}{n-1}{k-1}_q = \genfrac[]{0pt}{0}{n}{k}_q.
$$
Replacing $q$ by $-q$ gives the corresponding recurrence for the coefficients above.
We verified the base case $k=1$ above, so induction completes the proof.
\end{proof}

Our proof of Lemma~\ref{lemma: overall square switch argument} has three main steps.  First, Lemma~\ref{lemma: square switch is like straight lines and loop} shows that a strand in one of the web graph fragments with a square contributes the same as the strand in the web graph fragment with parallel lines together with the strand around a loop.  Then Lemma~\ref{lemma: flex to relate strand exponent for square to straigh} gives a combinatorial way to compute this strand contribution, Corollary~\ref{corollary: base case for flex induction in square switch} applies it to a specific pair of boundary vectors $\vec{b}_1, \vec{b}_2$, and Lemma~\ref{lemma: inductive step of square switch} extends combinatorially to all vectors via the natural $S_n$ action on $n$-bit binary vectors.  Finally, we combine these results to complete the proof.

Strands allow us to ignore extraneous details from many calculations.  For instance, if $G$ is a web graph with an interior square, we can analyze different ways that strands could run through the square, using properties of the strands in isolation from the rest of the web graph. Our next result shows that the strand contribution in these two cases differs by a closed strand around the interior square.

   \begin{lemma} \label{lemma: square switch is like straight lines and loop}
    Suppose a stranded web graph contains an interior square and that the $(i,j)$ strands alternate direction in/out along the four boundary edges (equivalently entries $i, j$ of the binary labels on the boundary alternate $10, 01$).  There are two possible ways these $(i,j)$ strands can pass through the square, shown in the schematics below.  Label the five regions around the square as shown below.  
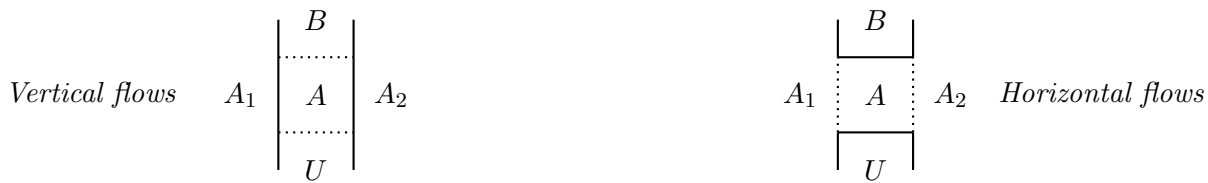
\begin{center}
\begin{tikzpicture}
    \draw[thick] (0,0) -- (0,2);
    \draw[thick] (1,0) -- (1,2);
    \draw[thick, dotted] (0,0.5) -- (1,.5);
    \draw[thick, dotted] (0,1.5) -- (1,1.5);
    \node at (-2.5,1) {\shortstack{Vertical\\ strands}};
    \node at (0.5,1) {$A$};
    \node at (-0.5,1) {$A_1$};
    \node at (1.5,1) {$A_2$};
    \node at (0.5,2) {$B$};
    \node at (0.5,0) {$U$};
    
\end{tikzpicture} \hspace{.5in}
\begin{tikzpicture}
    \draw[thick] (0,0) -- (0,0.5);
    \draw[thick] (1,0) -- (1,0.5);
    \draw[thick] (0,1.5) -- (0,2);
    \draw[thick] (1,1.5) -- (1,2);
    \draw[thick] (0,0.5) -- (1,.5);
    \draw[thick] (0,1.5) -- (1,1.5);
    \draw[thick, dotted] (0,0.5) -- (0,1.5);
    \draw[thick, dotted] (1,0.5) -- (1,1.5);
       \node at (3.5,1) {\shortstack{Horizontal\\ strands}};
           \node at (0.5,1) {$A$};
    \node at (-0.5,1) {$A_1$};
    \node at (1.5,1) {$A_2$};
    \node at (0.5,2) {$B$};
    \node at (0.5,0) {$U$};

\end{tikzpicture}
\end{center}
Let $x_{i,j}(S_V)-y_{i,j}(S_V)$ denote the contribution of the vertical $(i,j)$ strands in this schematic, respectively $S_H$ and horizontal.   Consider the closed curve around region $A$ that agrees in orientation with the vertical strands around the square, and define $\mu_{(i,j), V}(A)$ to be $1$ if this curve is directed clockwise, respectively $-1$ and counterclockwise.   Then 
    \[x_{(i,j)}(S_V)-y_{(i,j)}(S_V) +\mu_{(i,j), V}(A)= x_{(i,j)}(S_H)-y_{(i,j)}(S_H)\]
    \end{lemma}
    
    \begin{proof}
Note that switching the roles of vertical and horizontal strands in the schematic reverses the direction of the closed curve bounding the square and so negates $\mu_{(i,j), V}(A)$. But $-\mu_{(i,j), V}(A) = \mu_{(i,j),H}(A)$. Thus the formula holds whichever configuration is designated ``vertical."  

At least one of the regions around the square contains the point at infinity.  Without loss of generality, say this is region $U$.  

We start with several observations about the regions and the direction of the strands bounding the regions.  First, since $U$ contains the point at infinity, both $A_1$ and $A_2$ are bounded (and possibly $A_1 = A_2$), while $B$ may either be bounded or have $B=U$. Second, one of the vertical strands points up while the other points down, by hypothesis.  So the vertical strand that encloses region $A_1$ runs clockwise if and only if the same is true for the vertical strand that encloses region $A_2$.  (This, like the following claims, can also be confirmed by comparing the winding number in the regions on either side of a strand.)  By the same argument, the curves in (1) are all clockwise or all counterclockwise, and the curves in (2) are also all clockwise or all counterclockwise:
\begin{enumerate}
    \item the vertical strand enclosing region $A_1$, the vertical strand enclosing region $A_2$, the horizontal strand enclosing region $A_1 \cup A \cup A_2$
    \item the horizontal strand enclosing region $B$, the curve around $A$ that extends the direction of the vertical strands around the square
\end{enumerate}
Moreover, the curves in (1) are clockwise if and only if  those in (2) are counterclockwise.

If $A_1 = A_2$ forms a single region then the vertical strands form a single connected component, which must bound $B$ from above since $U$ contains the point at infinity.  Thus the horizontal strands form two connected components, one enclosing region $A_1 \cup A \cup A_2$ and the other enclosing region $B$.  The region $A_1=A_2$ is on the boundary if and only if $A_1 \cup A \cup A_2$ is, so points (1) and (2) prove the claim in this case.

Now suppose the vertical strands form two components, equivalently $A_1$ and $A_2$ form two disjoint regions.  These two disjoint components do not contain $B$ so the horizontal strand bounding $B$ contributes to $x_{(i,j)}(S_H)-y_{(i,j)}(S_H)$ only if it starts and ends on the boundary and runs counterclockwise.  In that case, the vertical strands must also start and end on the boundary, but run clockwise --- as does the horizontal strand bounding $A_1 \cup A \cup A_2$.  Since $\mu_{(i,j), V}(A) = -1$ in this case, the claim follows. Finally, the horizontal strand bounding $B$ contributes nothing to  $x_{(i,j)}(S_H)-y_{(i,j)}(S_H)$ if either $B=U$ (and the horizontal strand is a single connected component) or $B$ is on the boundary and the strand runs clockwise around it.  In both cases, the horizontal strands merge regions $A_1$ and $A_2$ thus eliminating the contribution of one region.  In other words, the difference
   \[\big(x_{(i,j)}(S_V)-y_{(i,j)}(S_V)\big) - \big(x_{(i,j)}(S_H)-y_{(i,j)}(S_H)\big) = \mu_{(i,j), V}(A)\]
This completes the proof.
\end{proof}
    
We can use the previous lemma to obtain an explicit combinatorial formula to compute strand contributions for these web graph fragments.  

\begin{lemma} \label{lemma: flex to relate strand exponent for square to straigh}
Suppose $\vec{b}_1$ and $\vec{b}_2$ are two binary vectors with $(\vec{b}_1)_i \neq (\vec{b}_2)_i$.  Consider the web graph fragment in the center stranded by $S_0$.  Exactly one of the web graphs on the left and right can be stranded as shown, with vertical edges around each square completed in the unique way that conserves binary labels: 
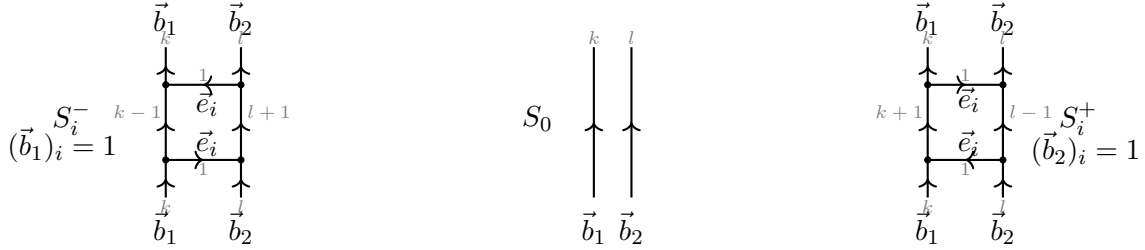
\begin{center}
          
    \raisebox{-70pt}{\begin{tikzpicture}[scale=.5]
    \draw[radius=.08, fill=black](0,0)circle;
     \draw[radius=.08, fill=black](0,2)circle;
        \draw[radius=.08, fill=black](2,0)circle;
     \draw[radius=.08, fill=black](2,2)circle;
     \node[color=gray] at (0,-1.25) {\tiny{$k$}};
     \node[color=gray] at (2,-1.25) {\tiny{$l$}};
      \node[color=gray] at (1,-.25) {\tiny{$1$}};
      \node[color=gray] at (1,2.25) {\tiny{$1$}};
       \node[color=gray] at (-.75,1.25) {\tiny{$k-1$}};
       \node[color=gray] at (2.75,1.25) {\tiny{$l+1$}};
  \node[color=gray] at (0,3.25) {\tiny{$k$}};
     \node[color=gray] at (2,3.25) {\tiny{$l$}};
     \node at (-2.5,1.1) {$S_i^-$};
     \node at (-2.75, 0.4) {$(\vec{b}_1)_i=1$};

       \node at (0,-1.8) {{$\vec{b}_1$}}; 
\node at (2,-1.8) {{$\vec{b}_2$}}; 
       \node at (0,3.75) {{$\vec{b}_1$}}; 
\node at (2,3.75) {{$\vec{b}_2$}}; 
\node at (1.1,.45) {{$\vec{e}_i$}}; 
\node at (1.1,1.5) {{$\vec{e}_i$}}; 

    \begin{scope}[thick,decoration={
    markings,
    mark=at position 0.5 with {\arrow{>}}}
    ] 
        \draw[postaction={decorate}, thick] (0,-1)--(0,0);
        \draw[postaction={decorate}, thick] (0,0)--(0,2);
        \draw[postaction={decorate}, thick] (0,2)--(0,3);
         \draw[postaction={decorate}, thick] (2,-1)--(2,0);
        \draw[postaction={decorate}, thick] (2,0)--(2,2);
        \draw[postaction={decorate}, thick] (2,2)--(2,3);
        \draw[postaction={decorate}, thick] (0,0)--(2,0);
        \draw[postaction={decorate}, thick] (2,2)--(0,2);
        
        \end{scope}
    \end{tikzpicture}} 
      \hspace{.25in}
    \raisebox{-70pt}{\begin{tikzpicture}[scale=.5]
    \node[color=gray] at (1,3.25) {\tiny{$k$}};
     \node[color=gray] at (2,3.25) {\tiny{$l$}};

   \node at (-.5,1.1) {$S_0$};

       \node at (1,-1.8) {{$\vec{b}_1$}}; 
\node at (2,-1.8) {{$\vec{b}_2$}};

    \begin{scope}[thick,decoration={
    markings,
    mark=at position 0.5 with {\arrow{>}}}
    ] 
        \draw[postaction={decorate}, thick] (1,-1)--(1,3);
         \draw[postaction={decorate}, thick] (2,-1)--(2,3);

        \end{scope}
    \end{tikzpicture}} 
 \hspace{.25in}
       \raisebox{-70pt}{\begin{tikzpicture}[scale=.5]
    \draw[radius=.08, fill=black](0,0)circle;
     \draw[radius=.08, fill=black](0,2)circle;
        \draw[radius=.08, fill=black](2,0)circle;
     \draw[radius=.08, fill=black](2,2)circle;
     \node[color=gray] at (0,-1.25) {\tiny{$k$}};
     \node[color=gray] at (2,-1.25) {\tiny{$l$}};
      \node[color=gray] at (1,-.25) {\tiny{$1$}};
      \node[color=gray] at (1,2.25) {\tiny{$1$}};
       \node[color=gray] at (-.75,1.25) {\tiny{$k+1$}};
       \node[color=gray] at (2.75,1.25) {\tiny{$l-1$}};
           \node at (4,1.1) {$S_i^+$};
           \node at (4.2, 0.3) {$(\vec{b}_2)_i=1$}; 
  \node[color=gray] at (0,3.25) {\tiny{$k$}};
     \node[color=gray] at (2,3.25) {\tiny{$l$}};

       \node at (0,-1.8) {{$\vec{b}_1$}}; 
\node at (2,-1.8) {{$\vec{b}_2$}}; 
       \node at (0,3.75) {{$\vec{b}_1$}}; 
\node at (2,3.75) {{$\vec{b}_2$}}; 
\node at (1.1,.45) {{$\vec{e}_i$}}; 
\node at (1.1,1.6) {{$\vec{e}_i$}}; 

    \begin{scope}[thick,decoration={
    markings,
    mark=at position 0.5 with {\arrow{>}}}
    ] 
        \draw[postaction={decorate}, thick] (0,-1)--(0,0);
        \draw[postaction={decorate}, thick] (0,0)--(0,2);
        \draw[postaction={decorate}, thick] (0,2)--(0,3);
         \draw[postaction={decorate}, thick] (2,-1)--(2,0);
        \draw[postaction={decorate}, thick] (2,0)--(2,2);
        \draw[postaction={decorate}, thick] (2,2)--(2,3);
        \draw[postaction={decorate}, thick] (0,2)--(2,2);
        \draw[postaction={decorate}, thick] (2,0)--(0,0);
        
        \end{scope}
    \end{tikzpicture}} 
    \end{center}
Define the function $\flex(i)$ from $\vec{b}_1$ and $\vec{b}_2$ by the rule
\[\flex(i) = \hspace{-.1in} \begin{array}{rl} 
&  + |\{j < i :  (\vec{b}_1)_j = 1 \neq (\vec{b}_2)_j \}| - |\{j < i : (\vec{b}_1)_j = 0 \neq (\vec{b}_2)_j\}|  \\ &  - |\{j > i :  (\vec{b}_1)_j = 1 \neq (\vec{b}_2)_j \}| + |\{j>i:  (\vec{b}_1)_j = 0 \neq (\vec{b}_2)_j\}| \end{array} \]
Then
\begin{align*}
\big(x(S_i^-)-y(S_i^-)\big)  - \big( x(S_0)-y(S_0)\big) &= \flex(i)\\ 
&= \big(x(S_i^+)-y(S_i^+)\big) - \big( x(S_0)-y(S_0)\big).
\end{align*} 
\end{lemma}

\begin{proof}
Any strand that only passes over vertical edges in $S_i^-$ contributes the same amount to the strand exponent for $S_i^-$ as for $S_0$.  Thus the difference between the strand exponents is determined exactly by strands involving $i$ in $S_i^-$.  

Moreover, strands can't start or end at interior vertices.  This means that if an $(i,j)$ or $(j,i)$ strand is supported only on one of $\vec{b}_1$ and $\vec{b}_2$ then this strand makes the same contribution to the strand exponent in $S_i^-$ as in $S_0$ even if its path through the square uses horizontal edges.  By hypothesis $(\vec{b}_1)_i \neq (\vec{b}_2)_i$ because $\vec{e}_i$ labels the horizontal edges.  We conclude the strands that contribute differently to $S_i^-$ versus $S_0$ are those for which $j$ satisfies $(\vec{b}_1)_j \neq (\vec{b}_2)_j$. 

If the edges labeled $\vec{b}_1$ and $\vec{b}_2$ have no vertical strands, i.e. $(\vec{b}_1)_i = (\vec{b}_1)_j$ respectively $\vec{b}_2$, then $S_i^-$ has a closed $(i,j)$ strand around its square that is clockwise if $j<i$ and counterclockwise else.  This gives two summands of $\flex(i)$.

If instead all of the edges labeled $\vec{b}_1$ and $\vec{b}_2$ have vertical strands, then the strands must alternate direction, again because $(\vec{b}_1)_i=1 \neq (\vec{b}_2)_i$.  Thus Lemma~\ref{lemma: square switch is like straight lines and loop} applies and describes the other two terms.  This completes the proof.

A similar argument holds for $S_i^+$ but since left and right are exchanged, signs change.  We obtain: 
\[\begin{array}{rl} \big(x(S_i^+)-y(S_i^+)\big) &- \big( x(S_0)-y(S_0)\big) = \\
& \hspace{-.2in} - |\{j < i :  (\vec{b}_2)_j = 1 \neq (\vec{b}_1)_j \}| + |\{j < i : (\vec{b}_2)_j = 0 \neq (\vec{b}_1)_j\}| \\ & \hspace{-.2in}  + |\{j > i :  (\vec{b}_2)_j = 1 \neq (\vec{b}_1)_j \}| - |\{j>i:  (\vec{b}_2)_j = 0 \neq (\vec{b}_1)_j\}| \end{array}\]
Since $\vec{b}_1$ is binary, each entry is $0$ or $1$ so this quantity is $\flex(i)$.  This completes the proof.
\end{proof}

Observe that $\flex$ depends only on the entries $j$ for which $(\vec{b}_1)_j \neq (\vec{b}_2)_j$. This gives rise to the following notation for the symmetric difference.

\begin{definition} \label{definition: symmetric difference}
Fix two binary vectors $\vec{b}_1$ and $\vec{b}_2$.  The \emph{symmetric difference} $S_{\vec{b}_1 \Delta \vec{b}_2}$ is the set
\[S_{\vec{b}_1 \Delta \vec{b}_2} = \{j: (\vec{b}_1)_j \neq (\vec{b}_2)_j \}\]
Suppose $S_{\vec{b}_1 \Delta \vec{b}_2} = \{i_1, \ldots, i_\nu\}$.  The \emph{symmetric difference vector} $\vec{b}_1 \Delta \vec{b}_2 \in \{0, 1\}^{\nu}$ is defined for each $1 \leq j \leq \nu$ by
\[(\vec{b}_1 \Delta \vec{b}_2)_j = (\vec{b}_1)_{i_j}\]
and we denote the function $flex_{\vec{b}_1 \Delta \vec{b}_2, \vec{1} - \vec{b}_1 \Delta \vec{b}_2}$ by $flex_{\vec{b}_1 \Delta \vec{b}_2}$.   We also use the multisets
\[S_{\vec{b}_1 - \vec{b}_2} = \{\flex(i): i \in S_{\vec{b}_1 \Delta \vec{b}_2} \textup{ and } (\vec{b}_1)_i = 1 \} \textup{ and }\] \[ S_{\vec{b}_2 - \vec{b}_1} = \{\flex(i): i \in S_{\vec{b}_1 \Delta \vec{b}_2} \textup{ and } (\vec{b}_1)_i = 0\} \]
\end{definition}

\begin{example}
If $\vec{b}_1 = 1101101001$ and $\vec{b}_2 = 0100011010$ then $S_{\vec{b}_1 \Delta \vec{b}_2} = \{1,4,5,6, 9, 10\}$ and $\vec{b}_1 \Delta \vec{b}_2 = 111001$.  We have the sequences $$(flex_{\vec{b}_1 \Delta \vec{b}_2}(i)) = (-1, 1, 3, 3, 1, 1) \textup{ and } (\flex(i)) = (-1, 0, 0, 1, 3, 3, 2, 2, 1, 1).$$
\end{example}

Note that if $S_{\vec{b}_1 \Delta \vec{b}_2} = \{i_1 < i_2 < \ldots < i_\nu\}$ then $flex_{\vec{b}_1 \Delta \vec{b}_2}(j) = \flex(i_j)$.  

We obtain the following corollaries.  The first formalizes the observation that $\flex$ depends on the symmetric difference.

\begin{corollary} \label{corollary: flex only depends on symm diff}
Let $\vec{b}_1, \vec{b}_2, \vec{b}_1', \vec{b}_2'$ be four vectors with $S_{\vec{b}_1 \Delta \vec{b}_2} = \{i_1 < \ldots < i_\nu\}$ and $S_{\vec{b}'_1 \Delta \vec{b}'_2} = \{j_1 < \ldots < j_\nu\}$.  If the symmetric difference vectors $\vec{b}_1 \Delta \vec{b}_2 = \vec{b}'_1 \Delta \vec{b}'_2$
then $\flex(i_k) = flex_{\vec{b}'_1, \vec{b}'_2}(j_k)$ for all $1 \leq k \leq \nu$.
\end{corollary}

Analyzing more carefully, we have the following.

\begin{corollary} \label{corollary: base case for flex induction in square switch}
Suppose $|S_{\vec{b}_1 \Delta \vec{b}_2}|=\nu$ and $\vec{b}_1, \vec{b}_2$ have $\vec{b}_1 \Delta \vec{b}_2 = 1^k0^{\nu-k}$ for some $k$.   If $k \geq \nu-k$ then 
\[ S_{\vec{b}_1 - \vec{b}_2} = S_{\vec{b}_2 - \vec{b}_1} \cup \{2k-\nu+1-2j: 1 \leq j \leq 2k-\nu \} \]
is a disjoint union of sets; similarly if $\nu-k>k$ then the disjoint union is
\[ S_{\vec{b}_2 - \vec{b}_1} = S_{\vec{b}_1 - \vec{b}_2} \cup \{\nu-2k+1-2j: 1 \leq j \leq \nu-2k \} \]
\end{corollary}

\begin{proof}
By definition, if $i>1$ then $flex_{\vec{b}_1 \Delta \vec{b}_2}(i)$ satisfies the following recurrence (regardless of $\vec{b}_1$ and $\vec{b}_2$):
    \begin{equation} \label{eqn: flex recurrence}
        flex_{\vec{b}_1 \Delta \vec{b}_2}(i) =\begin{cases}     flex_{\vec{b}_1 \Delta \vec{b}_2}(i-1) & \textup{ if }   (flex_{\vec{b}_1 \Delta \vec{b}_2})_i \neq (flex_{\vec{b}_1 \Delta \vec{b}_2})_{i-1}\\
    flex_{\vec{b}_1 \Delta \vec{b}_2}(i-1)+2 & \textup{ else if } (flex_{\vec{b}_1 \Delta \vec{b}_2})_i 
    =1 \\
    flex_{\vec{b}_1 \Delta \vec{b}_2}(i-1)-2 & \textup{ else if } 
    (flex_{\vec{b}_1 \Delta \vec{b}_2})_i  
    =0 \\
    \end{cases}
    \end{equation}
    Moreover 
    \[flex_{\vec{b}_1 \Delta \vec{b}_2}(1) = - |\{j > 1 :  (\vec{b}_1)_j = 1 \neq (\vec{b}_2)_j \}| + |\{j>1:  (\vec{b}_1)_j = 0 \neq (\vec{b}_2)_j\}|\]
    So we have $k$ elements in the set
    \[ S_{\vec{b}_1 - \vec{b}_2} = \{ (\nu-k)-(k-1), (\nu-k)-(k-1)+2, (\nu-k)-(k-1)+4, \ldots, (\nu-k)+(k-1)\} \]
    and $\nu-k$ elements in
    \[ S_{\vec{b}_2 - \vec{b}_1} = \{ k-(\nu-k-1), k-(\nu-k-1)+2, k-(\nu-k-1)+4, \ldots, k+(\nu-k-1)\} \]
    This proves the claim.
\end{proof}

Now we prove that the decomposition in the previous corollary holds for all vectors $\vec{b}_1$ and $\vec{b}_2$ using the previous corollary as the base case and inducting by comparing to the vectors obtained from $\vec{b}_1, \vec{b}_2$ by exchanging entries $i$ and $i+1$. The proof is a routine combinatorial exercise: we construct a Dyck path from a binary string, relate the action of $s_i$ to replacing a peak in the Dyck path with a valley, and then use the definition of $\flex$. We include all details for the benefit of readers who have not seen this sort of proof.

\begin{lemma} \label{lemma: inductive step of square switch}
Fix any binary vectors $\vec{b}_1$ and $\vec{b}_2$ and denote $|S_{\vec{b}_1 \Delta \vec{b}_2}|=\nu$.  Let $k = |S_{\vec{b}_1 - \vec{b}_2}|$ and $\nu - k = |S_{\vec{b}_2 - \vec{b}_1}|$ be the cardinalities of these multisets. If $k \geq \nu-k$ 
then $S_{\vec{b}_1 - \vec{b}_2} = S_{\vec{b}_2 - \vec{b}_1} \cup \{2k-\nu+1-2j: 1 \leq j \leq 2k-\nu \}.$
Similarly if $k \leq \nu-k$ then 
$S_{\vec{b}_2 - \vec{b}_1} = S_{\vec{b}_1 - \vec{b}_2} \cup \{\nu-2k+1-2j: 1 \leq j \leq \nu-2k \}$.
\end{lemma}

\begin{proof}
Let $s_i$ denote the simple transposition that exchanges positions $i$ and $i+1$.  We prove by induction, assuming the result holds for some $\vec{b}_1$ and $\vec{b}_2$ and then proving it holds for $s_i\vec{b}_1$ and $s_i \vec{b}_2$. 
 Corollary~\ref{corollary: base case for flex induction in square switch} proved the base case so the rest of this proof consists of the inductive step.

Corollary~\ref{corollary: flex only depends on symm diff} showed that the function $flex_{s_i\vec{b}_1, s_i\vec{b}_2}(i)$ only depends on the vector $s_i\vec{b}_1 \Delta s_i\vec{b}_2$.  So the result follows trivially unless both $i$ and $i+1$ are in the symmetric difference set $S_{\vec{b}_1 \Delta \vec{b}_2}$.  Moreover, if $(\vec{b}_1)_i=(\vec{b}_1)_{i+1}$ the claim is still trivial since $s_i\vec{b}_1 = \vec{b}_1$ and $s_{i+1}\vec{b}_2=\vec{b}_2$.  

We assume that $(\vec{b}_1)_i = 1$ and $(\vec{b}_1)_{i+1}=0$ and $i, i+1 \in S_{\vec{b}_1 \Delta \vec{b}_2}$.  Our proof is symmetric in the role of $1$ and $0$ so there is no loss of generality. Denote $S_{\vec{b}_1 \Delta \vec{b}_2} = \{i_1, \ldots, i_\nu \}$. We define a function $\dyk: [-1/2, \nu +1/2] \rightarrow \mathbb{R}$ that is translates the usual Dyck path created from a binary string, applied in our case to the symmetric difference.   Define $\dyk(1) = flex_{\vec{b}_1 \Delta \vec{b}_2}(1)$.  If $(\vec{b}_1 \Delta \vec{b}_2)_i=b$ then $\dyk$ has slope $(-1)^{b+1}2$ on interval $[i-1/2, i+1/2]$. Together, these two conditions give a continuous function.

By Equation \ref{eqn: flex recurrence} above, $\dyk(j)=flex_{\vec{b}_1 \Delta \vec{b}_2}(j)$ when $j \in \{1, 2, \ldots, \nu\}$. By the inductive hypothesis on $\vec{b}_1, \vec{b}_2$, the image of $\dyk$ contains the closed interval $[-|2k-\nu|+1,|2k-\nu|-1] \cap \mathbb{Z}$.  The intermediate value theorem applies to the intersection of the horizontal line $y=t_0$ with the image $Im(\dyk)$.  In particular, if $\exists j_0\in\mathbb{Z}$ such that $\dyk(j_0)=t_0$ then the horizontal line $y=t_0$ intersects the image $Im(\dyk)$ in an odd number of points when $-|2k-\nu|+1 \leq t_0 \leq |2k-\nu|-1$ and an even number of points otherwise.

Now compare $\dyk$ with $Dyck_{s_i\vec{b}_1 \Delta s_i\vec{b}_2}$.   For $j \in [i-1/2, i+3/2]$ the function $\dyk(j)$ either has slope $-2$ for $j \in [i-1/2, i+1/2]$ then slope $+2$ for $j \in [i+1/2, i+3/2]$ or vice versa. Acting by $s_i$ on $\vec{b}_1, \vec{b}_2$ negates these slopes in this interval:
    \[\frac{d}{dj} \dyk(j) = -\frac{d}{dj} Dyck_{s_i\vec{b}_1 \Delta s_i\vec{b}_2}(j)\]
    In other words, the $s_i$ action changes whether the function travels the top or bottom of a diamond, but does not change the endpoints:
        \[\dyk(i-1/2) = Dyck_{s_i\vec{b}_1 \Delta s_i\vec{b}_2}(i-1/2) \textup{  and   }\] \[\dyk(i+3/2) = Dyck_{s_i\vec{b}_1 \Delta s_i\vec{b}_2}(i+3/2)\]
    By construction $\dyk(j) = Dyck_{s_i\vec{b}_1 \Delta s_i\vec{b}_2}(j)$ unless $j \in \{i, i+1\}$. Moreover the continuous function $Dyck_{s_i\vec{b}_1 \Delta s_i\vec{b}_2}$ agrees with $\dyk$ on the endpoints of the interval $[1/2, \nu+1/2]$. Therefore the image of $Dyck_{s_i\vec{b}_1 \Delta s_i\vec{b}_2}$ also contains $[-|2k-\nu|+1, |2k-\nu|-1]$.

    At the same time $flex_{\vec{b}_1 \Delta \vec{b}_2}(j)$ differs from $flex_{s_i\vec{b}_1 \Delta s_i\vec{b}_2}(j)$ on exactly two inputs.  
    Indeed, we have
    \[flex_{s_i\vec{b}_1 \Delta s_i\vec{b}_2}(j) = \begin{cases} flex_{\vec{b}_1 \Delta \vec{b}_2}(j) & \textup{ if } j \neq i, i+1 \\
    flex_{\vec{b}_1 \Delta \vec{b}_2}(j)-2 & \textup{ if } j \in \{i, i+1\} \textup{ and } (\vec{b}_1)_i = 1 \\
    flex_{\vec{b}_1 \Delta \vec{b}_2}(j)+2 & \textup{ if } j \in \{i, i+1\} \textup{ and } (\vec{b}_1)_i = 0 \\
    \end{cases} \]
    So the interval in which $Dyck_{s_i\vec{b}_1 \Delta s_i\vec{b}_2}$ intersects the horizontal line $y=t_0$ an odd number of times is the same as that for $\dyk$.  This proves the result, and the rest of the claim follows by induction.
\end{proof}

\begin{example}
Continuing the previous example, we give the graph of $\dyk$ as well as $Dyck_{s_5\vec{b}_1 \Delta s_5\vec{b}_2} = Dyck_{111010}$ and $Dyck_{s_4s_5\vec{b}_1 \Delta s_4s_5\vec{b}_2} = Dyck_{111100}$. At each step, the previous Dyck path is shown, dotted.
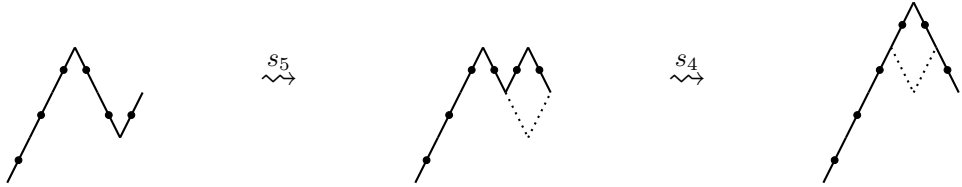
\begin{center}
    \begin{tikzpicture}[scale=0.3]

\draw[thick] (.5,-2) -- (1.5,0);
\draw[thick] (1.5,0) -- (2.5,2);
\draw[thick] (2.5,2) -- (3.5,4);
\draw[thick] (3.5,4) -- (4.5,2);
\draw[thick] (4.5,2) -- (5.5,0);
\draw[thick] (5.5,0) -- (6.5,2);

\draw[radius=.15, fill=black](1,-1)circle;
\draw[radius=.15, fill=black](2,1)circle;
\draw[radius=.15, fill=black](3,3)circle;
\draw[radius=.15, fill=black](4,3)circle;
\draw[radius=.15, fill=black](5,1)circle;
\draw[radius=.15, fill=black](6,1)circle;

    \end{tikzpicture} \hspace{0.25in} \raisebox{0.5in}{\Large $\stackrel{s_5}{\rightsquigarrow}$} \hspace{0.25in}
        \begin{tikzpicture}[scale=0.3]

\draw[thick] (.5,-2) -- (1.5,0);
\draw[thick] (1.5,0) -- (2.5,2);
\draw[thick] (2.5,2) -- (3.5,4);
\draw[thick] (3.5,4) -- (4.5,2);
\draw[thick] (4.5,2) -- (5.5,4);
\draw[thick] (5.5,4) -- (6.5,2);
\draw[dotted, thick] (4.5,2) -- (5.5,0);
\draw[dotted, thick] (5.5,0) -- (6.5,2);

\draw[radius=.15, fill=black](1,-1)circle;
\draw[radius=.15, fill=black](2,1)circle;
\draw[radius=.15, fill=black](3,3)circle;
\draw[radius=.15, fill=black](4,3)circle;
\draw[radius=.15, fill=black](5,3)circle;
\draw[radius=.15, fill=black](6,3)circle;

    \end{tikzpicture} \hspace{0.25in} \raisebox{0.25in}{\Large $\stackrel{s_4}{\rightsquigarrow}$} \hspace{0.5in}
           \begin{tikzpicture}[scale=0.3]

\draw[thick] (.5,-2) -- (1.5,0);
\draw[thick] (1.5,0) -- (2.5,2);
\draw[thick] (2.5,2) -- (3.5,4);
\draw[thick] (3.5,4) -- (4.5,6);
\draw[thick] (4.5,6) -- (5.5,4);
\draw[thick] (5.5,4) -- (6.5,2);
 \draw[dotted, thick] (3.5,4) -- (4.5,2);
\draw[dotted, thick] (4.5,2) -- (5.5,4);

\draw[radius=.15, fill=black](1,-1)circle;
\draw[radius=.15, fill=black](2,1)circle;
\draw[radius=.15, fill=black](3,3)circle;
\draw[radius=.15, fill=black](4,5)circle;
\draw[radius=.15, fill=black](5,5)circle;
\draw[radius=.15, fill=black](6,3)circle;

    \end{tikzpicture}
\end{center}
\end{example}

\begin{proof}[Proof of Lemma~\ref{lemma: overall square switch argument}]
First note, as indicated in Figure~\ref{figure: strand switch relation}, the binary vectors on the horizontal edges of $G_-$ (resp. $G_+$) must be the same because these web graph fragments conserve integer flow, and thus conserve binary labels.  Moreover, since the horizontal edges have weight $1$ there is at most one stranding for each choice of standard basis vector $\vec{e}_i$ to label the horizontal edge; denote the corresponding stranding $S_i$.  The stranding $S_i$ is valid for $G_-$ if and only if $(\vec{b}_1)_i = 1 \neq (\vec{b}_2)_i$ and is valid for $G_+$  if and only if $(\vec{b}_2)_i = 1 \neq (\vec{b}_1)_i$.  Using the notation of the symmetric difference, we obtain
\[w_\partial(G_-) = \sum_{\footnotesize \begin{array}{c} i \in S_{\vec{b}_1 \Delta \vec{b}_2}, (\vec{b}_1)_i = 1  \\ S_i \in \mathcal{S}tr(G_-),  \partial(S_i) = \vec{b}_1 \sqcup \vec{b}_2 \end{array}}  (-q)^{x(S_i)-y(S_i)}\]
and similarly for $G_+$.  

Denote the strand coefficient on the straight lines labeled $\vec{b}_1, \vec{b}_2$ in $G_0$ by $II(q)$ so $w_{\partial}(G_0) = w\left(G_{|k-l|, 1}\right) \cdot II(q)$ where $G_{|k-l|, 1}$ is the loop $\mathfrak{sl}_{|k-l|}$ web subgraph of $G_0$ and $w\left(G_{|k-l|, 1}\right)$ is its associated scalar invariant computed in Lemma \ref{lemma: loop graph calculation}. By Lemma~\ref{lemma: flex to relate strand exponent for square to straigh} we know
\[w_\partial(G_-) = II(q) \cdot  \sum_{\footnotesize \begin{array}{c} i \in S_{\vec{b}_1 \Delta \vec{b}_2} , (\vec{b}_1)_i = 1  \\ S_i \in \mathcal{S}tr(G_-),  \partial(S_i)  = \vec{b}_1 \sqcup \vec{b}_2 \end{array}} (-q)^{\flex(i)} \]
and similarly for $w_\partial(G_+)$.  The multiset of exponents in this sum for $w_\partial(G_-)$ are $S_{\vec{b}_1 - \vec{b}_2}$ respectively $w_\partial(G_+)$ and $S_{\vec{b}_2-\vec{b}_1}$.  Thus from Lemma~\ref{lemma: inductive step of square switch} we conclude
\[w_\partial(G_-)-w_\partial(G_+) = II(q) \cdot \Bigg(\sum_{j=1}^{|k-l|} (-q)^{|k-l|+1-2j} \Bigg)  = w_\partial(G_0) \]
where the last equality follows from Lemma~\ref{lemma: loop graph calculation}.
\end{proof}

\subsection{The generalized square switch relation}

\begin{proof}[Proof of Corollary~\ref{cor:generalized square switch}]
    If either $r=0$ or $s=0$, the right hand side of Equation \ref{eqn:Fontaine square switch general} has only one term and the two graphs differ by edge flips. If $r=s=1$, this is exactly the square switch relation on untagged webs from Theorem \ref{thm:Fontaine relations}. Now assume for some $m\geq 3$ the relation holds for all $r+s< m$ (and anytime $r=0$ or $s=0$). Consider a case where $r+s=m$, $r\neq 0$, and $s\neq 0$. Say $s\geq r$, and observe $s\geq 2$. To save space in the calculations below, we suppress all but the horizontal edge weights and assume vertical edges are directed upwards.

    The square removal relation on untagged webs, the inductive hypothesis, and some careful quantum integer computations yield:
    
    \begin{eqnarray*}
    \scalebox{.75}{\raisebox{-30pt}{\begin{tikzpicture}[scale=.5]
    \draw[radius=.08, fill=black](0,0)circle;
     \draw[radius=.08, fill=black](0,2)circle;
        \draw[radius=.08, fill=black](2,0)circle;
     \draw[radius=.08, fill=black](2,2)circle;
      \node at (1,-.5) {\tiny{$s$}};
      \node at (1,2.5) {\tiny{$r$}};

             \draw[thick] (0,-1)--(0,3);
 \draw[thick] (2,-1)--(2,3);
    \begin{scope}[thick,decoration={
    markings,
    mark=at position 0.5 with {\arrow{>}}}
    ] 
        \draw[postaction={decorate}, thick] (0,0)--(2,0);
        \draw[postaction={decorate}, thick] (2,2)--(0,2);
        
        \end{scope}
    \end{tikzpicture}}}
    &=& \frac{(-1)^{s-1}}{[s]_q}\;\scalebox{.75}{\raisebox{-35pt}{\begin{tikzpicture}[scale=.5]
    \draw[radius=.08, fill=black](0,-.5)circle;
    \draw[radius=.08, fill=black](0,1)circle;
    \draw[radius=.08, fill=black](2,1)circle;
     \draw[radius=.08, fill=black](0,2.5)circle;
        \draw[radius=.08, fill=black](2,-.5)circle;
     \draw[radius=.08, fill=black](2,2.5)circle;
      \node at (1,-1) {\tiny{$1$}};
      \node at (1,.5) {\tiny{$s-1$}};
      \node at (1,3) {\tiny{$r$}};

            \draw[thick](0,-1)--(0,3);  \draw[thick](2,-1)--(2,3);       
    \begin{scope}[thick,decoration={
    markings,
    mark=at position 0.5 with {\arrow{>}}}
    ] 
        \draw[postaction={decorate}, thick] (0,-.5)--(2,-.5);
        \draw[postaction={decorate}, thick] (0,1)--(2,1);
        \draw[postaction={decorate}, thick] (2,2.5)--(0,2.5);
        
        \end{scope}
    \end{tikzpicture}}}= \frac{(-1)^{s-1}}{[s]_q}
    \sum_{t=0}^{s-1} (-1)^{(k-l+r-s)t}\genfrac[]{0pt}{1}{k-l+r-s-1}{t}_q
    \scalebox{.75}{\raisebox{-37pt}{\begin{tikzpicture}[scale=.5]
    \draw[radius=.08, fill=black](0,-.5)circle;
    \draw[radius=.08, fill=black](0,1)circle;
    \draw[radius=.08, fill=black](2,1)circle;
     \draw[radius=.08, fill=black](0,2.5)circle;
        \draw[radius=.08, fill=black](2,-.5)circle;
     \draw[radius=.08, fill=black](2,2.5)circle;
      \node at (1,-1) {\tiny{$1$}};
      \node at (1,.5) {\tiny{$r-t$}};
      \node at (1,3.25) {\tiny{$s-1-t$}};

            \draw[thick](0,-1)--(0,3);  \draw[thick](2,-1)--(2,3);       
    \begin{scope}[thick,decoration={
    markings,
    mark=at position 0.5 with {\arrow{>}}}
    ] 
        \draw[postaction={decorate}, thick] (0,-.5)--(2,-.5);
        \draw[postaction={decorate}, thick] (2,1)--(0,1);
        \draw[postaction={decorate}, thick] (0,2.5)--(2,2.5);
        
        \end{scope}
    \end{tikzpicture}}}\\
    &=&
    \sum_{t=0}^{s-1} \frac{(-1)^{s-1+(k-l+r-s)t}}{[s]_q}\genfrac[]{0pt}{1}{k-l+r-s-1}{t}_q
    \left(\scalebox{.75}{\raisebox{-37pt}{\begin{tikzpicture}[scale=.5]
    \draw[radius=.08, fill=black](0,-.5)circle;
    \draw[radius=.08, fill=black](0,1)circle;
    \draw[radius=.08, fill=black](2,1)circle;
     \draw[radius=.08, fill=black](0,2.5)circle;
        \draw[radius=.08, fill=black](2,-.5)circle;
     \draw[radius=.08, fill=black](2,2.5)circle;
      \node at (1,-1) {\tiny{$r-t$}};
      \node at (1,.5) {\tiny{$1$}};
      \node at (1,3.25) {\tiny{$s-1-t$}};

            \draw[thick](0,-1)--(0,3);  \draw[thick](2,-1)--(2,3);       
    \begin{scope}[thick,decoration={
    markings,
    mark=at position 0.5 with {\arrow{>}}}
    ] 
        \draw[postaction={decorate}, thick] (2,-.5)--(0,-.5);
        \draw[postaction={decorate}, thick] (0,1)--(2,1);
        \draw[postaction={decorate}, thick] (0,2.5)--(2,2.5);
        
        \end{scope}
    \end{tikzpicture}}}\right.\\
    && \hspace{1in} \left. +(-1)^{k-l+r-t}[k-l+r-t-1]_q \scalebox{.75}{\raisebox{-37pt}{\begin{tikzpicture}[scale=.5]
    \draw[radius=.08, fill=black](0,-.5)circle;
     \draw[radius=.08, fill=black](0,2.5)circle;
        \draw[radius=.08, fill=black](2,-.5)circle;
     \draw[radius=.08, fill=black](2,2.5)circle;
      \node at (1,-1.25) {\tiny{$r-1-t$}};
      \node at (1,3.25) {\tiny{$s-1-t$}};

            \draw[thick](0,-1)--(0,3);  \draw[thick](2,-1)--(2,3);       
    \begin{scope}[thick,decoration={
    markings,
    mark=at position 0.5 with {\arrow{>}}}
    ] 
        \draw[postaction={decorate}, thick] (2,-.5)--(0,-.5);
        \draw[postaction={decorate}, thick] (0,2.5)--(2,2.5);
        
        \end{scope}
    \end{tikzpicture}}}\right)\\
    &=&
    \sum_{t=0}^{s-1} \frac{(-1)^{s-1+(k-l+r-s)t}}{[s]_q}\genfrac[]{0pt}{1}{k-l+r-s-1}{t}_q
    \left((-1)^{s-1-t}[s-t]_q\scalebox{.75}{\raisebox{-37pt}{\begin{tikzpicture}[scale=.5]
    \draw[radius=.08, fill=black](0,-.5)circle;

     \draw[radius=.08, fill=black](0,2.5)circle;
        \draw[radius=.08, fill=black](2,-.5)circle;
     \draw[radius=.08, fill=black](2,2.5)circle;
      \node at (1,-1) {\tiny{$r-t$}};
      \node at (1,3.25) {\tiny{$s-t$}};

            \draw[thick](0,-1)--(0,3);  \draw[thick](2,-1)--(2,3);       
    \begin{scope}[thick,decoration={
    markings,
    mark=at position 0.5 with {\arrow{>}}}
    ] 
        \draw[postaction={decorate}, thick] (2,-.5)--(0,-.5);
        \draw[postaction={decorate}, thick] (0,2.5)--(2,2.5);
        
        \end{scope}
    \end{tikzpicture}}}\right.\\
    && \hspace{1in} \left.+(-1)^{k-l+r-t}[k-l+r-t-1]_q \scalebox{.75}{\raisebox{-37pt}{\begin{tikzpicture}[scale=.5]
    \draw[radius=.08, fill=black](0,-.5)circle;
     \draw[radius=.08, fill=black](0,2.5)circle;
        \draw[radius=.08, fill=black](2,-.5)circle;
     \draw[radius=.08, fill=black](2,2.5)circle;
      \node at (1,-1.25) {\tiny{$r-1-t$}};
      \node at (1,3.25) {\tiny{$s-1-t$}};

            \draw[thick](0,-1)--(0,3);  \draw[thick](2,-1)--(2,3);       
    \begin{scope}[thick,decoration={
    markings,
    mark=at position 0.5 with {\arrow{>}}}
    ] 
        \draw[postaction={decorate}, thick] (2,-.5)--(0,-.5);
        \draw[postaction={decorate}, thick] (0,2.5)--(2,2.5);
        
        \end{scope}
    \end{tikzpicture}}}\right)\\
    &=& \scalebox{.75}{\raisebox{-37pt}{\begin{tikzpicture}[scale=.5]
    \draw[radius=.08, fill=black](0,-.5)circle;
     \draw[radius=.08, fill=black](0,2.5)circle;
        \draw[radius=.08, fill=black](2,-.5)circle;
     \draw[radius=.08, fill=black](2,2.5)circle;
      \node at (1,-1.25) {\tiny{$r$}};
      \node at (1,3.25) {\tiny{$s$}};

            \draw[thick](0,-1)--(0,3);  \draw[thick](2,-1)--(2,3);       
    \begin{scope}[thick,decoration={
    markings,
    mark=at position 0.5 with {\arrow{>}}}
    ] 
        \draw[postaction={decorate}, thick] (2,-.5)--(0,-.5);
        \draw[postaction={decorate}, thick] (0,2.5)--(2,2.5);
        
        \end{scope}
    \end{tikzpicture}}}+(-1)^{(k-l+r-s-1)s} \genfrac[]{0pt}{1}{k-l+r-s}{s}_q\scalebox{.75}{\raisebox{-30pt}{\begin{tikzpicture}[scale=.5]
    \draw[radius=.08, fill=black](0,1)circle;
     \draw[radius=.08, fill=black](2,1)circle;

      \node at (1,.25) {\tiny{$r-s$}};

            \draw[thick](0,-1)--(0,3);  \draw[thick](2,-1)--(2,3);       
    \begin{scope}[thick,decoration={
    markings,
    mark=at position 0.5 with {\arrow{>}}}
    ] 
        \draw[postaction={decorate}, thick] (2,1)--(0,1);
        
        \end{scope}
    \end{tikzpicture}}} \\
    && + \sum_{t=1}^{s-1}(-1)^{(k-l+r-s-1)t}\scalebox{.8}{$\left(\frac{[s-t]_q\genfrac[]{0pt}{1}{k-l+r-s-1}{t}_q + \genfrac[]{0pt}{1}{k-l+r-s-1}{t-1}_q[k-l+r-t]_q}{[s]_q}\right)$}\scalebox{.75}{\raisebox{-37pt}{\begin{tikzpicture}[scale=.5]
    \draw[radius=.08, fill=black](0,-.5)circle;

     \draw[radius=.08, fill=black](0,2.5)circle;
        \draw[radius=.08, fill=black](2,-.5)circle;
     \draw[radius=.08, fill=black](2,2.5)circle;
      \node at (1,-1) {\tiny{$r-t$}};
      \node at (1,3.25) {\tiny{$s-t$}};

            \draw[thick](0,-1)--(0,3);  \draw[thick](2,-1)--(2,3);       
    \begin{scope}[thick,decoration={
    markings,
    mark=at position 0.5 with {\arrow{>}}}
    ] 
        \draw[postaction={decorate}, thick] (2,-.5)--(0,-.5);
        \draw[postaction={decorate}, thick] (0,2.5)--(2,2.5);
        
        \end{scope}
    \end{tikzpicture}}}\\
    &=& \sum_{t=0}^s (-1)^{(k-l+r-s-1)t}\genfrac[]{0pt}{0}{k-l+r-s}{t}_q
        \scalebox{.75}{\raisebox{-30pt}{\begin{tikzpicture}[scale=.5]
    \draw[radius=.08, fill=black](0,0)circle;
     \draw[radius=.08, fill=black](0,2)circle;
        \draw[radius=.08, fill=black](2,0)circle;
     \draw[radius=.08, fill=black](2,2)circle;
    
      \node at (1,-.5) {\tiny{$r-t$}};
      \node at (1,2.5) {\tiny{$s-t$}};
            \draw[postaction={decorate}, thick] (0,-1)--(0,3);
         \draw[postaction={decorate}, thick] (2,-1)--(2,3);
    \begin{scope}[thick,decoration={
    markings,
    mark=at position 0.5 with {\arrow{>}}}
    ] 
 
        \draw[postaction={decorate}, thick] (2,0)--(0,0);
        \draw[postaction={decorate}, thick] (0,2)--(2,2);
        
        \end{scope}
    \end{tikzpicture}}}
    \end{eqnarray*}
\end{proof}

\def\cprime{$'$}

\end{document}